\documentclass[11pt]{amsart}

\makeatletter
\def\part{\@startsection{part}{0}%
  \z@{\linespacing\@plus\linespacing}{.5\linespacing}%
  {\normalfont\large\scshape\centering}} 
\makeatother

\usepackage[a4paper,
            bindingoffset=0in,
            left=1in,
            right=1in,
            top=1.2in,
            bottom=1.1in,
            footskip=.25in]{geometry}

\usepackage{mathpazo}
\usepackage{euler}

\usepackage{setspace}
\usepackage[dvipsnames]{xcolor}

\usepackage{tikz-cd}

\usepackage{mathtools}

\usepackage{enumerate}

\usepackage{amssymb}
\usepackage{hyperref}
\hypersetup{colorlinks,linkcolor={MidnightBlue},citecolor={Maroon},urlcolor={OliveGreen}}
\numberwithin{equation}{section}

\newtheorem{theorem}{Theorem}[section]

\newtheorem{lemma}[theorem]{Lemma}

\newtheorem{prop}[theorem]{Proposition}

\newtheorem*{theorem*}{Theorem}
\newtheorem*{corollary*}{Corollary}
\theoremstyle{definition}
\newtheorem{example}[theorem]{Example}
\newtheorem{remark}[theorem]{Remark}
\newtheorem{definition}[theorem]{Definition}
\newtheorem*{remark*}{Remark}

\newtheorem*{definition*}{Definition}
\newtheorem*{example*}{Example}

\newtheorem*{conj*}{Conjecture}

\theoremstyle{plain}
\newtheorem{maintheorem}{Theorem}

\newtheoremstyle{named}{}{}{\itshape}{}{\bfseries}{.}{.5em}{\thmnote{#3}}
\theoremstyle{named}

\newcommand{\BR}{\mathbb R} 
\newcommand{\BN}{\mathbb N} \newcommand{\BQ}{\mathbb Q}
 \newcommand{\BZ}{\mathbb Z}
\newcommand{\BF}{\mathbb F} 
 \newcommand{\BA}{\mathbb A}
\newcommand{\BP}{\mathbb P} \newcommand{\BG}{\mathbb G}
\newcommand{\CA}{\mathcal A} \newcommand{\CB}{\mathcal B}
\newcommand{\CC}{\mathcal C} 
\newcommand{\CE}{\mathcal E} \newcommand{\CF}{\mathcal F}
 
\newcommand{\CI}{\mathcal I} 
\newcommand{\CK}{\mathcal K} \newcommand{\CL}{\mathcal L}
 \newcommand{\CN}{\mathcal N}
\newcommand{\CO}{\mathcal O} \newcommand{\CP}{\mathcal P}
 
\newcommand{\CS}{\mathcal S} \newcommand{\CT}{\mathcal T}
 \newcommand{\CV}{\mathcal V}

\newcommand{\fb}{\mathfrak b}

\newcommand\smvee{\raise0.3ex\hbox{$\scriptscriptstyle\vee$}}

\newcommand{\Hilb}{Hilb}
\newcommand{\Nef}{Nef}

\newcommand{\CHom}{\mathcal{H}om}

\newenvironment{ack}{\textbf{\textit{Acknowledgements.}}}{}
\DeclareMathOperator{\Hom}{Hom}
\DeclareMathOperator{\Frob}{Frob}

\DeclareMathOperator{\trace}{Tr}

\DeclareMathOperator{\Mor}{Mor}

\DeclareMathOperator{\rank}{rank}

\DeclareMathOperator{\Supp}{Supp}
\DeclareMathOperator{\Pic}{Pic}
\DeclareMathOperator{\Jac}{Jac}

\DeclareMathOperator{\rel}{rel}

\DeclareMathOperator{\image}{Im}

\DeclareMathOperator{\Spec}{Spec}

\DeclareMathOperator{\bk}{\textbf{k}}
\DeclareMathOperator{\ch}{ch}

\DeclareMathOperator{\carac}{char}
\DeclareMathOperator{\Conf}{Conf}

\DeclareMathOperator{\alg}{alg}

\newcommand{\comment}[1]{}

\global\long\def\wangle#1{\left\langle #1\right\rangle }%

\allowdisplaybreaks

\theoremstyle{definition}
\newtheorem{construction}[theorem]{Construction}
\newcommand{\PP}{\BP}
\newcommand{\Aff}{\BA}
\newcommand{\ZZ}{\BZ}
\newcommand{\Ql}{\BQ_\ell}
\newcommand{\cO}{\CO}
\newcommand{\cP}{\CP}
\newcommand{\cI}{\CI}
\newcommand{\Morw}{\widetilde{\Mor}}
\DeclareMathOperator{\coker}{coker}
\DeclareMathOperator{\ev}{ev}
\DeclareMathOperator{\len}{len}
\newcommand{\bc}{b_c}
\newcommand{\Gm}{\BG_m}
\newcommand{\Ee}{E}
\newcommand{\Ge}{G}
\DeclareMathOperator{\rk}{rk}
\counterwithin{section}{part}

\title[Higher genus Betti bounds and Manin's conjecture]{Betti bounds for spaces of curves on varieties and\\ Manin's conjecture for quartic del Pezzo surfaces}

\author[Enhao Feng]{Enhao Feng\textsuperscript{1}}
\author[Matthew Hase-Liu]{Matthew Hase-Liu\textsuperscript{2}}

\thanks{\textsuperscript{1}Department of Mathematics, University of Kansas,
Lawrence, KS 66045, USA}
\thanks{\textsuperscript{2}Department of Mathematics, MIT,
Cambridge, MA 02139, USA}

\begin{document}
\setstretch{1.1}

\begin{abstract}
    We prove uniform exponential bounds for the compactly supported Betti numbers of spaces of morphisms from curves of fixed genus to projective varieties. For targets in a fixed projective space cut out by a prescribed number of equations of fixed degrees, the bound is exponential in the degree of the morphism and is independent of the ground field, the source curve, and the target. The proof constructs bounded-degree affine presentations involving only linearly many variables and equations, and then applies Katz's estimate.

    As an application, we establish a higher genus function field version of Manin's conjecture for split quartic del Pezzo surfaces, generalizing a recent theorem of Das--Lehmann--Tanimoto--Tosteson. The passage from $\PP^1$ to arbitrary source curves requires uniform bounds for compactly supported Betti numbers over all curves of fixed genus, together with new geometric results on higher genus morphism spaces, including irreducibility and dimension estimates for the relevant incidence strata. These results, combined with the necessary control of the configuration space contributions, allow the homological sieve and virtual height zeta function argument to yield, over sufficiently large finite fields and within a slightly shrunken nef cone, the predicted asymptotic with the expected leading constant.
\end{abstract}
\maketitle

\vspace{-2em}

\part*{Introduction}
Spaces of morphisms from curves to projective varieties arise naturally in both arithmetic geometry and the geometry of moduli spaces. Over a finite field, bounds for their compactly supported cohomology can be combined with the Grothendieck--Lefschetz trace formula to obtain point-counting estimates. For such applications, it is important to bound the total Betti number exponentially in the degree of the morphism. General estimates of Katz \cite{Katz} reduce this problem to finding affine presentations of the morphism space in which the number of variables and equations grows at most linearly with the degree, while the degrees of the equations remain bounded. The last condition is essential: equations whose degrees grow with the degree of the morphism would yield only a superexponential bound. See \cite{WanZhang26, PanZhangZhang26, Milnor64} for further examples of general Betti number estimates.

The main result of this paper constructs such presentations, and hence proves an exponential Betti bound, for morphisms from curves of arbitrary fixed genus. The resulting estimate is moreover uniform in the source curve. It generalizes the genus zero estimate used by Das--Lehmann--Tanimoto--Tosteson and Sawin--Shusterman \cite{DLTT25} in their proof of a function-field version of Manin's conjecture for split quartic del Pezzo surfaces. As a principal application, we extend their result from morphisms $\BP^1\to S$ to morphisms $B\to S$, where $B$ is a smooth projective curve of arbitrary fixed genus.

We first state the cohomological result. Fix integers $g\geqslant0$, $n\geqslant1$, $s\geqslant1$, and $d_1,\dots,d_s\geqslant1$. Let $k$ be a field, let $C/k$ be a smooth projective geometrically connected curve of genus $g$, and let $X\subseteq\PP^n_k$ be a reduced closed subscheme. We write $\Mor_e(C,X)$ for the quasi-projective $k$-scheme parametrizing morphisms $f\colon C\to X$ satisfying
\[
    \deg f^*\cO_{\PP^n}(1)=e.
\]
For a finite-type $k$-scheme $Y$ and a prime $\ell$ invertible in $k$, set
\[
    \bc(Y)\coloneqq\sum_i\dim_{\Ql}H_c^i(Y_{\bar k},\Ql).
\]

\begin{maintheorem}[Uniform Betti bounds]\label{thm:main}
    There is a constant $K=K(g;n,d_1,\dots,d_s)$ with the following property. For every separably closed field $k$, every prime $\ell\ne\carac k$, every
    smooth projective geometrically connected curve $C/k$ of genus $g$, every reduced closed subscheme $X\subseteq\PP^n_k$ cut out set-theoretically by $s$ homogeneous polynomials of degrees $d_1,\dots,d_s$, and every integer $e\geqslant\max(1,2g-1)$,
    \[
        \bc\bigl(\Mor_e(C,X)\bigr)\leqslant K^e.
    \]
\end{maintheorem}

For a fixed $X\subseteq\PP^n$, the constant depends only on $(g,X)$. In
particular, it does not depend on the isomorphism class of $C$ beyond its genus, on the ground field, or on the coefficients of the chosen equations of $X$. For a single fixed curve, the construction already gives a bound of the form $K(C,X)^e$. To remove the dependence on $C$, we work over moduli spaces of pointed genus $g$ curves with level structure and replace the Jacobian of a fixed curve by the universal Jacobian.

The proof expresses $\Mor_e(C,X)$ using a normalized Poincar\'e bundle over the Jacobian of $C$. After passing to a suitable affine-bundle modification, the morphism space is covered by finitely many strata, each admitting a presentation inside an affine space of dimension $O(e)$ by $O(e)$ equations of uniformly bounded degree. Katz's bound for compactly supported Betti numbers then gives the required exponential estimate. Uniformity in $C$ follows by carrying out the relevant constructions over the universal families.

We now describe the principal arithmetic application of Theorem \ref{thm:main}.

\section*{Application to Manin's conjecture}

One central problem in Diophantine geometry is to understand the distribution of rational points on projective varieties over global fields. For Fano varieties over number fields, a conjectural framework of Manin, developed in \cite{FMT89, BM90} and refined in \cite{Pey95}, predicts that, after removing an appropriate exceptional set, the number of rational points of height at most $T$ is asymptotic to
\[
    c_X T(\log T)^{\rho(X)-1},
\]
where $\rho(X)$ is the Picard rank and the leading constant $c_X$ is
described in terms of the geometry and arithmetic of $X$.

Over a global function field, a parallel conjectural framework also emerges. Let $B$ be a smooth projective geometrically integral curve over $\BF_q$ with function field $K$, and let $X$ be a Fano variety over $\BF_q$. Denote by $X_K$ the base change of $X$ over $K$. Then each $K$-point of $X_K$ extends uniquely to a morphism $f\colon B\to X$ via the valuative criterion. This induces a bijection between the sets
\[
    X_K(K) = \Mor(B, X)(\BF_q),
\]
where $\Mor(B, X)$ is the moduli space of $\BF_q$-morphisms from $B$ to $X$. Hence, one may use the anticanonical degree of morphisms to define the height function on $X_K(K)$. Manin's conjecture for $X_K$ over the function field then becomes a problem of counting the number of $\BF_q$-points on the moduli space $\Mor(B,X)$.

\begin{conj*}[Manin's conjecture over function fields; see \cite{LT26}]\label{conj: manin over function field}
    Suppose that $X_K(K)$ is not a thin subset. Let $r$ be the minimal positive anticanonical degree for a class $\alpha \in N_1(X)_{\BZ}$. Then there exists an exceptional set $Z\subset X_K(K)$ such that the quantity
    \[
        N_Z(B, X, d) \coloneqq \#\{ [f]\in \Mor(B, X)(\BF_q)\setminus Z \ |\ -K_X\cdot f_*B \leqslant rd \},
    \]
    satisfies
    \[
        N_Z(B, X, d) \sim_{d\to \infty} c(B, -K_X, X) q^{rd} (rd)^{\rho(X)-1}.
    \]
    Here, $c(B, -K_X, X)$ is Peyre's constant:
    \[
        c(B, -K_X, X) = (1-q^{-r})^{-1}\alpha(-K_X)\beta(X)\tau_{-K_X}(X),
    \]
    where $\alpha(-K_X)$ is the $\alpha$-constant for the nef cone of curves $\Nef_1(X)$, $\beta(X) = \# H^{1}(K,\Pic(X_{\overline{K}}))$, and $\tau_{-K_X}(X)$ is the Tamagawa constant.
\end{conj*}

The formulation of this conjecture relies on an important heuristic of Batyrev \cite{Bat88} which interprets the expected asymptotic formula through the geometry of $\Mor(B,X)$: its dimension and number of irreducible components should account for the powers of $q$ and $d$ in the main term. In a closely related setting, Ellenberg and Venkatesh \cite{EV05} proposed that such geometric heuristics can be made rigorous by combining homological stability of $\Mor(B,X)$ with the Grothendieck--Lefschetz trace formula.

The first breakthrough for this approach came in the work of Das--Lehmann--Tanimoto--Tosteson \cite{DLTT25}, who proved Manin's conjecture for split quartic del Pezzo surfaces after restricting to a slightly shrunken nef cone of curve classes, using a ``homological sieve" method. The proof combines the geometry of conic bundle structures on the surface, the inclusion-exclusion principle via a bar complex, Katz-type bounds for the cohomology of morphism spaces, and a virtual height zeta function whose leading term recovers Peyre's constant. Our second main theorem extends \cite[Theorem 1.2]{DLTT25} from the source curve $\BP^1$ to curves of arbitrary genus.

Before stating their result, we first specialize to the case of a split quartic del Pezzo surface $S$ over $\BF_q$. Here, split means that $\rho(S)=\rho(S_{\overline{\BF}_q})$. For a subset $\mathscr{C}$ in the nef cone $\Nef_1(S)$ and a positive integer $d$, define
\[
    N^{\mathscr C}(B,S,-K_S,d)\coloneqq \sum_{\substack{\alpha\in\mathscr C\cap N_1(S)_{\BZ}\\
    -K_S\cdot\alpha\leqslant d
    }}\#\Mor(B,S)_\alpha(\BF_q),
\]
where $\Mor(B,S)_\alpha$ parametrizes morphisms $f\colon B\to S$ satisfying
$f_*[B]=\alpha$. We call $\mathscr C\subseteq\Nef_1(S)$ a \textit{rational polyhedral conical region} if it is a finite union of rational polyhedral cones. Then equipping $N_1(S)_{\BR}$ with the normalized Lebesgue measure so that a fundamental domain for $N_1(S)_{\BZ}$ has measure one, we define the $\alpha$-constant as
\begin{equation}\label{eq:alpha-constant}
    \alpha(-K_S,\mathscr{C})\coloneqq \rho(S)\operatorname{Vol}\left\{\beta\in\mathscr{C}\ \middle|\ -K_S\cdot\beta\leqslant 1
    \right\}.
\end{equation}

If $\ell$ is a non-negative, rational, homogeneous, continuous, and piecewise linear function on $\Nef_1(S)$ and $\varepsilon>0$, we set
\[
    \Nef_1(S)_{\ell,\varepsilon}\coloneqq \left\{\alpha\in\Nef_1(S)\colon \ell(\alpha)\geqslant\varepsilon(-K_S\cdot\alpha)\right\}.
\]
The following theorem of Das--Lehmann--Tanimoto--Tosteson proves a version of
Manin's conjecture for $\Nef_1(S)_{\ell,\varepsilon}$ when $B=\BP^1$.

\begin{theorem*}[Das--Lehmann--Tanimoto--Tosteson, Theorem 1.2
\cite{DLTT25}]
There is an absolute constant $C_0$ with the following property. Let
$\varepsilon>0$ be a sufficiently small rational number, let $q$ be a prime
power satisfying $q^\varepsilon>C_0$, and let $S$ be a smooth split quartic
del Pezzo surface over $\BF_q$. Then there is a non-negative, rational,
homogeneous, continuous, and piecewise linear function $\ell$ on
$\Nef_1(S)$, positive on a dense open subcone, such that
\[
N^{\Nef_1(S)_{\ell,\varepsilon}}(\BP^1,S,-K_S,d)
\sim
(1-q^{-1})^{-1}
\alpha\bigl(-K_S,\Nef_1(S)_{\ell,\varepsilon}\bigr)
\beta(S)\tau_{-K_S}(S)q^d d^5
\]
as $d\to\infty$. Moreover, $\beta(S)=1$ and
\[
\tau_{-K_S}(S)
=
q^2(1-q^{-1})^{-6}
\prod_{c\in|\BP^1|}
\left(1-q^{-|c|}\right)^6
\frac{\#S(\BF_{q^{|c|}})}{q^{2|c|}}.
\]
\end{theorem*}

The proof of the theorem above uses an exponential bound for the compactly supported Betti numbers of spaces of morphisms from $\BP^1$ (\cite[Theorem B.1]{DLTT25}) to control the high codimensional error terms in the homological sieve. Theorem \ref{thm:main} provides the corresponding estimate when $\BP^1$ is replaced by a curve of arbitrary fixed genus. Combining Theorem \ref{thm:main} with higher genus versions of the geometric
and combinatorial arguments of \cite{DLTT25} gives our second main theorem.

\begin{maintheorem}\label{thm:main-manin}
    Fix $g\geqslant 0$. There is a constant $C_g$ with the following property:

    Fix a sufficiently small rational $\varepsilon>0$ and let $q$ be a prime power satisfying $q^\varepsilon>C_g$. Let $B$ be a smooth projective geometrically connected curve of genus $g$ over $\BF_q$, and let $S$ be a smooth split quartic del Pezzo surface over $\BF_q$. Then there is a non-negative, rational, homogeneous, continuous, and piecewise linear
    function $\ell$ on $\Nef_1(S)$, positive on a dense open subcone, such that
    \[
        N^{\Nef_1(S)_{\ell,\varepsilon}}(B,S,-K_S,d) \sim (1-q^{-1})^{-1} \alpha\bigl(-K_S,\Nef_1(S)_{\ell,\varepsilon}\bigr) \beta(S)\tau_{-K_S}(S)q^d d^5
    \]
    as $d\to\infty$. Moreover, $\beta(S)=1$ and
    \[
        \tau_{-K_S}(S) = q^{2-8g}(1-q^{-1})^{-6}
        \bigl(\#\Jac(B)(\BF_q)\bigr)^6
        \prod_{c\in|B|}
        \left(1-q^{-|c|}\right)^6
        \frac{\#S(\BF_{q^{|c|}})}{q^{2|c|}}.
    \]
\end{maintheorem}

Thus Theorem \ref{thm:main} provides the uniform cohomological estimate needed to
extend the homological sieve beyond the rational curve, while Theorem \ref{thm:main-manin} verifies the function field prediction for curves of arbitrary genus. In fact, our proof follows Peyre's all the heights version of Manin's conjecture \cite{Pey17, Pey21}, where we establish the following convergence
\[
    \lim_{d\to\infty} \frac{\#\Mor(B,S)_{d\alpha}(\BF_q)}{q^{-d K_S\cdot\alpha}} = \tau_{-K_S}(S)
\]
for any $\alpha \in \Nef_1(S)_{\ell,\varepsilon}$ with a uniform error term (See Theorem \ref{thm:point-counting on M_alpha}). We then sum over the lattice points in $\Nef_1(S)_{\ell,\varepsilon}$ and obtain the above theorem.

\subsubsection*{Related work in the function field setting}
Several complementary approaches to the function field setting have since been established. Bourqui proved Manin's conjecture for toric varieties using the universal torsors and harmonic analysis in \cite{Bou03, Bou11}, and Peyre established the case for flag varieties in \cite{Pey12}. The function field version of the circle method is employed to obtain Manin's conjecture for low degree hypersurfaces in \cite{BS23}.

For del Pezzo surfaces, the ones of degree at least $6$ are toric and hence fall into the scope of Bourqui's work mentioned above.
In the case of a split quintic del Pezzo surface, \cite{Tan25} applied the homological sieve method and obtained a similar result as \cite{DLTT25}, and the recent work \cite{BFG26} proved both Manin's conjecture and its motivic analogue using the Cox ring. When the degree is at most $5$, certain upper bounds on the number of rational points are also established in \cite{Gla25, GH24}.

\subsection*{Conventions}
All schemes are Noetherian unless explicitly specified otherwise. For a vector bundle $F$, we write $\PP(F)$ for the space of lines, so that if $q_F\colon\PP(F)\to S$ is the projection then $\cO_{\PP(F)}(-1)\hookrightarrow q_F^*F$. Compactly supported cohomology is always $\ell$-adic with $\ell$ invertible on the base, and $\bc(-)=\sum_i\dim H^i_c(-,\Ql)$ is computed after base change to an algebraic closure. Push-forwards are underived unless decorated with $R$.

\smallskip
\noindent\begin{ack}
    We would like to thank Brian Lehmann, Will Sawin, Mark Shusterman, Sho Tanimoto, Phil Tosteson, and Dingxin Zhang for their interest and helpful comments. We are also very thankful to Sho Tanimoto for many conversations and for sharing his idea that resulted in Lemma \ref{lem:min slope of elementary modification} and Proposition \ref{prop:dimboundWak}. 

    The first author is grateful to his advisor Brian Lehmann, for his tireless explanation of many ideas during the preparation of the article and for his precious encouragement and support throughout the years. The second author would also like to thank his advisor Will Sawin and Dingxin Zhang for suggesting many potential approaches to Theorem \ref{thm:main}; while the proof of Theorem \ref{thm:main} ultimately uses a different strategy, their ideas provided very interesting perspectives and insights. 
\end{ack}

\part{Betti bounds for spaces of curves on varieties}\label{part 1}

This part is devoted to the proof of Theorem \ref{thm:main}. The overall approach mimics that of \cite[Theorem B.1]{DLTT25} by using Katz's Betti bounds to obtain the desired exponential bound. The essential difficulty lies in expressing the equations in a form that fits Katz's framework, i.e. one must find affine presentations of the morphism spaces such that the degrees of the defining equations remain bounded independently of the degrees of the morphisms. The crucial insight is Proposition \ref{prop:regularity}, where we establish regularity of the bundles used in the construction of the moduli spaces.

\section{Outline of the proof}

Fix a smooth projective curve $C$ of genus $g$. Our first goal is to prove the desired Betti-number bound for $\bc(\Mor_e(C,X))$. We begin by describing morphisms from $C$ to $\PP^n$ using line bundles and their global sections. After choosing a marked point on $C$, the normalized Poincar\'e bundle gives a projective bundle model
\[
    \Mor_e(C,\PP^n)\subset\PP(\Ee_e^{\oplus(n+1)})
\]
over the Jacobian $J = \Pic^0(C)$. The equations defining $X\subseteq\PP^n$ then cut out a locus $Z_e(C;f_\bullet)$ with the same reduced structure as $\Mor_e(C,X)$.

The main difficulty is to place this locus in a form to which Katz's estimate can be applied. The condition that a tuple of sections has no common zero is open and therefore does not immediately give a closed affine presentation. To address this, we introduce auxiliary sections which produces an affine-bundle modification
\[
    \Morw_e\to Z_e(C;f_\bullet)
\]
of the morphism space with the same total compactly supported Betti number. We then stratify the resulting space so that the $\BG_m$-action can be normalized on each stratum.

After passing to suitable affine charts of the Jacobian $J$, each stratum can be written as a closed subscheme of an affine space. The number of variables and equations grows at most linearly with $e$, while the degrees of the equations remain bounded independently of $e$. Katz's theorem therefore gives an exponential bound for the compactly supported Betti numbers of $\Mor_e(C,X)$.

It remains to make this bound independent of the particular curve $C$. For this purpose, we carry out the construction simultaneously over families of pointed curves of genus $g$. The main technical input is a uniform Castelnuovo--Mumford regularity statement for the vector bundles appearing in the projective bundle model. There are two important ingredients that go into this. The first one is the relative Serre vanishing together with the peeling lemma of Section \ref{sec:toolbox}, and the second one is the uniform control of the ideal sheaf regularity developed in Section \ref{sec:regularity}. This allows us to choose projective embeddings, affine charts, generators, and equations uniformly over a dense open subset of the parameter space. A Noetherian-induction argument in Lemma \ref{lem:noeth} then extends the resulting bounds from these dense open subsets to the entire parameter space. Finally, every genus $g$ curve with a marked point over an algebraically closed field of any characteristic occurs as a geometric fibre of a universal curve over one of the two moduli schemes $S_3, S_4$ of finite type (see Lemma \ref{lem:unif}), so taking the maximum of the corresponding constants gives a bound depending only on $g$, $n$, and the degrees of the equations defining $X$.

We now indicate how the proof is organized. Section \ref{sec:toolbox} gives the basic tools for uniformity arguments and the Betti number comparisons used throughout the first part of the paper. In Section \ref{sec:families}, we formulate the family version of the Betti-number bound as Theorem \ref{T:family} and use it to deduce Theorem \ref{thm:main}. Then we construct the Poincar\'e bundle model and its affine bundle modification in Section \ref{sec:model}. Section \ref{sec:regularity} establishes the uniform regularity and bounded-degree results needed to control the defining equations. Finally, Section \ref{sec:presentation} assembles these ingredients into affine presentations and applies Katz's bound to prove Theorem \ref{T:family}.

\begin{remark}
    We observe that the resulting constant of Theorem \ref{thm:main} still depends on the genus $g$, since both the Jacobian and the number of affine charts vary with $g$. The only ineffective step is the Noetherian-induction argument; all other parts of the construction are effective in principle. The exponential shape $K^e$ of the bound is forced by Katz's bound with $O(e)$ variables and equations of uniformly bounded degree, and this method already gives that shape for $X=\PP^n$.
\end{remark} 

\section{Preliminaries}\label{sec:toolbox}

In this section, we record the basic technical tools that we employ throughout Part \ref{part 1}. The first three lemmas concern the uniformity argument.

\begin{lemma}[Peeling lemma]\label{lem:peel}
    Let $S$ be a Noetherian scheme, $f\colon Y\to S$ proper, and $F$ coherent on $Y$ and flat over $S$. If $R^if_*F=0$ for all $i\geqslant i_0$, then $H^i(Y_t,F_t)=0$ for all $i\geqslant i_0$ and all geometric points $t$ of $S$.
\end{lemma}

\begin{proof}
    The assertion is Zariski-local on $S$, since higher direct images commute with restriction to open subschemes, so we may write $S=\Spec R$. By cohomology and base change \cite[Chapter II, Section 5]{Mumford}, there is a bounded complex $K^\bullet$ of finite projective $R$-modules with functorial isomorphisms $H^i(Y\times_R R',F_{R'})\cong H^i(K^\bullet\otimes_R R')$ for every $R$-algebra $R'$. The hypothesis says $H^i(K^\bullet)=0$ for $i\geqslant i_0$. We claim $K^\bullet$ is quasi-isomorphic (in fact homotopy equivalent) to a complex of finite projectives in degrees $<i_0$; the conclusion then follows by taking $R'=\kappa(t)$.
    
    Let $b$ be the top degree of $K^\bullet$ and suppose $b\geqslant i_0$. Then $H^b(K^\bullet)=\coker(d^{b-1})=0$, so $d^{b-1}$ is a surjection onto the projective module $K^b$ and splits: $K^{b-1}\cong\ker(d^{b-1})\oplus K^b$, with $\ker(d^{b-1})$ finite projective. Since $d^{b-2}$ lands in $\ker(d^{b-1})$, the complex splits off a null-homotopic summand $[K^b\xrightarrow{\operatorname{id}}K^b]$, leaving an equivalent complex with top degree $b-1$. Iterating this gives the desired result.
\end{proof}

\begin{lemma}[Relative Serre vanishing]\label{lem:serre}
    Let $f\colon Y\to Z$ be a projective morphism of Noetherian schemes, $\cO_Y(1)$ an $f$-very ample line bundle, and $F$ coherent on $Y$. There is $k_0$ such that for all $k\geqslant k_0$: $R^if_*(F(k))=0$ for $i>0$, and the counit $f^*f_*(F(k))\to F(k)$ is surjective.
\end{lemma}

\begin{proof}
    The vanishing is \cite[III.8.8]{Hartshorne}. Relative global generation follows from the usual Serre argument, cf.\ \cite[II.5.17]{Hartshorne}, after restricting to affine opens of $Z$.
\end{proof}

\begin{remark}\label{rem:powers}
    Let $\pi\colon\mathcal{C}\to T$ be a smooth proper genus $g$ curve with a section $\infty$, and let $F$ be coherent on $\mathcal C$. The line bundle $\cO_{\mathcal{C}}(\infty)$ has fibre degree $1$ and is not relatively very ample for $g\geqslant 1$. To apply Lemma \ref{lem:serre} with twists $\cO(e\infty)$, write $e=(2g+1)a+b$ with $0\leqslant b\leqslant 2g$ and apply the lemma with $\cO(1)\coloneqq \cO((2g+1)\infty)$ (this is fibrewise very ample and relatively very ample upon shrinking $T$) to the finitely many sheaves $F(b\infty)$. If $a_0(b)$ is the resulting Serre threshold for $F(b\infty)$, one may take $e^\dagger=(2g+1)(1+\max_{0\leqslant b\leqslant 2g}\{a_0(b)\})$. Then $R^{>0}$-vanishing holds for all
    $e\geqslant e^\dagger$.
\end{remark}

\begin{lemma}[Noetherian induction]\label{lem:noeth}
    Let $S$ be a Noetherian scheme and let $P(t;K)$ be a property of pairs (geometric point $t$ of $S$, real number $K$), monotone in $K$. Suppose: for every integral closed subscheme $T\subseteq S$ there exist a finite surjective morphism $T'\to T$, a dense open $V'\subseteq T'$, and a constant $K_{V'}$ such that $P(t;K_{V'})$ holds for every geometric point $t$ of $V'$ (regarded as a geometric point of $T$ via $T'\to T$). Then there is a single constant $K_S$ such that $P(t;K_S)$ holds for every geometric point $t$ of $S$.
\end{lemma}

\begin{proof}
    Consider the set of reduced closed subsets $Z\subseteq S$ for which no single constant works at all geometric points of $Z$. We argue by contradiction: If this set is nonempty, it has a minimal element $Z$ by the Noetherian property. Write $Z=Z_1\cup\cdots\cup Z_c$ with each $Z_i$ integral and the decomposition irredundant, so that the generic point $\eta_i$ of $Z_i$ lies in no $Z_j$ with $j\ne i$. For each $i$ the hypothesis gives $T_i'\to Z_i$ finite surjective, a dense open $V_i'\subseteq T_i'$, and a constant $K_i$. Set
    \[
        Z'\coloneqq \bigcup_{i=1}^{c}\operatorname{image}(T_i'\setminus V_i'),
    \]
    a closed subset of $Z$, since finite morphisms are closed.

    We claim $Z'$ is strictly smaller than $Z$. Indeed, $Z'$ misses every $\eta_i$: for $j\ne i$ we have $\operatorname{image}(T_j'\setminus V_j')\subseteq Z_j$, which avoids $\eta_i$ by irredundancy, while any point of $T_i'$ lying over $\eta_i$ is a maximal point of $T_i'$ (its closure maps finitely and surjectively onto $Z_i$) and hence lies in the dense open $V_i'$. So $\eta_i\notin\operatorname{image}(T_i'\setminus
    V_i')$.

    Now let $t$ be a geometric point of $Z$ not factoring through $Z'$. Then $t$ lies in some $Z_i$ with $t\notin\operatorname{image}(T_i'\setminus V_i')$, so every lift of $t$ to $T_i'$ lands in $V_i'$, and at least one lift exists by surjectivity. Hence $P(t;K_i)$ holds. 
    
    Since $Z'$ is closed and strictly smaller than $Z$, minimality provides a constant $K'$ valid at all geometric points of $Z'$. Then $\max(K_1,\dots,K_c,K')$ works on all of $Z$, so we obtain a contradiction.
\end{proof}

The next two lemmas allow us to compare compactly supported Betti numbers.

\begin{lemma}[Excision inequality]\label{lem:excision}
    Let $Y$ be a scheme of finite type over an algebraically closed field, $Z\subset Y$ closed with open complement $V$. Then $\bc(Y)\leqslant\bc(Z)+\bc(V)$. More generally, if $Y=\bigsqcup_{a=1}^N Y_a$ is a finite decomposition into locally closed subschemes which is ``filtrable'', i.e.\ the pieces can be ordered so that each partial union $Y_1\cup\cdots\cup Y_i$ is open in $Y$, then $\bc(Y)\leqslant\sum_a\bc(Y_a)$.
\end{lemma}

\begin{proof}
    The long exact sequence $\cdots\to H^i_c(V)\to H^i_c(Y)\to H^i_c(Z)\to H^{i+1}_c(V)\to\cdots$ gives $\dim H^i_c(Y)\leqslant\dim H^i_c(V)+\dim H^i_c(Z)$; now sum over $i$. For the filtrable case we induct on $N$: the union $U\coloneqq Y_1\cup\cdots\cup Y_{N-1}$ is open in $Y$ with closed complement $Y_N$, so $\bc(Y)\leqslant\bc(U)+\bc(Y_N)$, and the decomposition of $U$ by $Y_1,\dots,Y_{N-1}$ is again filtrable.
\end{proof}

\begin{lemma}[Affine-bundle invariance]\label{lem:affbundle}
    Let $f\colon Y\to Z$ be a torsor under a vector bundle of rank $r$ on a scheme $Z$ of finite type over an algebraically closed field. Then $f$ is Zariski-locally trivial and $H^i_c(Y,\Ql)\cong H^{i-2r}_c(Z,\Ql)(-r)$ for all $i$; in particular $\bc(Y)=\bc(Z)$.
\end{lemma}

\begin{proof}
    Torsors under the additive group of a vector bundle are classified by coherent $H^1$ computed in the Zariski topology, so they are Zariski-locally trivial; the underlying group scheme is special in the sense of Serre. In particular, on a trivializing affine cover, homotopy invariance of compactly supported cohomology gives $H^i_c(Z_a\times\Aff^r)\cong H^{i-2r}_c(Z_a)(-r)$. Globally, $Rf_!\Ql\cong\Ql(-r)[-2r]$ and the Leray spectral sequence degenerates. Summing dimensions then gives
    $\bc(Y)=\bc(Z)$.
\end{proof}

\section{Reduction to finitely many families}\label{sec:families}

To obtain a bound that is independent of the isomorphism class of $C$, we place all curves of genus $g$ into finitely many algebraic families and prove the required bound uniformly over each family. For this purpose, we work with pointed curves equipped with a level structure. The marked point is used to normalize the Poincar\'e bundle, while the level structure provides a fine moduli scheme carrying a universal curve. These choices are auxiliary and do not affect the intrinsic morphism scheme $\Mor_e(C,X)$. Since every smooth projective curve over a separably closed field has a rational point, passing to pointed curves imposes no restriction. The following lemma shows that the universal families with level $3$ and level $4$ suffice in every characteristic.

\begin{lemma}[Uniformization by level covers]\label{lem:unif}
    Fix $g\geqslant 1$. For each $\nu\in\{3,4\}$ there is a scheme $S_\nu$ of finite type over $\ZZ[1/\nu]$, together with a smooth proper curve
    \[
        \pi_\nu\colon \mathcal{C}_\nu\longrightarrow S_\nu
    \]
    of genus $g$ with geometrically connected fibres and a section $\infty_\nu\colon S_\nu\to\mathcal{C}_\nu$, having the following property: If $k$ is an algebraically closed field, $(C,\infty)$ is a pointed smooth projective genus $g$ curve over $k$, and $\nu\in\{3,4\}$ is invertible in $k$, then there is a point $t\in S_\nu(k)$ such that
    \[
        (\mathcal{C}_{\nu,t},\infty_{\nu,t})
            \cong (C,\infty).
    \]
    In particular, every such pointed curve occurs in at least one of the two families.
\end{lemma}

\begin{proof}
    This argument is standard but we spell out some of the details for the sake of exposition. Fix $\nu\in\{3,4\}$. We first construct the family when $g\geqslant 2$. For a smooth proper genus $g$ curve $p\colon Y\to T$ over a $\ZZ[1/\nu]$-scheme, the sheaf
    \[
        \mathbb{H}_\nu(Y/T)\coloneqq R^1p_*(\ZZ/\nu\ZZ)
    \]
    is locally free of rank $2g$ over $\ZZ/\nu\ZZ$. Cup product and the trace map equip it with a perfect alternating pairing taking values in the rank one sheaf $R^2p_*(\ZZ/\nu\ZZ)$, and hence with the homogeneous symplectic structure of \cite[(5.3)]{DeligneMumford}. Recall that a Jacobi level $\nu$ structure is an isomorphism of $\mathbb{H}_\nu(Y/T)$ with the standard rank $2g$ module that respects these homogeneous symplectic structures
    \cite[(5.4)]{DeligneMumford}.

    For any fixed family $Y/T$, its Jacobi level $\nu$ structures form a finite \'etale $T$-scheme: \'etale-locally on $T$, both modules in the definition are constant, and the possible symplectic isomorphisms form a finite set. So the level stack (genus $g$ curves equipped with a Jacobi level $\nu$ structure) is finite \'etale over the moduli stack $\mathcal{M}_g$. Since $\nu\geqslant 3$, Serre's rigidity lemma implies that an automorphism preserving the level structure is the identity. Deligne--Mumford show that this level stack is represented by a scheme $M_g[\nu]$ of finite type over $\ZZ[1/\nu]$ \cite[(5.14)]{DeligneMumford}. Let
    \[
        p_\nu\colon U_\nu\longrightarrow M_g[\nu]
    \]
    be the universal curve. A marked point is a point of the universal curve, so we obtain the universal pointed curve by setting
    \[
        S_\nu\coloneqq U_\nu, \mathcal{C}_\nu\coloneqq U_\nu\times_{M_g[\nu]}U_\nu,\text{ and } \pi_\nu\coloneqq \operatorname{pr}_2.
    \]
    The diagonal $\Delta\colon U_\nu\to U_\nu\times_{M_g[\nu]}U_\nu$ is the section $\infty_\nu$. Then $\pi_\nu$ is a smooth proper genus $g$ curve with geometrically connected fibres, and both $S_\nu$ and $\mathcal{C}_\nu$ are of finite type over $\ZZ[1/\nu]$.

    When $g=1$, let $Y(\nu)$ denote the fine moduli scheme of elliptic curves with full level $\nu$ structure, where a full level structure on $E/T$ is an isomorphism
    \[
        (\ZZ/\nu\ZZ)^2_T\xrightarrow{\ \sim\ }E[\nu].
    \]
    For $\nu\geqslant 3$, this functor is represented by a smooth affine scheme $Y(\nu)$ of finite type over $\ZZ[1/\nu]$; see \cite[Chapters 2--4]{KatzMazur}. We set $S_\nu\coloneqq Y(\nu)$, let $\mathcal{C}_\nu\to S_\nu$ be the universal elliptic curve, and denote its zero section by $\infty_\nu$.
    
    It remains to check the asserted covering property. Let $(C,\infty)$ be defined over an algebraically closed field $k$, and let $\nu\in\{3,4\}$ be invertible in $k$. Suppose first that $g\geqslant 2$. Then $H^1_{\mathrm{\acute et}}(C,\ZZ/\nu\ZZ)$ is free of rank $2g$, and Poincar\'e duality gives it a perfect alternating pairing with values in the free rank one module $H^2_{\mathrm{\acute et}}(C,\ZZ/\nu\ZZ)$. Choose a generator of the latter module. A standard inductive symplectic basis argument then produces a Jacobi level $\nu$ structure. The level structure gives a point $m\in M_g[\nu](k)$ corresponding to $C$, and the marked point $\infty\in C(k)=(U_\nu)_m(k)$ gives a point $t\in U_\nu(k)=S_\nu(k)$. By the fibre product construction,
    \[
        (\mathcal{C}_{\nu,t},\infty_{\nu,t})\cong (C,\infty).
    \]
    If $g=1$, use $\infty$ as the origin, thereby making $C$ an elliptic curve. Since $\nu$ is invertible in $k$, its $\nu$-torsion is isomorphic to $(\ZZ/\nu\ZZ)^2$ over the algebraically closed field $k$ \cite[Theorem 2.3.1]{KatzMazur}. Choosing such an isomorphism gives a full level $\nu$ structure and hence a classifying point $t\in Y(\nu)(k)$ with the required property.

    Finally, note that at least one of $3$ and $4$ is invertible in every field: take $\nu=3$ unless $\operatorname{char}k=3$, in which case take $\nu=4$. This proves the last assertion.
\end{proof}

\begin{remark}\label{rem:g0}
    For $g=0$, the unique genus $0$ curve over algebraically closed $k$ is $\PP^1$, with $J\cong\Spec k$. Everything below goes through with essentially no modification, but it is easier to use the bound proved by Sawin--Shusterman in the appendix to \cite{DLTT25}. Their proof, applied to the given equations, gives $C^{e+1}$ with $C$ depending only on $n,s,d_1,\dots,d_s$; since $e\geqslant 1$, this is at most $(C^2)^e$.
\end{remark}

From now on, we assume $g\geqslant 1$. The following theorem proves the analogous Betti number bounds for a family of genus $g$ curves. 

\begin{theorem}[The family version]\label{T:family}
    Fix integers $g\geqslant 1$, $n,s\geqslant 1$, and $d_1,\dots,d_s\geqslant1$. Let $T$ be an integral scheme of finite type over $\ZZ$ and $(\pi\colon\mathcal{C}\to T,\infty)$ a smooth proper genus $g$ curve with geometrically connected fibres and a section. Then there exist a finite surjective $T'\to T$, a dense open $V\subseteq T'$, and a constant $K_V$ such that for every geometric point $t$ of $V$, with algebraically closed residue field $\kappa(t)$, every reduced closed subscheme $X\subset \PP^n_{\kappa(t)}$ cut out set-theoretically by $s$ equations of degrees $d_1,\dots,d_s$, every $\ell$ invertible in $\kappa(t)$, and every $e\geqslant 2g-1$,
    \[
        \bc\bigl(\Mor_e(\mathcal{C}_t,X)\bigr)\ \leqslant\ K_V^{\,e}.
    \]
\end{theorem}

\begin{proof}[Proof of Theorem \ref{thm:main} assuming Theorem \ref{T:family}]
    It suffices to assume that $g\geqslant 1$ and $k$ is algebraically closed. Apply Lemma \ref{lem:noeth} to $S=S_\nu$ for $\nu=3,4$, taking $P(t;K)$ to be the property that the bound holds with constant $K$ at $t$ for all $X,\ell,e\geqslant 2g-1$ as above. Theorem \ref{T:family} verifies the hypothesis: apply it to each integral closed $T\subseteq S_\nu$ with the restricted universal family, noting that such a $T$ is of finite type over $\ZZ$, being closed in the finite-type $S_\nu$, and that base change of the family along $T'\to T$ is again such a family. This yields constants $K_{S_3},K_{S_4}$ valid at all geometric points. Now let $(C,X)/k$ be given. Choose $\infty\in C(k)$, choose $\nu\in\{3,4\}$ invertible in $k$, and choose a classifying point $t\in S_\nu(k)$ with $\mathcal{C}_t\cong C$ (Lemma \ref{lem:unif}). Then $\bc(\Mor_e(C,X))=\bc(\Mor_e(\mathcal{C}_t,X))\leqslant K_{S_\nu}^{\,e}$, so we may take $K\coloneqq \max(K_{S_3},K_{S_4})$.
\end{proof}

The rest of the first part of this paper is dedicated to proving Theorem \ref{T:family}, and we fix $(\mathcal{C}\to T,\infty)$ as in its statement. During the proof, we may replace $T$ by a finite surjective cover or a dense open subscheme finitely many times, and we continue to denote the resulting base by $T$.

\section{A projective bundle model for the morphism space}\label{sec:model}

In this section, we construct the projective bundle model of the morphism space used in the proof. Starting from the normalized Poincar\'e bundle on the relative Jacobian, we obtain on each geometric fibre a projective bundle model for $\Mor_e(C,\PP^n)$ and use multiplication maps to impose the equations defining $X$. We then pass to an auxiliary affine bundle with the same total compactly supported Betti number and stratify it in preparation for the affine presentations of Section \ref{sec:presentation}.

\subsection{Building the projective bundle model}

The divisor $(2g+1)\infty$ has degree $2g+1$ on every geometric
fibre of $\mathcal C\to T$. Hence
$\cO_{\mathcal C}((2g+1)\infty)$ is fibrewise very ample and, after
shrinking $T$, relatively very ample. In particular, $\mathcal C\to T$
is projective.

Let
\[J\coloneqq\Pic^0_{\mathcal C/T}.\]
The relative Picard functor is represented by a separated $T$-scheme
locally of finite type, and its identity
component $J$ is an abelian scheme of relative dimension $g$.

The section $\infty$ rigidifies the Picard functor. Consequently, there
is a Poincar\'e line bundle $\cP$ on $J\times_T\mathcal C$, uniquely
normalized by requiring its restrictions to
$J\times_T\infty(T)$ and $0_J\times_T\mathcal C$ to be trivial.

For convenience, set $Q\coloneqq \infty^*\cO_{\mathcal C}(\infty)$ to be a line bundle on $T$. We write $\pi_J,\pi_{\mathcal{C}}$ for the projections of $J\times_T\mathcal{C}$ to $J$ and $\mathcal{C}$, and $\rho\colon J\to T$, $\pi\colon\mathcal{C}\to T$ for the structure maps. For an integer $e$ we let
\[
    \cP_e\coloneqq \cP\otimes\pi_{\mathcal{C}}^*\cO_{\mathcal{C}}(e\infty), \Ee_e\coloneqq \pi_{J*}\cP_e,\text{ and }  \Ge_e\coloneqq \pi_{J*}\bigl(\cP^\vee\otimes\pi_{\mathcal{C}}^*\cO_{\mathcal{C}}((e+2g-1)\infty)\bigr).
\]

\begin{lemma}\label{lem:Ee}
    For $e>2g-2$, we have  $R^1\pi_{J*}\cP_e=0$ and $\Ee_e$ is a vector bundle on $J$ of rank $e-g+1$ whose formation is compatible with arbitrary base change on $J$ (hence on $T$); its fibre at $[L]\in J$ over $t$ is $H^0(\mathcal{C}_t,L(e\infty_t))$. Similarly, for $e\geqslant 0$, $\Ge_e$ is a vector bundle of rank $e+g$ whose formation is compatible with base change and whose fibres are $H^0(\mathcal{C}_t,L^{-1}((e+2g-1)\infty_t))$.
\end{lemma}

\begin{proof}
    Fibrewise the line bundles have degrees $e>2g-2$ and $e+2g-1>2g-2$, so $H^1$ vanishes and $h^0$ equals the Euler characteristic, $e-g+1$ resp.\ $e+g$, by Riemann--Roch. The cohomology and base change theorem \cite[III.12.11]{Hartshorne} then gives local freeness and base change compatibility.
\end{proof}

The inclusion of relative Cartier divisors $(e-1)\infty\subset e\infty$ induces a short exact sequence 
\[
0\to\cP_{e-1}\to\cP_e\to\cP_e|_{J\times_T\infty(T)}\to 0,
\] and the normalization identifies $\cP_e|_{J\times_T\infty(T)}\cong\rho^*Q^{\otimes e}$. Pushing forward along $\pi_J$ yields a map
\[
    \ev_\infty\colon\Ee_e\longrightarrow\rho^*Q^{\otimes e},
\]
which fibrewise is the evaluation $H^0(\mathcal{C}_t,L(e\infty))\to L(e\infty)|_\infty\cong Q_t^{\otimes e}$. On each geometric fibre we choose a trivialization of $Q_t$ and use the resulting scalar notation for evaluations. Note that the nonvanishing conditions below do not depend on this choice.

Fix a geometric point $t$ of $T$, with algebraically closed field $\kappa\coloneqq \kappa(t)$. All statements in the rest of this section are over $\kappa$ and apply to the fibre $(C,\infty,J,\cP)\coloneqq  (\mathcal{C}_t,\infty_t,J_t,\cP_t)$. By Lemma \ref{lem:Ee}, the fibres of $\Ee_e,\Ge_e$ are the corresponding bundles for $(C,\infty)$.

\begin{prop}\label{prop:PnModel}
    Let $e>2g-2$. Then $\Mor_e(C,\PP^n)$ is naturally isomorphic to the open subscheme of $\PP(\Ee_e^{\oplus(n+1)})$ parametrizing $([L],[s_0:\cdots:s_n])$ with $s_i\in H^0(C,L(e\infty))$ having no common zero on $C$.
\end{prop}

\begin{proof}
    A morphism $f\colon C\to\PP^n$ with $\deg f^*\cO(1)=e$ is the same as a line bundle $M$ of degree $e$ together with sections $s_0,\dots,s_n$ without common zero, up to simultaneous scaling. Writing $M=L(e\infty)$ with $L\coloneqq M(-e\infty)\in\Pic^0(C)$ identifies this data with a point of $\PP(\Ee_e^{\oplus(n+1)})$ over $[L]\in J$, and the condition of having no common zero is open. The identification is functorial in families by base-change compatibility of $\Ee_e$ and the basic properties of $\Mor$ \cite[I.1.10]{Kollar}.
\end{proof}

In the model of Proposition \ref{prop:PnModel}, let $P\coloneqq \PP(\Ee_e^{\oplus(n+1)})$ with projection $q_P\colon P\to J$. We pull back the counit $\pi_J^*\Ee_e\to\cP_e$ along 
\[
    \rho_P\colon C\times P\to J\times C, \quad (x,p) \mapsto (q_P(p),x),
\]
and compose with the tautological inclusion. This gives 
\[
    \Phi\colon p^*\cO_P(-1)\to\rho_P^*\cP_e^{\oplus(n+1)}
\]
on $C\times P$ (where $p\colon C\times P\to P$), i.e.\ a section $\sigma$ of $\rho_P^*\cP_e^{\oplus(n+1)}\otimes p^*\cO_P(1)$. The basepoint locus is $p(Z(\sigma))$, and $\Mor_e(C,\PP^n)=P\setminus p(Z(\sigma))$.

\subsection{Imposing the equations of \texorpdfstring{$X$}{X} via multiplication maps}\label{subsec:mult}

Let $a\colon J\times J\to J$ be addition and $[d]\colon J\to J$ multiplication by $d$. By the universal property and normalization of the Poincar\'e bundle, $(a\times\operatorname{id}_C)^*\cP\cong p_{13}^*\cP\otimes p_{23}^*\cP$ on $J\times J\times C$ \cite[Section 13]{Mumford}; iterating and restricting to the diagonal gives $([d]\times\operatorname{id}_C)^*\cP\cong\cP^{\otimes d}$. For $d\geqslant 1$ we set
\[
    \Ee^{(d)}_e\coloneqq \pi_{J*}\bigl(\cP^{\otimes d}\otimes\pi_{\mathcal{C}}^*\cO(de\infty)\bigr),
\]
a vector bundle of rank $de-g+1$ whose formation commutes with base change. Its fibre at $[L]$ is
\[
    H^0(C,L^{\otimes d}(de\infty)).
\]
This follows as in Lemma \ref{lem:Ee}, since the fibrewise degree is $de>2g-2$. By flat base change along $[d]$ and the displayed isomorphism, we have $\Ee^{(d)}_e\cong[d]^*\Ee_{de}$; note that all this uses the normalization.

Moreover, the natural target of fibrewise multiplication over $[L]$ is $H^0(L^{\otimes d}(de\infty))$, which is the fibre of $\Ee_{de}$ at $[d]\cdot[L]$, not at $[L]$. So the multiplication map lands in the \emph{pullback} $[d]^*\Ee_{de}=\Ee^{(d)}_e$, not in $\Ee_{de}$ itself. Pushing forward the induced multiplication of sections along $J\times J\times C\to J\times J$, restricting to the diagonal, and symmetrizing thus yields, for each $d\geqslant 1$, a map of bundles on $J$
\[
    \mu_d\colon\operatorname{Sym}^d\Ee_e\longrightarrow\Ee^{(d)}_e=[d]^*\Ee_{de},
\]
which is fibrewise the multiplication $\operatorname{Sym}^d H^0(L(e\infty))\to H^0(L^{\otimes d}(de\infty))$.

\begin{construction}\label{constr:sigmaf}
    For a homogeneous $f\in\kappa[x_0,\dots,x_n]_d$, applying $\mu_d$ coordinate-wise gives $\Psi_f\colon\operatorname{Sym}^d(\Ee_e^{\oplus(n+1)})\to\Ee^{(d)}_e$ and, on $P$, a section $\sigma_f\in H^0(P,q_P^*\Ee^{(d)}_e \otimes\cO_P(d))$ whose zero scheme is fibrewise $\{[(s_0,\dots,s_n)]:f(s_0,\dots,s_n)=0\in H^0(L^{\otimes d}(de\infty))\}$.
\end{construction}

\begin{lemma}\label{lem:XinP}
    If $X\subset\PP^n_\kappa$ is cut out set-theoretically by homogeneous $f_1,\dots,f_s$ of degrees $d_1,\dots,d_s$, let
    \[
        Z_e(C;f_\bullet)\coloneqq \Bigl(P\setminus p(Z(\sigma))\Bigr)\ \cap\ Z(\sigma_{f_1})\cap\cdots\cap Z(\sigma_{f_s}).
    \]
    Here the sections are those of Construction \ref{constr:sigmaf}. Then $Z_e(C;f_\bullet)_{\mathrm{red}}=\Mor_e(C,X)_{\mathrm{red}}$.
\end{lemma}

\begin{proof}
    A basepoint free tuple $\mathbf{s}=(s_0,\dots,s_n)$ defines $\phi_{[\mathbf{s}]}\colon C\to\PP^n$, and this factors through $X$ if and only if $f_q(\mathbf{s})=0$ for all $q$. So the two loci have the same geometric points, hence the same reduction.
\end{proof}

In particular, compactly supported Betti numbers depend only on the reduced structure.

\subsection{Constructing a stratified affine space bundle over the morphism space}
Let $H_e\coloneqq \pi_{J*}\pi_{\mathcal{C}}^*\cO((2e+2g-1)\infty)$. After restricting this to the fixed geometric fibre $t$ whose residue field is $\kappa$, it is the trivial bundle $H^0(C,\cO((2e+2g-1)\infty))\otimes\cO_J$ of rank $2e+g$ by flat base change, since $2e+2g-1>2g-2$, and it contains the constant section $1_\infty$, the image of $1\in H^0(C, \cO_C)$. Since $\cP\otimes\cP^\vee\cong\cO$, multiplication of sections gives a pairing $\Ee_e\otimes\Ge_e\to H_e$.

\begin{definition}\label{def:Morw}
    For the chosen equations $f_\bullet=(f_1,\dots,f_s)$, let $\Gm$ act on $\Ee_e^{\oplus(n+1)}\times_J\Ge_e^{\oplus(n+1)}$ by $u\cdot(\boldsymbol{s},\boldsymbol{h})=(u\boldsymbol{s},u^{-1}\boldsymbol{h})$, where $\boldsymbol{s} = (s_0,\dots,s_n)$ and $\boldsymbol{h} = (h_0,\dots,h_n)$. Then $\Morw_e(C,X)$ is the locally closed subscheme of 
    \[
        \bigl((\Ee_e^{\oplus(n+1)}\setminus 0)\times_J\Ge_e^{\oplus(n+1)}\bigr)/\Gm
    \]
    defined by $f_q(\boldsymbol{s})=0$ for all $q$, $\sum_{i=0}^n s_ih_i=1_\infty$ in $H_e$, and the ($\Gm$-invariant) open condition that $\ev_\infty(s_j)\ne 0$ for some $j$. The quotient exists as a scheme: it is the total space of $q_P^*\Ge_e^{\oplus(n+1)}\otimes\cO_P(-1)$ over $P=\PP(\Ee_e^{\oplus(n+1)})$.We suppress from the notation the dependence of $\Morw_e(C,X)$ on the chosen equations $f_\bullet$.
\end{definition}

These conditions do force the $s_i$ to have no common zero. Note that as a section of $\cO((2e+2g-1)\infty)$, the element $1_\infty$ is \emph{not} nowhere vanishing: its zero divisor is exactly $(2e+2g-1)\infty$, so it vanishes at $\infty$ to order $2e+2g-1$. So $\sum_i s_ih_i=1_\infty$ rules out a common zero of the $s_i$ away from $\infty$, but is vacuous at $\infty$ itself, where both sides vanish. The open condition is what rules out a common zero at $\infty$, and it is exactly the condition partitioned by the strata of Definition \ref{def:strata} below.

\begin{lemma}\label{lem:affine}
    The forgetful map $\Morw_e(C,X)\to Z_e(C;f_\bullet)$, $(\boldsymbol{s},\boldsymbol{h})\mapsto[\boldsymbol{s}]$, is a torsor under a vector bundle of rank $(n-1)e+ng$. Hence
    \[
        \bc(\Morw_e(C,X))=\bc(Z_e(C;f_\bullet))=\bc(\Mor_e(C,X)).
    \]
\end{lemma}

\begin{proof}
    The map is well defined by the basepoint freeness built into Definition \ref{def:Morw}. Conversely, over a point of $Z_e(C;f_\bullet)$ the open condition of that definition is automatic, since a basepoint free tuple cannot have all $s_i$ vanishing at $\infty$, so the fibres are as computed below.
    
    Fix $[\boldsymbol{s}]\in Z_e(C;f_\bullet)(\kappa)$ over $[L]\in J$, with $M\coloneqq L(e\infty)$ base point free of degree $e$, and set $N_L\coloneqq L^{-1}((e+2g-1)\infty)=\CHom(M,\cO((2e+2g-1)\infty))$. After normalizing the $\Gm$-action, the fibre is $\{h\in H^0(C,N_L)^{\oplus(n+1)}:\sum_i s_ih_i=1_\infty\}$. Tensoring the evaluation sequence
    \[
        0\to M_s\to\cO_C^{\oplus(n+1)}\xrightarrow{(s_i)}M\to 0,
    \] 
    with $N_L$ gives
    \[
        0\to M_s\otimes N_L\to N_L^{\oplus(n+1)}\xrightarrow{\ \beta\ } M\otimes N_L=\cO((2e+2g-1)\infty)\to 0.
    \]
    The long exact sequence gives $\coker H^0(\beta)\hookrightarrow H^1(C,M_s\otimes N_L)$, and we claim that the latter group vanishes. By Serre duality we have $H^1(M_s\otimes N_L)^\vee=H^0(M_s^\vee\otimes\omega_C\otimes N_L^{-1})$. Dualizing the evaluation sequence
    \[
        0\to M^{-1}\to\cO^{\oplus(n+1)}\to M_s^\vee\to 0
    \] 
    shows $M_s^\vee$ is globally generated of degree $e$. A nonzero section of $M_s^\vee\otimes\omega_C\otimes N_L^{-1}$ would give a line subbundle $\Lambda\subseteq M_s^\vee$ of degree $\geqslant\deg(N_L\otimes\omega_C^{-1})=(e+2g-1)-(2g-2)=e+1$. But the torsion-free part of $M_s^\vee/\Lambda$ is a quotient of a globally generated bundle, hence globally generated of degree $\geqslant 0$, which forces $\deg\Lambda\leqslant\deg M_s^\vee=e$. This is a contradiction.
    
    Hence $H^0(\beta)$ is surjective and the fibre is a nonempty torsor under $H^0(C,M_s\otimes N_L)$, of dimension $\chi(M_s\otimes N_L)=(n+1)(e+g)-(2e+g)=(n-1)e+ng$, as $H^1=0$. The $H^1$'s vanish uniformly in $[\boldsymbol{s}]$. More precisely, over $Z_e(C;f_\bullet)\subset P$ the universal contraction
    \[
        q_P^*\Ge_e^{\oplus(n+1)}\otimes\cO_P(-1)\longrightarrow q_P^*H_e
    \]
    is surjective by cohomology and base change, and its kernel is the required vector bundle. Thus $\Morw_e$ is a torsor under this kernel, and applying Lemma \ref{lem:affbundle} gives the first equality. The second equality follows from Lemma \ref{lem:XinP} and invariance of \'etale cohomology under reduction.
\end{proof}

\begin{definition}\label{def:strata}
    For $0\leqslant j\leqslant n$, let $\Morw_e^{\,j}\subseteq\Morw_e(C,X)$ be the locus where $\ev_\infty(s_0)=\cdots=\ev_\infty(s_{j-1})=0$ and $\ev_\infty(s_j)\ne 0$.
\end{definition}

The strata $\Morw_e^{\,j}$ are disjoint, and they cover $\Morw_e$ precisely because of the open condition in Definition \ref{def:Morw}. On the preimage of $\Morw_e^{\,j}$ before taking the quotient, each $\Gm$-orbit has a unique representative satisfying $\ev_\infty(s_j)=1$; here the constraint $\sum_i s_ih_i=1_\infty$ is $\Gm$-invariant, with $h$ scaling inversely. This yields the following result:

\begin{lemma}\label{lem:stratsum}
    $\Morw_e^{\,j}$ is isomorphic to the locally closed subscheme of $\Ee_e^{\oplus(n+1)}\times_J\Ge_e^{\oplus(n+1)}$ cut out by $\ev_\infty(s_i)=0$ $(0\leqslant i\leqslant j-1)$, $\ev_\infty(s_j)=1$, $f_q(\boldsymbol{s})=0$ $(1\leqslant q\leqslant s)$, and $\sum_i s_ih_i=1_\infty$. Moreover
    \[
        \bc(\Mor_e(C,X))=\bc(\Morw_e(C,X))\leqslant\sum_{j=0}^n\bc\bigl(\Morw_e^{\,j}\bigr).
    \]
\end{lemma}

\begin{proof}
    The unique scaling with $\ev_\infty(s_j)=1$ trivializes the quotient and identifies $\Morw_e^{\,j}$ with the described subscheme. The equations $\ev_\infty(s_i)=0$ and $\ev_\infty(s_j)=1$ are pulled back from the map $\ev_\infty$. The left equality is by Lemma \ref{lem:affine}, and the inequality is by Lemma \ref{lem:excision}, whose ``filtrability" hypothesis holds here: the partial union $\bigcup_{j\leqslant j_0}\Morw_e^{\,j}$ is the open locus where $\ev_\infty(s_i)\ne 0$ for
    some $i\leqslant j_0$.
\end{proof}

\section{Uniform regularity and affine charts}\label{sec:regularity}

We now return to the relative setting of the family $(\CC\to T,\infty)$. In this section, we make uniform the geometric data needed for the affine presentations in Section \ref{sec:presentation}. We first choose a relative polarization on the universal Jacobian and establish a uniform regularity bound for the vector bundles appearing in the construction. We then choose finitely many affine charts and show that the coefficient sections arising on these charts can be represented by polynomials of uniformly bounded degree.

\subsection{Uniform regularity}
Let $\Lambda\coloneqq \bigl(\det R\pi_{J*}\cP_{g-1}\bigr)^{-1}$ be the inverse determinant of cohomology of the degree $(g-1)$ Poincar\'e twist. This is a line bundle on $J$ whose restriction to each fibre $J_t$ is the line bundle of the theta divisor $\Theta_t=\{[L]:h^0(L((g-1)\infty))>0\}$, which is ample.

Set $\cO_J(1)\coloneqq \Lambda^{\otimes 3}$. Fibrewise this is very ample, since the cube of an ample line bundle on an abelian variety is very ample in every characteristic \cite[Section 17]{Mumford}. Moreover, its higher cohomology vanishes on each fibre since an ample line bundle on an abelian variety has cohomology concentrated in degree $0$. After shrinking $T$, $\cO_J(1)$ is relatively very ample. Moreover, since $H^{>0}(J_t,\cO_J(1))=0$ on every fibre and $\cO_J(1)$ is $T$-flat, cohomology and base change shows $\rho_*\cO_J(1)$ is locally free of rank $N+1\coloneqq \chi(\cO_{J_t}(1))$ with formation commuting with base change; after shrinking $T$ it is free. This gives a closed embedding
\[
    J\hookrightarrow\PP^N_T
\]
under which $\cO_J(1)$ is the hyperplane bundle and $H^0(\PP^N_t,\cO(1))\to H^0(J_t,\cO_{J_t}(1))$ is an isomorphism on every fibre. For a coherent sheaf $F$ on $J$ and $k\in\ZZ$, we write $F(k)\coloneqq F\otimes\cO_J(k)$.

\begin{definition}\label{def:regular}
    Let $t$ be a geometric point of $T$. A coherent sheaf $F$ on $J_t$ is $m$-regular (with respect to $\cO_{J_t}(1)$) if $H^i(J_t,F(m-i))=0$ for all $i>0$. 
\end{definition}

The key feature we use later is that if $F$ is $m$-regular, then $F(m)$ is globally generated. Next, we define $(*)$ to consist of the following data:
\begin{itemize}
    \item the fixed pointed family $(\mathcal C\to T,\infty)$,
    
    \item the chosen polarization $\CO_J(1)$,

    \item and the degrees $d_1,\dots,d_s$.
\end{itemize}

\begin{prop}\label{prop:regularity}
    There is an integer $m=m(*)\geqslant 0$ depending on the data of $(*)$ such that for every geometric point $t$ of $T$ and every $e>2g-2$, the sheaves $\Ee_{e,t},\ \Ee_{e,t}^\vee,\ \Ge_{e,t},\ \Ge_{e,t}^\vee$ on $J_t$, together with, for each $d\in\{d_1,\dots,d_s\}$, the sheaves $\Ee^{(d)}_{e,t}$ and $\bigl(\Ee^{(d)}_{e,t}\bigr)^\vee$ of Section \ref{subsec:mult}, are $m$-regular.
\end{prop}

\begin{proof}
    Write $Y\coloneqq J\times_T\mathcal{C}$ and $\tau\colon Y\to T$, and recall that $\pi_{\mathcal{C}}\colon Y\to\mathcal{C}$ is the base change of $\rho\colon J\to T$, so $\pi_J^*\cO_J(1)$ is $\pi_{\mathcal{C}}$-very ample.
    
    \begin{enumerate}[(1)]
        \item $E_e$. Fix a geometric point $t$ and $k\in\ZZ$. Over the fibre of $t$ we have $R^{>0}(\pi_J)_{t*}\cP_{e,t}=0$ by Lemma \ref{lem:Ee}, so the Leray spectral sequence and the projection formula give
        \[
            H^i\bigl(J_t,\Ee_{e,t}(k)\bigr)\cong H^i\bigl(Y_t,\,(\cP\otimes\pi_J^*\cO_J(k)\otimes\pi_{\mathcal{C}}^*\cO(e\infty))_t\bigr).
        \]
        Set $\mathcal{G}_k\coloneqq \cP\otimes\pi_J^*\cO_J(k)$ on $Y$, which is a $T$-flat line bundle. By Lemma \ref{lem:serre} applied to $\pi_{\mathcal{C}}$ there is $k_+$ with $R^{>0}\pi_{\mathcal{C}*}\mathcal{G}_k=0$ for all $k\geqslant k_+$. Set $\mathcal{F}_k\coloneqq \pi_{\mathcal{C}*}\mathcal{G}_k$ to be a coherent sheaf on $\mathcal{C}$. For $k\geqslant k_+$ the Leray sequence for $\tau=\pi\circ\pi_{\mathcal{C}}$ degenerates and the projection formula gives
        \[
            R^i\tau_*\bigl(\mathcal{G}_k\otimes\pi_{\mathcal{C}}^*\cO(e\infty)\bigr)\cong R^i\pi_*\bigl(\mathcal{F}_k\otimes\cO(e\infty)\bigr).
        \]
        The right-hand side vanishes in degrees $i\geqslant2$, since $\pi$ has one-dimensional fibres. Lemma \ref{lem:peel} then gives $H^i(Y_t,(\mathcal{G}_k\otimes\pi_{\mathcal{C}}^* \cO(e\infty))_t)=0$ for every geometric fibre. For $i=1$, we use the single twist $k=k_+ +g-1$. Remark \ref{rem:powers} applied to the fixed sheaf $\mathcal{F}_{k_+ +g-1}$ on $\mathcal{C}$ then gives some $e^\dagger$ such that $R^1\pi_*(\mathcal{F}_{k_+ +g-1}\otimes\cO(e\infty))=0$ for $e\geqslant e^\dagger$; then Lemma \ref{lem:peel} gives the corresponding fibrewise vanishing throughout the range.
    
        Now set $m_1\coloneqq k_++g$. The displayed isomorphism shows $\Ee_{e,t}$ is $m_1$-regular for all $t$ and all $e\geqslant e^\dagger$: the case $i=1$ is the fixed twist $k=m_1-1=k_++g-1$; the cases $2\leqslant i\leqslant g$ have $k=m_1-i\geqslant k_+$, so the $H^{\geqslant 2}$-vanishing above applies; and $H^{>g}$ vanishes for dimension reasons. For each of the finitely many $e$ with $2g-2<e<e^\dagger$, $\Ee_e(k)$ is a single $T$-flat coherent sheaf on $J$. Lemma \ref{lem:serre} applied to $\rho$ gives $k_e$ with $R^{\geqslant 1}\rho_*(\Ee_e(k))=0$ for $k\geqslant k_e$, and Lemma \ref{lem:peel} with $i_0=1$ gives $H^{\geqslant 1}(J_t,\Ee_{e,t}(k))=0$ for all $t$ and $k\geqslant k_e$; hence $\Ee_{e,t}$ is $(k_e+g)$-regular for all $t$. Replacing $m_1$ by the maximum of its previous value and the finitely many numbers $k_e+g$ gives one constant valid for all $e>2g-2$. 
    
        \item $\Ee_e^\vee$. Write $(\pi_J)_t\colon J_t\times \CC_t\to J_t$. The restriction of $\cP_{e,t}^\vee\otimes\pi_{\mathcal C}^*\omega_{\CC_t}$ to every fibre of $(\pi_J)_t$ has degree $2g-2-e<0$ and therefore has no global sections. Since this line bundle is flat over $J_t$, cohomology and base change gives $(\pi_J)_{t*}\bigl(\cP_{e,t}^\vee\otimes \pi_{\mathcal C}^*\omega_{\CC_t}\bigr)=0$. The relative Grothendieck--Serre duality then gives
        \[
            \Ee_{e,t}^\vee\cong R^1(\pi_J)_{t*}\bigl(\cP_{e,t}^\vee\otimes\pi_{\mathcal{C}}^*\omega_{\CC_t}\bigr),
        \]
        and we have
        \[ 
            H^i\bigl(J_t,\Ee_{e,t}^\vee(k)\bigr)\cong H^{i+1}\bigl(Y_t,\,(\mathcal{G}'_k\otimes\pi_{\mathcal{C}}^* (\omega_{\mathcal{C}/T}(-e\infty)))_t\bigr)
        \]
        by the Leray spectral sequence. Here, we write $\mathcal{G}'_k\coloneqq \cP^\vee\otimes\pi_J^*\cO_J(k)$ and $\mathcal{F}'_k\coloneqq \pi_{\mathcal{C}*}\mathcal{G}'_k$. By Lemma \ref{lem:serre}, there is some $k_-$ with $R^{>0}\pi_{\mathcal{C}*}\mathcal{G}'_k=0$ for $k\geqslant k_-$. Then, as in (1), we have 
        \[
            R^i\tau_*(\mathcal{G}'_k\otimes\pi_{\mathcal{C}}^*(\omega(-e\infty)))\cong R^i\pi_*(\mathcal{F}'_k\otimes \omega(-e\infty))=0
            \]
        for $i\geqslant 2$, all $e$, all $k\geqslant k_-$, and Lemma \ref{lem:peel} with $i_0=2$ gives vanishing of $H^{\geqslant 2}$ on every fibre. Hence $H^i(J_t,\Ee_{e,t}^\vee(k))=0$ for all $i\geqslant 1$, all $k\geqslant k_-$, all $e>2g-2$, and all $t$. Note that no threshold in $e$ is needed here, because the obstruction sits in degree $i+1\geqslant 2$ of a relative curve. So $m_2\coloneqq k_-+g$ works for all $e$ and $t$.
    
        \item $\Ge_e$ and $\Ge_e^\vee$. This is identical to (1) and (2) with $\cP$ replaced by $\cP^\vee$ and $\cO(e\infty)$ by $\cO((e+2g-1)\infty)$; the fibrewise $R^0$-vanishing needed for $\Ge_e^\vee$ holds since $2g-2-(e+2g-1)=-e-1<0$. As in (1), taking the maximum of the large $e$ constant and the finitely many small $e$ constants produces a single constant $m_3$ valid for all $e>2g-2$, and the dual argument produces an unconditional constant $m_4$ as in (2).
    
        \item $\Ee^{(d)}_e$ and $(\Ee^{(d)}_e)^\vee$ for $d\in\{d_1,\dots,d_s\}$. This is identical to (1) and (2) with $\cP$ replaced by $\cP^{\otimes d}$ and $\cO(e\infty)$ by $\cO(de\infty)$: the relevant fibrewise degrees are $de>2g-2$ (for the $R^{>0}$-vanishing as in Lemma \ref{lem:Ee}) and $2g-2-de<0$ (for the $R^0$-vanishing in the duality step), and the twists $\cO(de\infty)$ with $e\geqslant e^\dagger$ form a subset of those handled by Remark \ref{rem:powers}. Since $d$ ranges over a fixed finite set, taking the relevant large-$e$ constants together with the finitely many small-$e$ constants produces constants $m_1^{(d)},m_2^{(d)}$ valid for all $e>2g-2$.
    \end{enumerate}
    
    Finally, let
    \[
        m\coloneqq \max\Bigl(0,m_1,m_2,m_3,m_4,\{m_1^{(d)},m_2^{(d)}:d\in\{d_1,\dots,d_s\}\}\Bigr).
    \]
    Then $m$-regularity implies $m'$-regularity for every $m'\geqslant m$ by \cite[Theorem 1.8.5]{LazarsfeldI}.
\end{proof}

\subsection{Affine charts and bounded-degree coefficients}

\begin{lemma}\label{lem:ideal}
    Let $\cI\subset\cO_{\PP^N_T}$ be the ideal sheaf of $J\hookrightarrow\PP^N_T$. There exists some nonnegative integer $b_1$ and finitely many forms $F_1,\dots,F_{r_0}\in H^0(\PP^N_T,\cI(b_1))$ such that for every geometric point $t$:
    \begin{enumerate}
        \item[(i)] the restrictions $F_{1,t},\dots,F_{r_0,t}$ generate $\cI_t(b_1)$; in particular they cut out $J_t\subset\PP^N_{\kappa(t)}$ scheme-theoretically;
        
        \item[(ii)] for all $b\geqslant b_1$, we have $H^1(\PP^N_t,\cI_t(b))=0$ so that $H^0(\PP^N_t,\cO(b))\to H^0(J_t,\cO_{J_t}(b))$ is surjective.
    \end{enumerate}
\end{lemma}

\begin{proof}
    Shrink $T$ first to a nonempty affine open. Since $\cO_J$ is $T$-flat, the sequence $0\to\cI\to\cO_{\PP^N_T}\to\cO_J\to 0$ remains exact after any base change, so $\cI_t$ is the ideal sheaf of $J_t$ and $\cI$ is $T$-flat. Lemma \ref{lem:serre} applied to $p\colon\PP^N_T\to T$ gives $b_1$ with $p^*p_*\cI(b_1)\twoheadrightarrow\cI(b_1)$ and $R^{\geqslant 1}p_*\cI(b)=0$ for $b\geqslant b_1$ and we take $F_1,\dots,F_{r_0}$ to be generators of $H^0(\PP^N_T,\cI(b_1))$ over $T$. Restricting the surjection to the fibre gives (i). For (ii), we apply Lemma \ref{lem:peel} with $i_0=1$ and $\cI(b)$ flat to conclude $H^{\geqslant 1}(\PP^N_t,\cI_t(b))=0$; the long exact sequence of the ideal sequence
    on the fibre then gives the surjectivity.
\end{proof}

We now choose affine charts on which the sections considered above can be written as regular functions.

\begin{lemma}\label{lem:charts}
    After replacing $T$ by a finite surjective cover and shrinking, there are linear forms $\sigma_0,\dots,\sigma_g\in H^0(\PP^N_T,\cO(1))$ whose common zero locus on $J_t$ is empty for every geometric point $t$.
\end{lemma}

\begin{proof}
    We work over the generic point $\eta$ of $T$. Since $J_\eta$ is $g$-dimensional and $\cO_J(1)$ gives the embedding $J_{\bar\eta}\hookrightarrow\PP^N$, the locus of $(g+1)$-tuples of hyperplanes meeting the image in a common point is a proper closed subset of the parameter space. Indeed, the incidence variety of tuples $(x,H_0,\dots,H_g)$ with $x\in J_{\bar\eta}\cap H_0\cap\cdots\cap H_g$ has dimension $g+(g+1)(N-1)=(g+1)N-1$, whereas $\bigl((\PP^N)^\vee\bigr)^{g+1}$ has dimension $(g+1)N$, and its image is closed because the incidence projection is proper. Hence a suitable tuple exists over a finite extension of $\kappa(\eta)$. We absorb this extension into a finite surjective cover $T'\to T$. After clearing denominators and shrinking $T'$, the chosen hyperplanes extend to global linear forms $\sigma_0,\dots,\sigma_g$. Upon further shrinking, each $\sigma_l$ has a unit coefficient, so a linear change of coordinates identifies its complement with $\Aff^N_{T'}$. The locus where the fibrewise common zero locus $Z(\sigma_0,\dots,\sigma_g)\cap J_t$ is nonempty is the image of a closed subscheme under the proper map $J_{T'}\to T'$. Hence it is closed and misses the generic point. Shrinking to its complement completes the proof.
\end{proof}

For each $l$, choose homogeneous coordinates on $\PP^N_T$ in which $\sigma_l=X_0$, and set 
\[
    U_l\coloneqq J\cap D_+(\sigma_l).
\]
Since $D_+(\sigma_l)\cong\Aff^N_T$, this realizes $U_l$ as a closed affine subscheme of $\Aff^N_T$. The section $\sigma_l$ trivializes $\cO_J(1)$ over $U_l$ and therefore induces, for every coherent sheaf $F$ and every $k\in\ZZ$, an identification 
\[
    F(k)|_{U_l}\cong F|_{U_l}.
\]
The opens $U_0,\dots,U_g$ cover $J$. On each geometric fibre, we refine this cover to the disjoint stratification
\[
    J_t=\bigsqcup_{l=0}^g\Sigma_{l,t},\] where \[\Sigma_{l,t}\coloneqq U_{l,t}\cap V(\sigma_0,\dots,\sigma_{l-1}).
\]
We have $\Sigma_{l,t}$ is closed in $U_{l,t}$ and is defined by $l\leqslant g$ linear equations. Combining these charts with Lemma \ref{lem:ideal} gives the uniform degree bound needed in the affine presentations. Set
\[
    a'\coloneqq \max\{b_1,\,(\max_{1\leqslant q\leqslant s}\{d_q\}+1)m,\,2m\}.
\]
Recall from Proposition \ref{prop:regularity} that $m\geqslant 0$, so all the twists $a\in\{m,\,2m,\,(d_q+1)m:1\leqslant q\leqslant s\}$ arising in Section \ref{sec:presentation} satisfy $0\leqslant a\leqslant a'$, as required here.

\begin{lemma}[Bounded-degree expression of coefficients]\label{lem:coeff}
    For every geometric point $t$, $0\leqslant a\leqslant a'$, and $\theta\in H^0(J_t,\cO_{J_t}(a))$, the regular function $\theta/\sigma_l^{\,a}$ on $U_{l,t}$ is the restriction of a polynomial of degree $\leqslant a'$ in the chart coordinates of $\Aff^N$.
\end{lemma}

\begin{proof}
    The element $\theta\cdot\sigma_l^{\,a'-a}\in H^0(J_t,\cO(a'))$ lifts to a form $F_\theta$ of degree $a'$ on $\PP^N_{\kappa(t)}$ by Lemma \ref{lem:ideal} (ii). On the chart $U_{l,t}$, we have that 
    \[
        \frac{\theta}{\sigma_l^{\,a}} = \frac{F_\theta}{\sigma_l^{\,a'}}
    \]
    is the dehomogenization of $F_\theta$ and therefore has degree at most $a'$.
\end{proof}

Note that the degree bound $a'$ is uniform over the family.

\section{Putting everything together}\label{sec:presentation}

We now assemble the proof of Theorem \ref{T:family}. Using the uniform data constructed in Section \ref{sec:regularity}, we replace each stratum of the modified morphism space by a closed subscheme of affine space with the same compactly supported Betti numbers. We then show that the numbers of variables and equations grow at most linearly with $e$, while the degrees of the equations remain bounded independently of $e$. Katz's estimate will then give the desired exponential bound.

Fix the data produced in Section \ref{sec:regularity}: the constant $m$, the embedding $J\hookrightarrow\PP^N_T$, the forms $F_i$ of degree $b_1$, the charts defined by $\sigma_l$, and $a'$. After replacing $T$ by a dense open subset of a finite surjective cover, we may assume that all these data are defined over $T$. Fix a geometric point $t$ with $\kappa\coloneqq \kappa(t)$, a reduced closed subscheme $X\subset\PP^n_\kappa$ cut out set-theoretically by $f_1,\dots,f_s$ of degrees $d_1,\dots,d_s$, and $e\geqslant 2g-1$. We work henceforth over the fibre at $t$, write $C\coloneqq\mathcal{C}_t$, and suppress $t$ from the notation for the restrictions of the vector bundles to $J_t$.

\subsection{Choosing fibrewise generators}

By Proposition \ref{prop:regularity} and Definition \ref{def:regular}, the sheaves $\Ee_e(m)$, $\Ge_e(m)$, and $\bigl(\Ee^{(d_q)}_e\bigr)^\vee(m)$, for $1\leqslant q\leqslant s$, on $J_t$ are globally generated and have no higher cohomology. Hence, setting their spaces of global sections to have dimensions
\[
    r_1\coloneqq h^0(\Ee_e(m)), r_3\coloneqq h^0(\Ge_e(m)), \text{ and } r_2^{(q)}\coloneqq h^0\bigl((\Ee^{(d_q)}_e)^\vee(m)\bigr),
\]
they are respectively equal to the corresponding Euler characteristics.

These dimensions are affine-linear in $e$. Indeed, writing $D_t=J_t\times\{\infty_t\}\subset J_t\times C$, the divisor exact sequence and the normalization of $\cP$ give 
\[
    \chi\bigl(J_t\times C,\cP_e\otimes\pi_J^*\cO(m)\bigr)- \chi\bigl(J_t\times C, \cP_{e-1}\otimes\pi_J^*\cO(m)\bigr) =\chi(J_t,\cO(m)).
\]
Since $R^1\pi_{J*}\cP_e=0$ for $e>2g-2$, this proves the asserted linearity for $\chi(J_t,\Ee_e(m))$. The same argument applies to $\Ge_e(m)$. For the dual bundle, relative duality gives
\[
    \chi\bigl(J_t,(\Ee_e^{(d_q)})^\vee(m)\bigr) =-\chi\bigl(J_t\times C, \cP^{-d_q}\otimes\pi_{\mathcal C}^*\omega_{C}(-d_qe\infty) \otimes\pi_J^*\cO(m)\bigr).
\]
Then the conclusion follows by applying the divisor exact sequence $d_q$ times as $e$ increases by one. Since the Euler characteristics involved are locally constant in the proper flat family, there are constants $A_{\mathrm{lin}},B_{\mathrm{lin}}$, uniform on $T$, such that
\[
    r_1,\ r_3,\ r_2^{(q)} \leqslant A_{\mathrm{lin}}e+B_{\mathrm{lin}} \text{ for } 1\leqslant q\leqslant s.
\]
Choose bases $\{v_\alpha\}\subset H^0(\Ee_e(m))$, $\{u_\gamma\}\subset H^0(\Ge_e(m))$, $\{w^\beta_{(q)}\}\subset H^0\bigl((\Ee^{(d_q)}_e\bigr)^\vee(m))$ together with a basis of $H^0(C,\cO((2e+2g-1)\infty))$.

\begin{lemma}\label{lem:pairing}
    Let $F$ be a vector bundle on an open subscheme $U\subseteq J_t$, and let $w^1,\dots,w^r\in\Gamma(U,F^\vee)$ be sections whose values span $F^\vee|_u$ at every point $u\in U$. Then, for every $\xi\in\Gamma(U,F)$, the zero scheme $Z(\xi)\subseteq U$ is cut out scheme-theoretically by the functions
    \[
        \langle w^1,\xi\rangle,\dots,\langle w^r,\xi\rangle\in\Gamma(U,\cO_U).
    \]
\end{lemma}

\begin{proof}
    The sections define a morphism
    \[
        \varphi\colon\cO_U^{\oplus r}\to F^\vee \text{ with } e_\beta\mapsto w^\beta.
    \]
    Its cokernel is coherent and has zero fibre at every point, so it vanishes by Nakayama's lemma. Hence $\varphi$ is surjective. On the other hand, the ideal sheaf of $Z(\xi)$ is the image of the evaluation map
    \[
        F^\vee\to\cO_U \text{ with } w\mapsto\langle w,\xi\rangle.
    \]
    Since $\varphi$ is surjective, this image is also the image of the composition
    \[
        \cO_U^{\oplus r}\xrightarrow{\varphi}F^\vee \xrightarrow{\langle-\, ,\xi\rangle}\cO_U.
    \]
    The latter is the ideal generated by $\langle w^1,\xi\rangle,\dots,\langle w^r,\xi\rangle$, and this proves the claim.
\end{proof}

\subsection{Introducing affine coordinates}
On the chart $U_{l,t}$, where all twists are trivialized by $\sigma_l$, the chosen sections give surjections $\cO_{U_{l,t}}^{\oplus r_1}\twoheadrightarrow \Ee_e|_{U_{l,t}}$ and $\cO_{U_{l,t}}^{\oplus r_3}\twoheadrightarrow \Ge_e|_{U_{l,t}}$. Write each $s_i$ as the image of $c^{(i)}\in\Aff^{r_1}$ and each $h_i$ as the image of $\delta^{(i)}\in\Aff^{r_3}$, and set $\mathbf c=(c^{(0)},\dots,c^{(n)})$ and $\boldsymbol\delta=(\delta^{(0)},\dots,\delta^{(n)})$. This defines a surjective map
\[
    \Pi\colon U_{l,t}\times\Aff^{(n+1)r_1}\times\Aff^{(n+1)r_3}\longrightarrow \bigl(\Ee_e^{\oplus(n+1)}\times_J\Ge_e^{\oplus(n+1)}\bigr)\big|_{U_{l,t}}.
\] 
Let $K$ be the kernel of the resulting surjection of vector bundles on $U_{l,t}$ and let $p$ denote the projection from the target of $\Pi$ to $U_{l,t}$. Then for every locally closed subscheme $W$ of the target, the map $\Pi^{-1}(W)\to W$ is a torsor under $(p|_W)^*K$. Lemma \ref{lem:affbundle} therefore gives
\[
    \bc(\Pi^{-1}(W))=\bc(W).
\]

\subsection{Writing down the equations and bounding their degrees}

Let $\widehat{\Mor}{}^{\,j,l}_e\coloneqq \Pi^{-1}\bigl(\Morw_e^{\,j}|_{\Sigma_{l,t}}\bigr)$. It is the closed subscheme of $\Aff^{N+(n+1)(r_1+r_3)}_\kappa$ cut out by:
\begin{itemize}
    \item the dehomogenized forms $F_1,\dots, F_{r_0}$ of degree at most $b_1$, cutting out $U_{l,t}\subset\Aff^N$, together with the $l$ linear equations
    \[
        \sigma_0 = \dots = \sigma_{l-1} = 0
    \]
    cutting out $\Sigma_{l,t}$;
    
    \item the condition $f_q(\boldsymbol s)=0$ for each $1\leqslant q\leqslant s$. After trivializing the twists on $U_{l,t}$, the section $\Psi_{f_q}(\mathbf{s})$ is a section of $\Ee^{(d_q)}_e$ over the affine parameter space. Since the sections $w^\beta_{(q)}$ span the fibres of its dual, Lemma \ref{lem:pairing} shows that this condition is imposed scheme-theoretically by the $r_2^{(q)}$ equations
    \[
        \langle w^\beta_{(q)},\Psi_{f_q}(\boldsymbol s)\rangle=0.
    \]
    
    Each of these equations is a polynomial of degree $d_q$ in the $\mathbf c$-variables. Here $\mu_{d_q}$ denotes the induced twisted map
    \[
        \operatorname{Sym}^{d_q}\bigl(\Ee_e(m)\bigr)\longrightarrow\Ee^{(d_q)}_e(d_qm).
    \]
    The coefficients of these equations $\langle w^\beta_{(q)},\mu_{d_q}(v_{\alpha_1}\otimes\cdots\otimes v_{\alpha_{d_q}})\rangle$ lie in $H^0(J_t,\cO((d_q+1)m))$, hence are polynomials of degree at most $a'$ in the chart coordinates by Lemma \ref{lem:coeff}. The total degree of the polynomial is therefore at most $d_q+a'$;
    
    \item the condition $\sum_i s_ih_i=1_\infty$ is imposed as follows. The multiplication pairing induces \[\Ee_e(m)\otimes\Ge_e(m)\longrightarrow H_e(2m).\] With respect to the chosen basis of $H^0(C,\cO((2e+2g-1)\infty))$, its coefficients lie in $H^0(J_t,\cO(2m))$. By Lemma \ref{lem:coeff}, these become polynomials of degree at most $a'$ on $U_{l,t}$. Thus $\sum_i s_ih_i=1_\infty$ gives $2e+g$ equations, bilinear in $(\mathbf c,\boldsymbol\delta)$, each of total degree at most $2+a'$;  
   \item after trivializing $Q_t^{\otimes e}$, the evaluation map induces \[\Ee_e(m)\longrightarrow\cO_{J_t}(m).\] Its coefficients therefore lie in $H^0(J_t,\cO(m))$ and, by Lemma \ref{lem:coeff}, become polynomials of degree at most $a'$ on $U_{l,t}$. Consequently, the conditions $\ev_\infty(s_i)=0$ for $i<j$ and $\ev_\infty(s_j)=1$ give $j+1\leqslant n+1$ equations, linear in $\mathbf c$, each of total degree at most $1+a'$.
\end{itemize}

Set $B\coloneqq \max\bigl(b_1,\ \max_{1\leqslant q\leqslant s} d_q+a',\ 2+a'\bigr)$. Then every defining equation above has degree at most $B$.

By the affine linear bounds above, there is a constant $A_0>0$, depending only on the fixed family data and on $n$, such that, for every $e\geqslant 2g-1$, both the ambient dimension and the number of equations are at most $A_0e$. In particular, we have \[N+(n+1)(r_1+r_3)\leqslant A_0 e\] and \[r_0+g+\sum_{q=1}^s r_2^{(q)}+(2e+g)+(n+1) \leqslant A_0e.\] Here, we are using that $e\geqslant 1$ in order to absorb the constant terms.

\subsection{Applying Katz's bound}

\begin{theorem}[Katz's bound {\cite[Remark following Theorem 12]{Katz}}]\label{thm:katz} 
    Let $\kappa$ be algebraically closed and $\ell$ invertible in $\kappa$. Let $W$ be a closed subscheme of $\Aff^N_\kappa$ defined by $r$ polynomials of degree $\leqslant d$. Then we have $\sum_i\dim H^i_c(W,\Ql)\leqslant 3\,(2+d)^{N+r}$.
\end{theorem}

\begin{proof}
    Katz proves the estimate when $\kappa$ is the algebraic closure of a finite field, and we reduce the general
    case to this one. Write $W=V(F_1,\dots,F_r)$ and choose a finitely generated subring $A\subset\kappa$
    (necessarily a domain) containing the coefficients of the $F_i$ and in which $\ell$ is invertible. Then $A$ is a
    finitely generated $\mathbb F_p$-algebra if $\carac\kappa=p>0$, and a finitely generated $\ZZ[1/\ell]$-algebra
    if $\carac\kappa=0$. Let $S=\Spec A$ and let $\mathcal W\subset\Aff^N_S$ be cut out by the same $r$ equations,
    regarded as elements of $A[x_1,\dots,x_N]$.

    For $\varpi\colon\mathcal W\to S$, the sheaves
    $R^i\varpi_!\Ql$ are constructible and for every geometric point $\bar s\to S$, there is a canonical
    identification
    \[
        (R^i\varpi_!\Ql)_{\bar s}\cong H^i_c(\mathcal W_{\bar s},\Ql).
    \]
    After possibly shrinking $S$, all the finitely many non-zero $R^i\varpi_!\Ql$ are lisse.

    Note that $S$ is of finite type over $\mathbb F_p$ or $\ZZ[1/\ell]$, so its closed points are dense and have finite residue fields. Choose a closed point $s\in S$ and a geometric point $\bar s$ above it; its characteristic is different from $\ell$. Since $A\subset\kappa$ is injective, $\bar\eta\coloneqq \Spec\kappa\to S$ is a geometric generic point, and $\mathcal W_{\bar\eta}=W$. The open $S$ is integral, hence connected, so lisse-ness gives, for every $i$,
    \[
        \dim H^i_c(W,\Ql) =\rk R^i\varpi_!\Ql =\dim H^i_c(\mathcal W_{\bar s},\Ql).
    \]
    After discarding equations that specialize to zero, $\mathcal W_{\bar s}$ is cut out by at most $r$ equations of degrees at most $d$. Then the monotonicity of $3(2+d)^{N+r}$ in $r$ and $d$ and Katz's finite field estimate give 
    \[
        \sum_i\dim H^i_c(W,\Ql) =\sum_i\dim H^i_c(\mathcal W_{\bar s},\Ql) \leqslant 3(2+d)^{N+r},
    \]
    as desired.
\end{proof}

\begin{proof}[Completion of the proof of Theorem \ref{T:family}]
    The fibrewise stratification $J_t=\bigsqcup_l\Sigma_{l,t}$ is filtrable: the partial union $\bigcup_{l\leqslant l_0}\Sigma_{l,t}$ is the complement of the closed locus $\{\sigma_0=\cdots=\sigma_{l_0}=0\}$. Applying Lemmas \ref{lem:stratsum} and \ref{lem:excision}, together with the earlier equalities $\bc(\Morw_e^{\,j}|_{\Sigma_{l,t}})=\bc(\widehat{\Mor}{}^{\,j,l}_e)$, gives
    \[
        \bc\bigl(\Mor_e(C,X)\bigr)\leqslant\sum_{j=0}^n\sum_{l=0}^g\bc\bigl(\widehat{\Mor}{}^{\,j,l}_e\bigr) \leqslant (n+1)(g+1)\cdot 3\,(2+B)^{2A_0 e}\leqslant K_V^{\,e}
    \]
    with $K_V\coloneqq \bigl(3(n+1)(g+1)\bigr)\,(2+B)^{2A_0}$, using $e\geqslant 1$ so that $3(n+1)(g+1)\leqslant (3(n+1)(g+1))^e$. This proves Theorem \ref{T:family}, and as explained earlier, Theorem \ref{thm:main} as well.
\end{proof}

\part{Manin's conjecture for higher genus curves\\ on split quartic del Pezzo surfaces}\label{part 2}

This part is devoted to the proof of Theorem \ref{thm:main-manin}. The general strategy follows that of \cite{DLTT25}, but requires a few new inputs for a source curve of higher genus. The first is the uniform Betti number bound of Theorem \ref{thm:main}. The second is the analysis of the geometry of higher genus morphism spaces and their incidence strata, established in Section \ref{sec:Higher genus curves on del Pezzo surfaces}. The third is an estimation for the Frobenius traces on configuration spaces developed in Section \ref{sec:The Virtual Bar Complex}. We give a brief outline of the strategy and describe how these inputs enter the proof.

Throughout Part \ref{part 2}, let $k=\BF_q$, and fix an algebraic
closure $\overline{k}$ of $k$. Unless otherwise specified, all schemes
and morphisms are defined over $k$.

\section{Outline of the proof}

Fix a numerical curve class $\alpha$. We let $M_\alpha$ denote the space parametrizing morphisms whose images have class $\alpha$. Using the chamber decomposition of $\Nef_1(S)$, we may choose a birational morphism 
\[
    \rho\colon S \to \BP^1\times\BP^1
\]
contracting four disjoint $(-1)$-curves. Sections \ref{subsec:The moduli space of curves and sections} -- \ref{subsec:Stratifying the space of sections} use this presentation and recall the construction of a $\BG_m^2$-torsor 
\[
    \widetilde{M}_{\alpha} \to M_\alpha
\]
as an open subset of an affine bundle $E$ over a parameter space of line bundles and configurations of points on $B$. The bundle is stratified by linear incidence conditions indexed by a certain poscheme.

The idea to count the number of $\BF_q$-points on $M_\alpha$ is then to apply the Grothendieck--Lefschetz trace formula to $\widetilde{M}_{\alpha}$:
\[
    \#\widetilde{M}_{\alpha}(\BF_q) = \sum_{i} (-1)^i \trace (\Frob \curvearrowright H^i_c((\widetilde{M}_{\alpha})_{\overline{\BF}_q}, \Ql)).
\]
We divide the alternating sum into two parts based on cohomological degrees. To bound the contribution of low degree, we apply Theorem \ref{thm:main} together with Deligne's estimate of Frobenius eigenvalues. For the contribution of high degree, we rely on the explicit geometry of $\widetilde{M}_\alpha$ and a cohomological inclusion-exclusion.

The main geometric input is carried out in Section \ref{sec:Higher genus curves on del Pezzo surfaces}. We study the configuration cover parametrizing sections through prescribed subschemes of $B$, and use it to prove a structural theorem for $M_\alpha$, deducing the irreducibility and smoothness in Theorem \ref{thm:irreducibility of M_alpha}. Along the way, we obtain dimension estimates in Proposition \ref{prop:dimboundWak} for the linear incidence spaces stratifying the affine bundle $E$. These dimensional estimates allow us to remove a locus $F$ of high codimension so that all incidence strata relevant to the principal term have the expected dimension. The codimension of $F$ also ensures that its contribution occurs outside the range in which the higher degree trace is computed. 

Section \ref{sec:The Virtual Bar Complex} then packages the inclusion-exclusion among the incidence strata into a virtual bar complex. Since the relevant strata have their expected
dimensions over the good locus, the virtual bar complex computes the contribution of these strata with the desired error bounds. The higher genus contribution
enters through the Picard groups and the configuration spaces of points on $B$, and the Frobenius traces of the virtual bar complex are controlled by Theorem \ref{thm:bounding truncated virtual bar complex}.

Finally, Section \ref{sec:Manin's conjecture for higher genus curve} combines the low degree estimate from
Theorem \ref{thm:main} with the high degree trace computed by the virtual
bar complex. The resulting inclusion-exclusion sum is evaluated using
a height zeta function, giving a uniform point-counting estimate 
\[
    \#M_\alpha(\mathbb{F}_q) = \tau_{-K_S}(S)q^{- K_S\cdot\alpha} + O\left(q^{(1-\delta\varepsilon)(-K_S\cdot\alpha)}\right)
\]
for sufficiently positive $\alpha$ and small positive rationals $\delta$ and $\varepsilon$. We then sum these estimates over the rational polyhedral conical region $\Nef_1(S)_{\ell,\varepsilon}$ to obtain Theorem \ref{thm:main-manin}.

\section{Preliminaries} \label{sec:Preliminaries}

\subsection{Quartic del Pezzo surfaces}
\label{subsec:quartic del pezzo surfaces}

We record some basic facts about split quartic del Pezzo surfaces following \cite[Section 3]{DLTT25}. 

\begin{definition}
    A degree $d$ del Pezzo surface $S$ over a field $k$ is a smooth projective Fano variety of dimension $2$ whose anticanonical bundle satisfies $(-K_S)^2 = d$. We say $S$ is split if its Picard rank satisfies $\rho(S) = \rho(S_{\overline{k}})$. 
\end{definition}

Given a split quartic del Pezzo surface $S$, i.e. $(-K_S)^2 =4$, \cite[Proposition 3.3]{DLTT25} provides a chamber decomposition of $\Nef_1(S)$ into subcones:
\[
\Nef_1(S) = \bigcup_{\overset{\rightarrow}{E} = (E_1, E_2, E_3)} \CT_{\overset{\rightarrow}{E}}
\]
where $(E_1, E_2, E_3)$ runs through all ordered triples of disjoint $(-1)$-curves on $S$, and the intersection of the interior of these subcones is empty for different ordered triples. 

For any $\alpha \in \CT_{\overset{\rightarrow}{E}}$, there exists a sequence of contractions of disjoint $(-1)$-curves $E_1, E_2, E_3$ on $S$: 
\[
S \to S_3 \to S_2 \to S_1.
\]
Denote by $\psi_i\colon S\to S_i$ the morphism from $S$ to a degree $8-i$ split del Pezzo surface, by $H$ the pullback of the hyperplane section of $\BP^2$ to $S_1$, and by $F_1$ and $F_2$ the classes of conics on $S_1$. There exist non-negative integers $b,b_3,b_2,x,y_1,y_2$ such that
\[
\alpha = -bK_S - b_3\psi_3^*K_{S_3} - b_2\psi_2^*K_{S_2} + \psi_1^*(xH+y_1F_1+y_2F_2).
\]
If all these integers except possibly $b$ are positive, then the contraction of $E_1, E_2, E_3$ giving such an expression of $\alpha$ is unique. Consider the morphism
\[
\rho_\alpha\colon S \to \BP^1\times \BP^1
\]
obtained by composing $\psi_1$ with the unique contraction of a $(-1)$-curve $E_4$ on $S_1$. Let $p_\alpha$ (resp. $p'_\alpha$) be the composition of $\rho_\alpha$ with the first (resp. second) projection $\BP^1\times\BP^1\to\BP^1$. Denote by $F_\alpha$ (resp. $F'_\alpha$) a general fibre of $p_\alpha$ (resp. $p'_\alpha$). We define the following intersection numbers:
\[
h(\alpha) = -K_S\cdot \alpha, \hspace{5mm} a(\alpha) = F_\alpha \cdot \alpha, \hspace{5mm} a'(\alpha) = F'_\alpha \cdot \alpha, \hspace{5mm}  k_i(\alpha) = E_i \cdot \alpha.
\]
Since $-K_S = 2F_\alpha + 2F'_\alpha - \sum_{i=1}^4 E_i$, we have the relation
\[
h(\alpha) = 2a(\alpha) + 2a'(\alpha) - \sum_{i=1}^4 k_i(\alpha).
\]
We will often suppress $\alpha$ from the above notation in the rest of the paper.

\subsection{The moduli space of curves and sections}\label{subsec:The moduli space of curves and sections}

In this section, we introduce two parameter spaces of curves which will appear in the paper. We first introduce the morphism space.

Let $B$ be a smooth projective curve of genus $g$ and $X$ be a smooth projective variety. Let $\Mor(B, X)$ be the quasi-projective scheme parametrizing morphisms from $B$ to $X$. 

\begin{definition}
    Given $\alpha\in N_1(X)_{\BZ}$, we let $\Mor(B, X)_\alpha$ denote the subscheme of $\Mor(B,X)$ parametrizing morphisms such that $f_*B = \alpha$. 
\end{definition}

Let $[f]\in\Mor(B, X)$ be a morphism. By \cite[Section 2.11]{Deb01}, the dimension of the space $\Mor(B, X)$ at $[f]$ satisfies:
\[
    \dim_{[f]}\Mor(B, X) \geqslant -K_X\cdot f_*B + \dim(X)(1-g).
\]
We call the quantity $-K_X\cdot \alpha + \dim(X)(1-g)$ the \textit{expected dimension} of $\Mor(B, X)$ at $[f]$. When $X = S$ is a split quartic del Pezzo surface, we can express the expected dimension of $\Mor(B, S)_\alpha$ using the notation from the previous section:
\[
    \text{exp. dim}(\Mor(B, S)_\alpha) = 2a + 2a' - \sum_{i}k_i + 2(1-g) = h + 2(1-g).
\]
When $B=\BP^1$, the geometry of $\Mor(\BP^1,S)$ is well-studied:

\begin{theorem}[\cite{Tes09, BLRT23, DLTT25, LT23}]
    Let $S$ be a quartic del Pezzo surface over a field $k$. For every $\alpha \in \Nef_1(S)_{\BZ}$, the moduli space $\Mor(\BP^1,S)_\alpha$ is geometrically irreducible and smooth of the expected dimension.
\end{theorem}

In Section \ref{sec:Higher genus curves on del Pezzo surfaces}, we will show in Theorem \ref{thm:irreducibility of M_alpha} that a similar statement also holds for higher genus curves provided that $\alpha$ is sufficiently positive. We will frequently use $M_\alpha$ to abbreviate $\Mor(B,S)_\alpha$ for morphisms into our split quartic del Pezzo surface $S$.

Next, we introduce the space of sections. While Manin's conjecture concerns rational points on the morphism space $M_\alpha$, our actual counting takes place on the space of sections. Let $\BP^1\times\BP^1$ be identified with $\BP(V_1)\times\BP(V_2)$ for two-dimensional vector spaces $V_1$ and $V_2$. Let $a$, $a'$ be two non-negative integers and denote by $\CB = \Pic^a(B)\times\Pic^{a'}(B)$. Let $\CP_a$ and $\CP_{a'}$ be line bundles on $B\times \CB$ which are pullbacks of the Poincar\'e bundles on $B\times\Pic^a(B)$ and $B\times\Pic^{a'}(B)$ respectively and set
\[
\CV_1 = \CP_a \otimes V_1,\ \CV_2 = \CP_{a'}\otimes V_2,\ \CV = \CV_1\oplus \CV_2.
\]

\begin{definition}
    The space of sections $\Gamma(B, \CV)$ is the scheme over $\CB$ representing the functor
    \[
        (\phi\colon T\to\CB) \to H^0(B\times T, (\text{id}_B,\phi)^*\CV).
    \]
\end{definition}

Let $(p_i,p'_i) \in \BP^1\times\BP^1$ denote the image of the exceptional divisor $E_i$ under the morphism $\rho_\alpha:S\to \BP^1\times\BP^1$. Then $(p_i,p'_i)$ corresponds to a tuple of one-dimensional subspaces $l_{i,1}\oplus l_{i,2} \subset V_1\oplus V_2$.

\begin{definition}
    Let $\widetilde{M}_\alpha$ denote the locally closed subscheme of $\Gamma(B,\CV)$ such that the fibre over $(\CL_a, \CL_{a'})\in\CB$ parametrizes pairs of nowhere-vanishing sections 
    \[
        (s,t)\in H^0(B, (\CL_a\otimes V_1)\oplus (\CL_{a'}\otimes V_2))
    \]
    satisfying $\text{len}((s,t)^*(l_{i,1}\oplus l_{i,2})) = k_i$ for each $i$.
\end{definition}

In particular, we have a morphism $\widetilde{M}_\alpha \to M_\alpha$ which is a $\BG_m^2$-torsor.

\subsection{Poschemes}\label{subsec:Poschemes}

The incidence conditions arising from intersecting the exceptional curves of $S$ allow us to define a stratification on $\Gamma(B, \CV)$. This can be encoded scheme-theoretically by poschemes. We recall below only the part of this formalism needed in the subsequent construction; further details can be found in \cite[Sections 5 and 6]{DLTT25}. 

\begin{definition}
    Let $X$ be a scheme. A poscheme over $X$ is a morphism $P \to X$, locally of finite type, together with a closed poset relation $\leqslant_P \subset P\times_X P$. In other words, for any $T\to X$, the pair $(P(T), \leqslant_P\hspace{-1mm}(T))$ is a poset. 
\end{definition}

Let $Q$ be the poscheme given by the points corresponding to subspaces of $V = V_1\oplus V_2$:
\[
    \{V,\ V_1\oplus \{0\},\ \{0\}\oplus V_2,\ \{0\}\}\ \cup\ \{l_{i,1}\oplus\{0\},\ \{0\}\oplus l_{i,2},\ l_{i,1}\oplus l_{i,2}\}_{i=1}^4.
\]
In other words, each point of $Q$ corresponds to an element of the above set and the order is given by inclusions of vector spaces represented by the points.

For $d\geqslant0$, let $\Hilb^d(B)$ denote the Hilbert scheme of length-$d$ subschemes of $B$, and set
\[  
    \Hilb(B) =\coprod_{d\geqslant0}\Hilb^d(B), \text{  and  }\ \overline{\Hilb}(B) =\Hilb(B)\sqcup\{B\}.
\]
We order these schemes by containment of closed subschemes and declare $B$ to be the maximal element. Thus $\overline{\Hilb}(B)$ is a poscheme over $k$.

\begin{definition}
    Let $\overline{\Hilb}(B)^Q$ be the poscheme of homomorphisms from $Q\to \overline{\Hilb}(B)$ sending $V$ to $B$. We let $\Hilb(B)^Q$ denote the subposcheme parametrizing morphisms whose preimage of $B$ is $\{V\}$.
\end{definition}

An element $x = (x_q)_{q\in Q}$ of $\Hilb(B)^Q$ can be interpreted as the data of an intersection profile of some morphism $f:B \to \BP^1\times \BP^1$. For example, if $q = l_{1,1}\oplus l_{1,2}$, then $x_q$ records a finite subscheme $D_q$ of $B$ satisfying $f(D_q) = (p_1, p_1')$. In particular, let $\widetilde{Q} = Q \setminus \{V\}$. We may associate to $x$ a finitely supported function
\[
    \tilde{g}\colon B \to \Hom(\widetilde{Q}, \BN)
\]
that encodes the length of these subschemes. We define $\tilde{g}_c = \tilde{g}(c)\colon \widetilde{Q} \to \BN$ to be the poset homomorphism satisfying $\tilde{g}_c(q) = \text{len}_c(x_q)$ for $q\in \widetilde{Q}$. 

Among all elements of $\Hilb(B)^Q$, only a certain subset is geometrically meaningful under such interpretation. These are the saturated elements:

\begin{definition}[\cite{DLTT25} Definition 5.2]
    We say an element $(x_q)_{q\in Q}$ is saturated if for any subset $P \subset Q$, the natural containment of subschemes
    \[
    x_{\wedge_{p\in P} p} \subset \bigcap_{p\in P} x_p
    \]
    is an equality. We denote by $Q^{JB} \subset \Hilb(B)^Q$ the Zariski open subset of saturated elements.
\end{definition}

In particular, there is a saturation function $\text{sat}:\Hilb(B)^Q(\overline{k}) \to Q^{JB}(\overline{k})$ assigning to each $x$ the smallest saturated element greater than or equal to $x$. Given a saturated element $x = (x_q)$ and a point $c\in B$, we may define a poset homomorphism
\[
    g_c\colon Q \to \BN\cup\{\infty\}
\]
by extending $\tilde{g}_c$ to $Q$ via $g_c(V) = \infty$. To each $g_c$, we can then define its left adjoint $f_c\colon \BN\cup\{\infty\} \rightarrow Q$ to be 
\[
    f_c(n) = \min\{q\in Q\ |\ g_c(q) \geqslant n\}.
\]
This is again a poset homomorphism and we call it a \textit{chain} (see \cite[Definition 5.3]{DLTT25}). We denote the set of all chains by $\text{ch}(Q)$ and we say a chain $f$ is trivial if $f(0) = 0$ and $f(n) = V$ for all $n>0$. Given two chains $f, f'$, we say $f \leqslant f'$ holds if $f'(n) \leqslant f(n)$ as elements of $Q$ for all $n$. The reason to introduce chains is that later we will use them to stratify the space of sections.

\begin{definition}[\cite{DLTT25} Definition 5.15 and 5.16]
    Let $x\in \Hilb(B)^Q$ correspond to a finitely supported function
    \[
        \tilde{g}\colon B \to \Hom(\widetilde{Q}, \BN).
    \]
    The combinatorial type associated to $x$ is the multiset 
    \[
        \text{type}(x) = \{\tilde{g}_{c_1},\dots, \tilde{g}_{c_n}\},
    \]
    where $c_i\in\text{supp}(x)$. We say $\text{type}(x)$ is saturated if $x\in Q^{JB}$.
    For a fixed combinatorial type $T$, we let $\CN_T$ denote the Zariski locally closed subscheme of all $x\in\Hilb(B)^Q$ with $\text{type}(x) = T$.
\end{definition}

We remark that there is a bijection between saturated types and multisets of $\ch(Q)$, and we will freely use this correspondence later.

\begin{example}
    Let $T$ be a combinatorial type consisting of elements
    \[
        f^{w_1}, \dots, f^{w_t}
    \]
    with each $f^{w_i}$ occurring with multiplicity $m_i$. Then there is an isomorphism
    \[
        \CN_T \cong (\Conf_{\sum_i m_i} B)/(S_{m_1}\times \dots \times S_{m_t}).
    \]
\end{example}

The advantage of passing to combinatorial types is that, once the local incidence data is fixed, the remaining parameters are the positions of the support points on $B$. Thus the geometry of a stratum $\CN_T$ is governed by a configuration space of points on $B$, up to the finite permutation groups appearing in the preceding example. This description will be used in Section \ref{sec:The Virtual Bar Complex}, where the strata $\CN_T$ are assembled into the virtual bar complex and their cohomology enters the Frobenius trace calculation.

In our application, we will usually start with a prescribed intersection profile $w$ and consider additional incidence conditions $w<x$. It is therefore convenient to work with the following relative version. Let $W\subset \Hilb(B)^Q$ be a locally closed subposcheme. We define $(W < \Hilb(B)^Q)$ to be the poscheme over $W$ with points
\[
    (W < \Hilb(B)^Q) = \{(w,x)\in W\times \Hilb(B)^Q\ |\ w < x\}
\]
and the poset structure:
\[
    \leqslant_{(W<\Hilb(B)^Q)} = \{(w < x_1, w < x_2)\ |\ x_1\leqslant x_2\}.
\]
We define $(W \leqslant \Hilb(B)^Q)$ similarly. The previous construction carries naturally to this relative setting, and we again refer the reader to \cite[Section 5]{DLTT25} for more generalities.

We finally recall the notion of an essential type. For $f\in\ch(Q)$, let
\[
    \operatorname{Cov}(f) = \{f'\in\ch(Q)\mid f\prec f'\},
\]
where $f\prec f'$ means that there is no chain $h$ satisfying $f<h<f'$. A pair $f\leqslant g$ is called \textit{essential} if $g$ is the join of a subset of $\operatorname{Cov}(f)$. Essential types are precisely the combinatorial types that contribute to the cohomology computation in Section \ref{sec:The Virtual Bar Complex}.

\begin{definition}[\cite{DLTT25} Definition 5.7]
    Let $W\subset Q^{JB}$ be a locally closed reduced subscheme. We say $(w\leqslant x) \in (W \leqslant Q^{JB})$ is essential if for every $c\in \text{supp}(w\leqslant x)$, the pair $f^w_c \leqslant f^x_c$ is essential. 
\end{definition}

\subsection{Stratifying the space of sections}\label{subsec:Stratifying the space of sections}

We now use the poscheme constructed above to organize the incidence conditions on the space of sections. 

Given an element $q\in Q$, we let $\CK_q\subset \CV$ denote the corresponding subbundle over $B\times \CB$.

\begin{prop}[Section 6.3 \cite{DLTT25}]
    There is a closed subscheme $\Gamma_Q(B, \CV) \subset \Gamma(B, \CV) \times \Hilb(B)^Q$ whose fibre over a point $\mathfrak{b} \in \CB$ parametrizes pairs of sections 
    \[
        (s, (x_q)_{q\in Q}) \in \Gamma(B, \CV_{\mathfrak{b}}) \times\Hilb(B)^Q
    \]
    satisfying $s(x_q) \subset (\CK_q)_{\fb}$ for every $q\in Q$.
\end{prop}

These conditions define linear subspaces $\Gamma_Q(B,\CV)_x \subseteq \Gamma(B,\CV)$, and they are compatible with the order on $\Hilb(B)^Q$: if $x\leqslant y$, then
\[
    \Gamma_Q(B,\CV)_y \subseteq \Gamma_Q(B,\CV)_x.
\]
Moreover, the resulting linear subspace depends only on the saturation of $x$ in the sense that $\Gamma_Q(B, \CV)_{\fb, x} = \Gamma_Q(B, \CV)_{\fb, \text{sat}(x)}$.

\begin{example}
    Fix $\mathfrak{b} = (\CL, \CL')$ and let $s = (s_1, s_2)\in H^0(B,\CV_{\fb})$.

    If $q = l_{i,1}\oplus l_{i,2}$, then the condition $s(x_q) \subset (\CK_q)_{\fb}$ means that the sections $s_1$ and $s_2$ restricted to the subscheme $x_q \subset B$ lie in $\CL\otimes l_{i,1}$ and $\CL'\otimes l_{i,2}$ respectively.

    If $q = l_{i,1}\oplus \{0\}$, then the section $s_2$ restricted to $x_q$ lies in the zero subspace of $V_2$. In this case, if $x_q$ represents a non-empty subscheme of $B$, then $(s, (x_q))$ does not correspond to an honest morphism $B\to \BP^1\times\BP^1$.
\end{example}

Let $U_{\bk} = \bigcup_T \CN_T$ be the sublocus of $\Hilb(B)^Q$ as $T$ varies over saturated combinatorial types of the form:
\[
    T = \left\{m_{ij}[l_{i,1}\oplus l_{i,2}]\ \Big|\ \sum_{j=1}^{t_i} m_{ij} = k_i \right\}_{\substack{i = 1,\dots, 4 \\ j = 1, \dots,t_i}}
\]
Here, the chain $m_{ij}[l_{i,1}\oplus l_{i,2}]$ is defined as:
\[
    m_{ij}[l_{i,1}\oplus l_{i,2}](n) = 
    \begin{cases}
        0 \hspace{15mm} & \text{if } n = 0\\
        [l_{i,1} \oplus l_{i,2}] & \text{if } 1\leqslant n \leqslant m_{ij}\\
        V & \text{if } n>m_{ij}
    \end{cases}
\]
Geometrically, such a $T$ corresponds to a morphism $f\colon B\to S$ for which the intersection with the exceptional curve $E_i$ is supported at $t_i$ geometric points. The multiplicity of the intersection at the $j$-th point in the support is $m_{ij}$.

Let $E$ be the space defined via the fibre product:
\[\begin{tikzcd}
	E & {\Gamma_Q(B,\CV)} \\
	{\CB \times U_{\bk}} & {\CB\times \Hilb(B)^Q}
	\arrow[from=1-1, to=1-2]
	\arrow[from=1-1, to=2-1]
	\arrow[from=1-2, to=2-2]
	\arrow[from=2-1, to=2-2]
\end{tikzcd}\]
and set $E_w = \Gamma_Q(B, \CV)_w$. Let
\[
    P_{\infty}:=(U_{\mathbf k}<\Hilb(B)^Q)
\]
viewed as a poscheme over $U_{\mathbf k}$ via $(w<x)\mapsto w$. We define the universal incidence family
\[
    Z^{\alg} \subset (\CB\times P_{\infty}) \times_{\CB\times U_{\bk}}E
\]
by declaring its fibre over $(\fb,w<x)$ to be
\[
    Z^{\alg}_{\fb,w<x}
    := \Gamma_Q(B, \CV)_{\fb,x}
    \subseteq E_{\fb,w}.
\]
Here the inclusion follows from $w<x$ and the order-reversing property of the incidence conditions. More generally, if $w<x\leqslant y$, then
\[
    Z^{\alg}_{\fb,w<y} \subseteq Z^{\alg}_{\fb,w<x}.
\]
Thus $Z^{\alg}$ is a poscheme-indexed space of linear incidence subspaces of $E$.

Observe that if $w \in U_{\bk}$ is an element with $\text{type}(w) = T$, then any $x \in \Hilb(B)^Q$ such that $w < x$ will correspond to some rational map $B\to \BP^1\times \BP^1$ with non-empty basepoints. This gives the following decomposition:

\begin{prop}[Section 6.3 \cite{DLTT25}]\label{prop:decomp of M_alpha}
    The moduli space $\widetilde{M}_{\alpha}$ satisfies
    \[
        \widetilde{M}_\alpha = \bigcup_{(\fb,w)\in \CB\times U_{\bk}} (E_{\fb,w} \setminus \cup_{(\fb, w<x) \in \CB \times (U_{\bk} < \Hilb(B)^Q)} Z^{\alg}_{\fb, w<x}).
    \]
\end{prop}

We also record a few combinatorial functions from $Q^{JB}$ to $\BN$ which are used to compute various expected dimensions of subspaces of $E$. Since they mostly come up in the counting argument in Section \ref{sec:Manin's conjecture for higher genus curve}, and are less essential to the overall paper, we refer interested readers to \cite[Section 5.3]{DLTT25} for a more detailed discussion. 

Let $x$ be an element of $Q^{JB}$. The \textit{multiplicity} function $m_q$ for $q\neq V$ is given by the total multiplicities of points in $B$ that land in $q$.  The \textit{rank} function $\rank(x)$ is defined to be the cohomological dimension of the simplicial scheme $N([V,x]\cap Q^{JB})$. The \textit{support} function $|\Supp(x)|$ is set to be the number of $c\in B$ such that $f_c$ is non-trivial. The \textit{incidence} function $\gamma(x)$ is given by
\[
    \gamma(x) = 4m_0(x) + 3\sum_{i = 1}^4\sum_{j = 1}^2 m_{l_{i,j}}(x) + 2\sum_{j = 1}^2 m_{V_j}(x) + 2\sum_{i = 1}^4 m_{l_{i,1}\oplus l_{i,2}}(x),
\]
and measures the expected codimension of the space $\Gamma_Q(B,\CV)_x$. Finally, we set 
\[
    \kappa(x) = 2\gamma(x) - \rank(x) - 2|\Supp(x)|.
\]
Given such a combinatorial function $h$, we extend $h$ to $(W\leqslant \Hilb(B)^Q)$ by defining $h(w<x)$ to be $h(x) - h(w)$. Moreover, we may extend $h$ to the set of all combinatorial types as $h$ takes the same value for elements belonging to the same combinatorial type.

\subsection{Homology of configuration spaces and bounds on traces of endomorphisms} \label{subsec:Homology of configuration spaces}

The following lemma is the higher genus analogue of \cite[Proposition 2.10]{DLTT25}, which is the main result of \cite[Section 2.2]{DLTT25}. This plays an important role in the proof of Theorem \ref{thm:bounding truncated virtual bar complex}. To make sense of this, we first need the following definition.

\begin{definition}
    Let $K$ be a field with a norm on its algebraic closure: $|\cdot |\colon \overline{K} \to \BR$. For a finite-dimensional $K$-vector space $V$ equipped with a linear operator $F$, let the $L^1$-trace of $F$ on $V$ be 
    \[
        |F, V|\coloneqq \sum_{\lambda \in \overline{K}}\dim \left((V\otimes \overline{K})_\lambda\right)|\lambda|,
    \] 
    where $(V\otimes \overline{K})_\lambda$ is the generalized $\lambda$-eigenspace of $F$.
\end{definition}

In particular, the $L^1$-trace gives an upper bound on the actual trace: $|\trace(F\curvearrowright V)| \leqslant |F,V|.$

\begin{lemma}\label{L1_trace_invariants}
    Suppose the ground field $k$ is $\BF_q$. Let $V$ be a bigraded vector space with an action of a linear operator $F$ that is concentrated in homological degrees $\geqslant 2$ and in grading degree $\geqslant 1$. Also, let $B$ be a smooth projective curve of genus $g$ over $\BF_q$ and $\Conf_m(B)$ the ordered configuration space of $m$ points on $B$. 

    Assume $F$ acts on $H^*_c\left(\Conf_m(B),\BQ_\ell) \right)$ as the Frobenius. Suppose that $|F,V_{\bullet,2}|\leqslant D$. If $|F,V_{n,i}|\leqslant Ed^nb^i$ for some $E,d,b>0$ with $E\geqslant 1$ and $b<1$, then $W\coloneqq \bigoplus_m H^*_c\left(\Conf_m(B),\BQ_\ell\right)\otimes_{S_m}V^{\otimes m}$ is concentrated in homological degrees $\geqslant 0$ and 
    \[
        |F,(W_{n,i})_{S_n}|=O\left(ni\max\left(2Ed+8Edbgq^{1/2}+8Edb^2q,D\right)^n(4b)^i\right).
    \]
\end{lemma}

\begin{proof} The proof is the same as that of \cite[Proposition 2.10]{DLTT25}, so we only point out some of the differences. First, \cite[Theorem 2.7]{DLTT25} remains correct as stated without any modifications when $\BP^1$ is replaced with $B$. This implies that \cite[Proposition 2.9]{DLTT25} with $\BP^1$ replaced by $B$ also follows directly. Then, the only modification to the proof is in the third-to-last line, where we now instead have 
\[
    \left|F,\left(\bigoplus_{m\ge1}H^*_c(B_{\overline{k}},\BQ_\ell)[1]\otimes V[-1]^{\otimes m}\right)_{n,i}\right|\leqslant \left(\frac{1}{2b}+2gq^{1/2}+2bq\right)(2Ebd)^n(2b)^i.
\] 
To see this, recall that $h^0(B_{\overline{k}},\BQ_\ell)=1, h^1(B_{\overline{k}},\BQ_\ell)=2g$, and $h^2(B_{\overline{k}},\BQ_\ell)=1$. Also, the Weil conjectures give the corresponding magnitudes of Frobenius eigenvalues for these cohomology groups. 

Finally, setting $E$ to be $\frac{1}{2b}+2gq^{1/2}+2bq$, $d$ to be $2Ebd$, and $b$ to be $2b$ in \cite[Lemma 2.6]{DLTT25}, it follows that \[|F,(W_{n,i})_{S_n}|=O\left(ni\max\left(2Ed+8Edbgq^{1/2}+8Edb^2q,D\right)^n(4b)^i\right).\] 
\end{proof}

\section{Higher genus curves on del Pezzo surfaces}\label{sec:Higher genus curves on del Pezzo surfaces}

Let $k = \BF_q$ be a finite field not equal to $\BF_2$. We denote by $B$ a smooth projective curve of genus $g$ and by $S$ a split quartic del Pezzo surface over $k$. Given a numerical class $\alpha \in \Nef_1(S)_{\BZ}$, we let $M_\alpha \coloneqq \Mor(B, S)_{\alpha}$. Recall from Section \ref{subsec:quartic del pezzo surfaces} that the numerical class $\alpha$ corresponds to a birational morphism 
\[
    \rho_\alpha\colon S \to \BP^1\times \BP^1
\]
contracting disjoint $(-1)$-curves $E_1, E_2, E_3, E_4$. The morphism $p_\alpha\colon S\to \BP^1$ realizes $S$ as a conic bundle over $\BP^1$, and this induces a morphism 
\[
    p_*\colon M_\alpha\to N_a \subset \Mor(B, \BP^1),
\]
where $N_a$ is the irreducible component parametrizing degree $a$ covers. Symmetrically, $p'_\alpha$ induces a morphism $p'_{*}\colon M_\alpha \to N_{a'}$. Our goal in this section is to prove the irreducibility of $M_\alpha$ using the conic bundle structure of $S$. 

From now on, we assume throughout the section that
\begin{equation*}\label{eq:bounds on a and a'}
    a, a' > \max_i\{k_i\}  + 5g. \tag{$\star$}
\end{equation*}

\subsection{Geometry of configuration cover}

Let $\CP$ denote the universal line bundle over $B \times \Pic^a(B)$. Let $\CE = (\pi_*\CP)^{\oplus 2}$, where $\pi$ is the projection from $B \times \Pic^a(B)$ to $\Pic^a(B)$. Then the space $N_a \subset \Mor(B,\BP^1)$ parametrizing degree $a$ covers can be identified as an open subset of a $\BP^{2a+1-2g}$-bundle $\BP_{\Pic^a(B)}(\CE)$ over $\Pic^a(B)$. We denote this projective bundle by $\BP_{a}$.

\begin{definition}\cite[Definition 4.1]{DLTT25}
    For $\bk = (k_1,k_2,k_3,k_4)$ as defined in Section \ref{subsec:quartic del pezzo surfaces}, we let the closed subscheme 
    \[
        Z_{a,\bk} \subset \BP_a \times \Hilb^{[k_1]}(B_1) \times \dots \times \Hilb^{[k_4]}(B_4)
    \]
    denote the incidence correspondence
    \[
        Z_{a,\bk} = \{(C, [T_1],\dots,[T_4])\ |\ [C] \in \BP_a, T_i\subset C\cap B_i\},
    \]
    where $B_i = B\times\{p_i\}$. We call $Z_{a,\bk}$ the \textit{configuration cover}.
\end{definition}

Define the Zariski open subset $U_{\bk} \subset \prod_{i = 1}^4 \Hilb^{[k_i]}(B_i)$ as 
\[
    U_{\bk} = \{([T_1],\dots,[T_4]) \ |\ \Supp(T_i)\cap \Supp(T_j) = \emptyset, i\neq j\},
\]
and set $W_{a, \bk} = \Pic^a(B) \times U_{\bk}$. Notice that there is a bijection between the set of degree $a$ morphisms $B\to \BP^1$ and an open subset of sections of the linear series $|\pi_1^*\CL_a\otimes\pi_2^*\CO_{\BP^1}(1)|$ on $B\times\BP^1$, where $\CL_a \in \Pic^a(B)$. We denote by $Z^\circ_{a,\bk}$ the preimage of $N_a$ under the projection map $Z_{a,\bk} \to \BP_a$. Consider the projections
\[
    \phi_{a,\bk}\colon Z_{a,\bk} \to \Pic^a(B) \times \prod_{i = 1}^4 \Hilb^{[k_i]}(B_i)
\]
and $\psi_{a,\bk}\colon Z_{a, \bk} \to \prod_{i = 1}^4 \Hilb^{[k_i]}(B_i)$. We let $\phi^\circ_{a,\bk}$ and $\psi^\circ_{a,\bk}$ denote their restrictions to $Z^\circ_{a,\bk}$.

\begin{prop} \label{prop:equidimension and smooth} 
    The space $Z^\circ_{a,\bk}$ is equidimensional of dimension $2a - g + 1$. If furthermore $2a - \sum_i k_i > 2g - 2$, then the map $\psi^\circ_{a,\bk}\colon Z^\circ_{a,\bk} \to U_{\bk}$ is smooth.
\end{prop}

\begin{proof}
    The equidimensionality of $Z^\circ_{a,\bk}$ follows directly from applying the proof of \cite[Proposition 4.2]{DLTT25} and \cite[Corollary 4.3]{DLTT25} to the higher genus setting.
    
    To show that $\psi^\circ_{a,\bk}$ is smooth, by \cite[Lemma 2.1]{DLTT25}, it suffices to prove that the fibre of $\psi_{a,\bk}$ is smooth and equidimensional. Let $([T_i]) \in U_{\bk}$ be a point in the image of $\psi^\circ_{a,\bk}$ and $\sigma\colon B \to B\times\BP^1$ be a map in the fibre $(\psi^\circ_{a,\bk})^{-1}([T_i])$. Then the deformation theory of the fibre is controlled by the twisted normal bundle $N_\sigma(-\sum_i T_i)$. Since $\deg N_\sigma(-\sum_i T_i) = 2a - \sum_i k_i > 2g - 2$, the obstruction vanishes and we see that the fibre $(\psi^\circ_{a,\bk})^{-1}([T_i])$ has dimension $2a - \sum_i k_i - g + 1$.
\end{proof}

Now, we study the fibre of the map $\phi_{a,\bk}$. Let $(\CL, [T_i]) \in W_{a,\bk}$ be a point. The fibre of $\phi_{a,\bk}$ above $(\CL,[T_i])$ can be identified as the set of divisors in the linear series $|\pi_1^*\CL\otimes \pi_2^*\CO_{\BP^1}(1)|$ passing through the $T_i$'s. Let $T = \sum_i T_i$ and $\CF_T$ be the elementary modification along $T$, i.e. the kernel of the short exact sequence
\[
    0 \to \CF_T \to \CO_B^{\oplus 2} \to \CO_T \to 0.
\]
Suppose there exists an honest section $\sigma \in \phi_{a,\bk}^{-1}(\CL, [T_i])$. Pulling back the relative Euler exact sequence along $\sigma$ gives
\[
    0 \to \CO_B \to \CL^{\oplus 2} \to N_{\sigma} \to 0.
\]
Then we may identify $\CF_T\otimes \CL = \ker(\CL^{\oplus 2} \to N_{\sigma}/N_{\sigma}(-T))$. In particular, the nonvanishing sections of $H^0(B, \CF_T\otimes \CL)$ correspond to sections passing through the $T_i$'s. 

The following lemma, due to Tanimoto, is the key input to Proposition \ref{prop:dimboundWak} later.

\begin{lemma}\label{lem:min slope of elementary modification}
    Suppose $2a - \sum_i k_i \geqslant 0$ and $\phi_{a,\bk}^{-1}(\CL, [T_i])$ contains an honest section $\sigma$. Then $\mu_{\min}(B,\CF_T\otimes \CL) \geqslant 0$.
\end{lemma}

\begin{proof}
    Since the degree of $\CF_T\otimes \CL$ is equal to $2a - \sum_i k_i$, the claim is immediate if $\CF_T\otimes \CL$ is semi-stable. Hence we assume that $\CF_T\otimes \CL$ is unstable. Let $\CF_1$ be the maximal destabilizing subsheaf and $\CF_2$ be the quotient $\CF_T\otimes \CL/\CF_1$. Then $\mu_{\min}(\CF_T\otimes \CL) = \deg \CF_2$. Consider the following short exact sequence
    \[
        0 \to \CO_B \to \CF_T\otimes \CL \to N_{\sigma}(-T) \to 0.
    \]
    If the map $\CF_1 \to N_{\sigma}(-T)$ is the zero map, then $\CF_1$ factors through $\CO_B$ and hence has non-positive degree. This contradicts it being destabilizing. Hence $\CF_1 \to N_{\sigma}(-T)$ is a non-zero map, i.e. $\deg \CF_1 \leqslant 2a - \sum_i k_i$. Thus $\deg \CF_2 \geqslant 0$.
\end{proof}

Next, we measure the dimension of the locus of $(\CL, [T_i])$ in $W_{a,\bk}$ such that the fibre of $\phi_{a,\bk}$ above consists of reducible divisors. To do this, we first study the map $\psi_{a,\bk}$. The following lemma follows the steps of \cite[Lemma 4.11]{DLTT25}:

\begin{lemma}\label{lem:dimboundUak} 
    Suppose $2a - \sum_i k_i > 2g - 2$. Then the complement of the image of $Z^\circ_{a,\bk}$ in $U_{\bk}$ has codimension at least:
    \[
        2a - \sum_i k_i - 4g + 1.
    \]
\end{lemma}

\begin{proof}
    Let $([T_i]_i)\in U_{\bk}$ be an element not in the image of $\psi^\circ_{a,\bk}$. Then the fibre $\psi_{a,\bk}^{-1}([T_i])$ can be identified as
    \[
        \psi_{a,\bk}^{-1}([T_i]) = \bigcup_{\CL \in \Pic^a(B)} \BP(H^0(B, \CF_T\otimes \CL)),
    \]
    where $\BP(H^0(B, \CF_T\otimes \CL))$ parametrizes reducible divisors. Let $D$ be a general divisor contained in this fibre. Let $l_1$ be the number of $\pi_1$-vertical components of $D$ which do not intersect the $T_i$'s and let $l_2$ be the number of $\pi_1$-vertical components of $D$ which intersect some of the $T_i$'s. Then the $\pi_1$-horizontal component of $D$ is an irreducible curve $D'$ going through the rest of the $T_i$'s and whose projection under $\pi_2$ has degree $a-l_1-l_2$. The choice of $D'$ is given by the divisors in the linear series $|\pi_1^*\CL'\otimes\pi_2^*\CO_{\BP^1}(1)|$ that go through $\sum_i k_i - l_2$ many points, where $\CL'$ is a line bundle of degree $a - l_1 - l_2$. 

    Suppose first that the relative anticanonical degree of $D'$ is not zero so that $D'$ is not contracted under the other projection. To compute the choice of such $D'$, notice that the deformation of a degree $(a-l_1-l_2)$-section through $\sum_i k_i - l_2$ points is given by the twisted normal bundle $N_{D'/B\times\BP^1}(-T')$, where $T'$ indicates the points on $B$ meeting the prescribed points. If $\deg N_{D'/B\times\BP^1}(-T') \leqslant 2g - 2$, then by Clifford's theorem, we have $h^0(B, N_{D'/B\times\BP^1}(-T')) \leqslant g$. If $\deg N_{D'/B\times\BP^1}(-T') > 2g - 2$, then by Riemann--Roch, we have $h^0(B, N_{D'/B\times\BP^1}(-T')) = 2a - 2l_1 - l_2 - \sum_i k_i - g + 1$. Hence
    \[
        h^0(B, N_{D'/B\times\BP^1}(-T')) \leqslant \max\{\ 2a - 2l_1 - l_2 - \sum_i k_i - g + 1, g\}.
    \]
    On the other hand, there are finitely many choices for the $l_2$ vertical components which intersect the $T_i$'s and $l_1$-dimensional freedom for the vertical components which miss the $T_i$'s, so the dimension of $\psi_{a,\bk}^{-1}([T_i])$ is at most
    \[
        \max\{\ 2a - 2l_1 - l_2 - \sum_i k_i - g + 1, g\} + l_1.
    \]
    Since the dimension of $\psi_{a,\bk}^{-1}([T_i])$ is at least $2a - \sum_i k_i - g + 1$ by Proposition \ref{prop:equidimension and smooth}, we obtain
    \[
        \max\{\ 2a - 2l_1 - l_2 - \sum_i k_i -g + 1, g\} + l_1 \geqslant 2a - \sum_i k_i - g + 1.
    \]
    If $2a - 2l_1 - l_2 - \sum_i k_i -g + 1 > g$, then we have
    \[
        2a - l_1 - l_2 - \sum_i k_i -g + 1 \geqslant 2a - \sum_i k_i - g + 1,
    \]
    which implies that 
    \[
        l_1 + l_2 \leqslant 0.
    \]
    But this is a contradiction to our assumption. Hence we have
    \[
        l_1 \geqslant 2a - \sum_i k_i - 2g + 1.
    \]

    If the relative anticanonical degree of $D'$ is zero, then either $D' = B_j$ for some $j$, or $D'$ is not equal to $B_i$ for all $i$. The dimension of such locus is at most
    \[
        a - (\sum_i k_i - k_j) + g.
    \]
    Since, by assumption, $a - \max_i\{k_i\} > 5g$, we have 
    \[
        2a - \sum_i k_i - 2g + 1 - (a - (\sum_i k_i - k_j) + g) > 0. 
    \]
    This shows that such locus cannot be an irreducible component of a fibre.
    
    Now, we may estimate the dimension of $U_{\bk}\setminus \psi_{a,\bk}(Z^\circ_{a,\bk})$ as follows. Then the space of divisors with $l_1+l_2$ number of $\pi_1$-vertical components in $\BP_{\Pic^a(B)}(\CE)$ has dimension at most
    \[
        (\deg(N_{D'/B\times\BP^1}) + 1) + l_1 + l_2 = 2(a - l_1 - l_2) + 1 + l_1 + l_2 = 2a - l_1 - l_2 + 1
    \]
    by \cite[Lemma 2.10]{LRT26}. Since the fibre of $\phi_{a,\bk}$ over $(\CL, [T_i])$ has dimension at least $l_1$ by assumption, we have
    \[
        \dim_{[T_i]} (U_{\bk}\setminus \image(\psi^\circ_{a,\bk})) \leqslant 2\sum_i k_i - 2a + 4g - 1.
    \]

\end{proof}

We then study the complement of $Z^{\circ}_{a,\bk}$ in $W_{a,\bk}$. 

\begin{prop}\label{prop:dimboundWak}
    Suppose that $2a - \sum_i k_i > 6g - 2$. Then the image of $\phi^\circ_{a,\bk}$ contains $\Pic^a(B) \times (U_{\bk}\setminus F)$, where $F$ is some closed subset of codimension at least:
    \[
       2a - \sum_i k_i - 8g + 1.
    \]
\end{prop}

\begin{proof} 
    First, let's consider the configuration cover $Z_{a-2g,\bk}$. Let $[T_i]$ be a point in the image of $\psi^\circ_{a - 2g,\bk}$. We claim that for any $\CL \in \Pic^a(B)$, the space $\BP(H^0(B, \CF_T\otimes \CL))$ generically parametrizes irreducible sections. Indeed, since the fibre of $\psi_{a-2g,
    \bk}$ above $[T_i]$ contains an honest section $\sigma$ in $\BP(H^0(B, \CF_T\otimes\CL'))$ for some $\CL'\in \Pic^{a-2g}(B)$, Lemma \ref{lem:min slope of elementary modification} implies that $\mu_{\min}(\CF_T\otimes\CL') \geqslant 0$. This implies that for any $\CL \in \Pic^a(B)$, $\mu_{\min}(\CF_T\otimes\CL) \geqslant 2g$. Hence \cite[Corollary 2.9]{LRT26} shows that $\CF_T\otimes\CL$ is globally generated for any $\CL \in \Pic^a(B)$, so $\BP(H^0(B, \CF_T\otimes \CL))$ generically parametrizes irreducible sections. Finally, we may take $F = U_{\bk}\setminus \image(\psi^\circ_{a-2g,\bk})$, and the dimension bound follows from Lemma \ref{lem:dimboundUak}.
\end{proof}

\begin{prop} \label{prop:configuration cover is smooth and irred}
    Suppose that $2a - \sum_i k_i > 6g - 2$. Then the space $Z_{a,\bk}^\circ$ restricted to its fibre over $\Pic^a(B) \times (U_{\bk}\setminus F)$ is smooth and geometrically irreducible of dimension $2a - g + 1$.
\end{prop}

\begin{proof}
    The smoothness of $Z^\circ_{a,\bk}$ is established in Proposition \ref{prop:equidimension and smooth}, and we are left to prove the claim on irreducibility. Let $F$ be the set in Proposition \ref{prop:dimboundWak}. Then given $(\CL,[T_i]) \in \Pic^a(B) \times (U_{\bk}\setminus F)$, the fibre $\phi^{-1}(\CL,[T_i]) $ is an open subspace of $\BP(H^0(B, \CF_T\otimes \CL))$. Since $\CF_T\otimes \CL$ is globally generated, every fibre is irreducible of the same dimension. Then Proposition \ref{prop:equidimension and smooth} implies that $Z^\circ_{a,
    \bk}$ is irreducible over $\Pic^a(B) \times (U_{\bk}\setminus F)$.
\end{proof}

\subsection{Irreducibility of the space of higher genus curves}

Given the configuration cover $Z_{a,\bk}$, there is a morphism 
\[
    \psi_\alpha\colon M_\alpha \to Z^\circ_{a,\bk}
\]
sending
\[
    [s:B\to S] \mapsto ([p_\alpha\circ s], [s^*E_1],\dots, [s^*E_4])
\]

\begin{prop} \label{prop:Malphasmooth} 
    Suppose $2a - \sum_{i=1}^4 k_i > 2g - 2$ and $2a' - \sum_{i=1}^4 k_i > 2g - 2$. Then the morphism 
    \[
    \psi_\alpha\colon M_\alpha \to Z_{a,\bk}^\circ
    \] is smooth. In particular, $M_\alpha$ is smooth.
\end{prop} 
\begin{proof}
    By \cite[Lemma 2.1]{DLTT25}, it suffices to show that $M_\alpha$ is equidimensional and that the fibre of $\psi_\alpha$ is smooth of the expected dimension. We first consider $M_\alpha \to U_{\bk}$. For each $\{[T_i]\}_i \in U_{\bk}$, the fibre can be identified as the set of morphisms:
    \[
    \{s\colon B\to S \ |\  s(T_i) \subset E_i\}.
    \]
    The tangent space and obstruction space to a point $s:B\to S$ in the fibre are determined by the pullback of log tangent bundle
    \[
    s^*T_S(-\log(\sum_{i=1}^4 E_i)).
    \]
    By definition, the log tangent bundle is the subsheaf of the tangent bundle consisting of vector fields tangent to $E_i$, and it fits into the following exact sequence
    \[
    0 \to T_S(-\sum_{i=1}^4 E_i) \to T_S(-\log(\sum_{i=1}^4 E_i)) \to \bigoplus_{i=1}^4 T_{E_i}  \to 0.
    \]
    On the other hand, we have the exact sequence coming from the birational morphism $\rho:S \to \BP^1\times\BP^1$:
    \[
    0 \to T_S \to \rho^*T_{\BP^1\times\BP^1} \to \bigoplus_{i=1}^4 T_{E_i}(E_i) \to 0
    \]
    Note that $T_{E_i}(E_i) \cong \CO_{E_i}(2 - 1) = \CO_{E_i}(1)$. Tensoring this exact sequence by $\CO_S(-\sum_{i=1}^4 E_i)$ gives
    \[
    0 \to T_S(-\sum_{i=1}^4 E_i) \to \rho^*T_{\BP^1\times\BP^1}(-\sum_{i=1}^4 E_i) \to \bigoplus_{i=1}^4 T_{E_i} \to 0
    \]
    In particular, we have the isomorphism 
    \[
    s^*T_S(-\log(\sum_{i=1}^4 E_i)) \cong s^*\rho^*T_{\BP^1\times\BP^1}(-\sum_{i=1}^4 E_i) \cong \CL \oplus \CL',
    \]
    where $\CL$ (resp. $\CL'$) is a line bundle of degree $2a - \sum_i k_i$ (resp. $2a' - \sum_i k_i$). The hypothesis in the proposition shows that the obstruction $h^1(B, s^*T_S(-\log(\sum_{i=1}^4 E_i))) = 0$, so the fibre has dimension $2a + 2a' - 2\sum_i k_i - 2g + 2$. Hence the dimension of $M_\alpha$ satisfies 
    \[
        \dim M_\alpha \leqslant \dim U_{\bk} + 2a + 2a' - 2\sum_i k_i - 2g + 2.
    \]
    Since the right-hand side is also the expected dimension, we conclude that $M_\alpha$ is equidimensional. This shows that $M_\alpha \to U_{\bk}$ is smooth.
    
    Next, we consider $Z_{a,\bk}^\circ \to U_{\bk}$. The deformation theory of the fibre is controlled by the sheaf $s^*p^*T_{\BP^1}(-\sum_{i=1}^4 T_i)$. Hence, we see the deformation theory to the fibre of $\psi_\alpha$ is given by the kernel of the map
    \[
    s^*T_S(-\log(\sum_{i=1}^4 E_i)) \to s^*p^*T_{\BP^1}(-\sum_{i=1}^4 T_i).
    \]
    Since this map is the projection of $s^*T_S(-\log(\sum_{i=1}^4 E_i))$ to the first factor, we see the kernel is precisely $\CL'$, which has vanishing $h^1$. Then the conclusion follows from \cite[Lemma 2.1]{DLTT25}.
\end{proof}

Let $R_\alpha$ denote the image of the morphism $M_\alpha \to Z_{a,\bk}^\circ$. Since this morphism is flat, $R_\alpha$ is open. Let $\pi_\alpha\colon R_\alpha \to N_a$ denote the restriction of the projection $Z_{a,\bk}^\circ \to N_a \subset \BP_a$. We have a Cartesian diagram
\[\begin{tikzcd}
	Q_\alpha = R_\alpha \times_{U_{\bk}} Z^\circ_{a',\bk} && Z^\circ_{a',\bk} \\
	R_\alpha && U_{\bk}
	\arrow[from=1-1, to=2-1]
	\arrow[from=1-1, to=1-3]
	\arrow[from=2-1, to=2-3]
	\arrow[from=1-3, to=2-3]
\end{tikzcd}\]
and an induced morphism $i_\alpha\colon M_\alpha \to Q_\alpha$.

Using Proposition \ref{prop:configuration cover is smooth and irred}, we define $U^\circ = \image(\psi^\circ_{a-2g,\bk}) \cap \image(\psi^\circ_{a'-2g,\bk}) \subset U_{\bk}$. Set $R^\circ_\alpha$ to be the base change of $R_\alpha$ along $U^\circ \to U_{\bk}$ and $Q^\circ_\alpha = Q_\alpha \times_{U_{\bk}} U^\circ$.

\begin{prop} \label{prop:structure theorem for M_alpha}
    Suppose that $2a - \sum_i k_i > 6g - 2$ and $2a' - \sum_i k_i > 6g - 2$. The morphism $i_\alpha$ is an open immersion. Moreover, the space $(M_\alpha)|_{U^\circ}$ is an open subset of a $\BP^{h-2a-2g+1}$-bundle over $R^\circ_\alpha \times \Pic^{a'}(B)$, and hence is irreducible.
\end{prop}

\begin{proof}
    We first show $i_\alpha$ is an open immersion. Notice that since $M_\alpha$ and $Q_\alpha$ are smooth of the same dimension, it suffices to show that $i_\alpha$ is \'etale and injective. The tangent space to the fibre of $M_\alpha \to U_{\bk}$ is 
    \[
        H^0(B, s^*T_S(-\log(\sum_{i=1}^4 E_i))) \cong H^0(B, \CL) \oplus H^0(B, \CL')
    \]
    as in Proposition \ref{prop:Malphasmooth}, where $\CL \cong s^*p^*T_{\BP^1}(-\sum_{i=1}^4 T_i)$ and $\CL' \cong s^*p'^*T_{\BP^1}(-\sum_{i=1}^4 T_i)$. Since $H^0(B, \CL)$ and $H^0(B, \CL')$ are relative tangent spaces of $R_\alpha \to U_{\bk}$ and $Z^\circ_{a',\bk} \to U_{\bk}$, we see that the relative tangent space of  $M_\alpha \to U_{\bk}$ is naturally identified with that of $Q_\alpha \to U_{\bk}$. The injectivity follows from the fact that $\Mor(B,S) \to \Mor(B,\BP^1\times\BP^1)$ is injective on the set of morphisms whose images are not contracted by $\rho$.

    Next, the restriction of $Z_{a',\bk}$ over the image of $\psi^\circ_{a'-2g,\bk}$ is an open subset of a $\BP^{2a'-2g+1-\sum_i{k_i}}$-bundle over $\Pic^{a'}(B)\times \image(\psi^\circ_{a'-2g,\bk})$ by the same argument as Proposition \ref{prop:configuration cover is smooth and irred}, thus $Q^\circ_\alpha$ retains the bundle structure via pullback. In particular, $Q^\circ_\alpha$ is a $\BP^{2a'-2g+1-\sum_i{k_i}}$-bundle over $R^\circ_\alpha \times \Pic^{a'}(B)$. Since $h = 2a+2a'-\sum_i k_i$, we have $(M_\alpha)|_{U^\circ}$ is an open subset of a $\BP^{h-2a-2g+1}$-bundle over $R^\circ_\alpha \times \Pic^{a'}(B)$. 
\end{proof} 

Finally, we conclude the smoothness and irreducibility of $M_\alpha$. We remind the reader that the inequality (\ref{eq:bounds on a and a'}) is assumed throughout the section.

\begin{theorem}\label{thm:irreducibility of M_alpha}
    Suppose that $a, a' > \max_i\{k_i\}  + 5g$ and $2a, 2a'> \sum_i k_i + 6g - 2$. Then $M_\alpha$ is smooth and geometrically irreducible.
\end{theorem}

\begin{proof}
    Proposition \ref{prop:Malphasmooth} shows that the map $M_\alpha \to U_{\bk}$ is smooth, hence the image of any irreducible component of $M_\alpha$ meets the dense open $U^\circ$. Then Proposition \ref{prop:structure theorem for M_alpha} shows that $M_{\alpha}$ must be geometrically irreducible.
\end{proof}

\section{The Virtual Bar Complex}\label{sec:The Virtual Bar Complex}

The purpose of this section is to package the incidence stratification introduced in Section \ref{subsec:Stratifying the space of sections} into a cohomological object suitable for the trace computation in the final counting arguments. We first recall the bar complex construction associated with the incidence family $Z^{\alg}$ from \cite[Section 7]{DLTT25}.

Let $R\subset P_{\infty}$ be a downward-closed subposcheme that is proper over $U_{\bk}$, and write
\[
    Z^{\alg}_R = Z^{\alg}
    \times_{\CB\times P_{\infty}}
    (\CB\times R).
\]
The bar complex is the simplicial scheme $B(\CB\times R, Z^{\alg}_R)$ which assigns to the simplex $[m]$ the scheme
\[
    B(\CB\times R, Z^{\alg}_R)([m]) = \left\{
    (\fb,w<x_0\leqslant\cdots\leqslant x_m,z)\ \middle|\ (\fb,w<x_i)\in\CB\times R,\ z\in Z^{\alg}_{\fb,w<x_m}
    \right\}
\]
and which assigns to a non-decreasing map $\iota\colon[n] \to [m]$ the morphism $B(\CB\times R, Z^{\alg}_R)([m]) \to B(\CB\times R, Z^{\alg}_R)([n])$ sending
\[
    (\fb,w<x_0\leqslant\cdots\leqslant x_m,z \in Z^{\alg}_{\fb,w<x_m}) \mapsto (\fb,w<x_{\iota(0)}\leqslant\cdots\leqslant x_{\iota(n)},z \in Z^{\alg}_{\fb,w<x_{\iota(n)}})
\]
This is well defined as $Z^{\alg}_{\fb,w<x_m} \subset Z^{\alg}_{\fb,w<x_{\iota(n)}}$. There is also a natural augmentation morphism $\varepsilon\colon B(\CB\times R, Z^{\alg}_R) \to E$ given by 
\[
    (\fb,w<x_0\leqslant\cdots\leqslant x_m,z) \mapsto (\fb, w, z).
\]

The bar complex $B(\CB\times R,Z^{\alg}_R)$ inherits a stratification from the poset of saturated combinatorial types occurring in $R$. After choosing a linear ordering of these types compatible with their partial order, we may filter $B(\CB\times R,Z^{\alg}_R)$ according to the saturated type of the terminal element $x_m$ of a simplex. The graded piece corresponding to a type $T$ is supported on the incidence stratum $Z^{\alg}|_{\CB\times\CN_T}$, and its combinatorial contribution is encoded by the complex $\mu'(T)[1]$ which we now recall. For a saturated combinatorial type $T$ of the relative poscheme $R$, set 
\[
    \CS_T = \coprod_{\operatorname{sat}(T')=T}\CN_{T'}.
\]
The nerve associated to $Z^{\alg}|_{\CB\times\CN_T}$ is defined to be the simplicial scheme 
\[
    \varepsilon_T:
    \left(
        N(-\infty,Z^{\alg}|_{\CB\times\CS_T}) < Z^{\alg}|_{\CB\times\CN_T}
    \right)
    \longrightarrow Z^{\alg}|_{\CB\times\CN_T}
\]
whose $m$-th simplex is
\[
    \left\{ (\fb,w<x_0\leqslant\cdots\leqslant x_m<y,z)\ \middle|\
        \begin{array}{l}
            (w<y)\in\CN_T,\ (w<x_m)\notin\CS_T,\\
            \fb\in\CB,\ z\in Z^{\alg}_{\fb,w<y}
        \end{array}
    \right\}
\]
and whose augmentation morphism $\varepsilon_T$ sends
\[
    (\fb, w<x_0\leqslant\cdots\leqslant x_m<y,z) \mapsto (\fb, w<y,z).
\]
Then $\mu'(T)[2]$ is defined by the distinguished triangle
\[
    \underline{\BZ_\ell} \longrightarrow R\varepsilon_{T*} (\varepsilon_T)^*\underline{\BZ_\ell}
    \longrightarrow
    \mu'(T)[2]
    \longrightarrow
    \underline{\BZ_\ell}[1].
\]
When $T$ is not an essential type, \cite[Lemma 7.2]{DLTT25} shows that $\mu'(T)\cong 0$, and more importantly, \cite[Theorem 7.3]{DLTT25} gives the following spectral sequence 
\[
    \bigoplus_{i+j=n}E_1^{i,j} = \bigoplus_{T\colon\text{ essential type of } R} H^n_c((Z^{\alg}|_{\CB\times\CN_T})_{\overline{k}}, \mu'(T)[1]) \Longrightarrow H^n_c (B(\CB\times R, Z^{\alg}_R)_{\overline{k}}, \BZ_\ell).
\]

However, the dimensions of the linear
strata of $Z^{\alg}_R$ need not always equal their expected dimensions. For this reason, it is convenient to consider the corresponding \textit{virtual bar complex}, which is a complex of sheaves computing the cohomology of a bar complex assuming that each stratum has the expected dimension. A useful feature in the higher genus setting is that the combinatorial part of the construction of the virtual bar complex is unchanged from \cite{DLTT25}. The dependence on the curve $B$ enters through the configuration spaces on $B$ and the product of Picard groups $\CB=\Pic^a(B)\times\Pic^{a'}(B)$. We first isolate this contribution before defining the virtual bar complex. 

Consider the nerve $\left(N(-\infty,\mathcal{S}_T) < \mathcal{N}_T\right)$ whose $m$-th simplex is given by 
\[
    \{w<x_0\leqslant\cdots\leqslant x_m<y\mid (w<y)\in \mathcal{N}_T,(w<x_m)\notin \mathcal{S}_T\}.
\]
Let $\tilde\varepsilon_T$ be its augmentation morphism and define $\tilde\mu(T)[2]$ to be the complex filling the distinguished triangle
\[
    \underline{\BZ_\ell}
    \longrightarrow
    R\tilde\varepsilon_{T*} \tilde\varepsilon_T^*\underline{\BZ_\ell}
    \longrightarrow
    \tilde\mu(T)[2]
    \longrightarrow
    \underline{\BZ_\ell}[1].
\]

Denote by $\pi\colon \mathcal{B} \times \mathcal{N}_T \to \mathcal{N}_T$ the projection. Let $\mu(T)[2]$ be the complex defined analogously using the nerve $\left(N(-\infty, \CB\times\mathcal{S}_T) <  \CB\times\mathcal{N}_T\right)$ with augmentation morphism $\hat{\varepsilon}_{T}$ in \cite[Section 7.4]{DLTT25}. We then have:

\begin{lemma}
    \label{lem:mobiusdescent} There is an isomorphism $\mu(T)\cong \pi^*\tilde\mu(T)$.
\end{lemma} 

\begin{proof}
    By smooth base change, observe that $\pi^*R\tilde{\varepsilon}_{T*}\tilde{\varepsilon}_{T}^*\underline{\BZ_\ell}\cong R\hat{\varepsilon}_{T*}\pi^*\underline{\BZ_\ell},$ where by abuse of notation the second $\pi$ is also the projection $\left(N(-\infty,\mathcal{B}\times \mathcal{S}_T)< \mathcal{B}\times \mathcal{N}_T\right)\to \left(N(-\infty,\mathcal{S}_T)< \mathcal{N}_T\right)$. This gives the diagram below of distinguished triangles, and since two of the vertical arrows are isomorphisms, the result follows from the ``derived'' version of the five lemma.
    \[\begin{tikzcd}
    	{\underline{\BZ_\ell}} & {R\hat{\varepsilon}_{T*}\pi^*\underline{\BZ_\ell}} & {\mu(T)[2]} & \bullet \\
    	{\underline{\BZ_\ell}} & {\pi^*R\tilde{\varepsilon}_{T*}\tilde{\varepsilon}_{T}^*\underline{\BZ_\ell}} & {\pi^*\tilde\mu(T)[2]} & \bullet
    	\arrow[from=1-1, to=1-2]
    	\arrow["\cong"', from=1-1, to=2-1]
    	\arrow[from=1-2, to=1-3]
    	\arrow["\cong"', from=1-2, to=2-2]
    	\arrow[from=1-3, to=1-4]
    	\arrow[dashed, from=1-3, to=2-3]
    	\arrow[from=2-1, to=2-2]
    	\arrow[from=2-2, to=2-3]
    	\arrow[from=2-3, to=2-4]
    \end{tikzcd}\] 
\end{proof}

Now, fix a nef curve class $\alpha$. Let $E(\alpha) = 2(a-g+1) + 2(a'-g+1)$, where $a = a(\alpha)$ and $a' = a'(\alpha)$ are the degrees introduced in Section \ref{subsec:quartic del pezzo surfaces}, and let $n_{\alpha,T} = E(\alpha)-2\sum_i k_i(\alpha)-\gamma(T)$ to be the expected dimension of $Z^{\alg}_{\fb, w<x}$, where $w<x$ has type $T$. We also let $\langle d \rangle$ denote the operation $(d)[2d]$ of shifting by $2d$ and twisting by $d$. Then for an essential type $T$, we define
\[
    \CA_{\alpha, T}^{\rel} = H_c^*\left(\mathcal{N}_{T,\overline{k}},\tilde{\mu}(T)[1]\otimes \BQ_\ell\wangle{-n_{\alpha,T}}\right),
\]
where $\tilde{\mu}(T)[1]$ plays the role of remembering the inclusion-exclusion of the strata and the shift by $-n_{\alpha,T}$ computes the cohomology of the linear fibre. The relative virtual bar complex is 
\[
    \CA_{\alpha}^{\rel}  = \bigoplus_{T\colon\text{ essential}} \CA_{\alpha, T}^{\rel}.
\]
Next, let
\[
     \CA_{\alpha, T} = H_c^*\left(\mathcal{B}, \BQ_\ell\right)\otimes \CA_{\alpha, T}^{\rel}. 
\]
We define the total virtual bar complex as
\[
    \CA_{\alpha} = \bigoplus_{T\colon\text{ essential}} \CA_{\alpha, T}.
\]

The following proposition computes the Frobenius trace of $\CA_{\alpha, T}$. Its proof follows directly from that of \cite[Proposition 7.8]{DLTT25}.

\begin{prop}
    Let $T$ be an essential saturated type. Then we have
    \[
        \sum_i (-1)^i \trace(\Frob\curvearrowright \CA^i_{\alpha, T}) = -(\#\Jac(B)(k))^2 q^{2a+2a'+4-4g} \sum_{(w<x)\in \CN_T(k)} \mu_k(w,x)q^{-\gamma(x)},
    \]
    where $\mu_k(w,x)$ is the M\"obius function in \cite[Definition 7.6]{DLTT25}.
\end{prop}

Finally, we estimate the convergence of the total virtual bar complex. This allows us to compute the error term arising from estimating the bar complex using the virtual bar complex.

\begin{theorem}\label{thm:bounding truncated virtual bar complex}
    Fix $0<\eta<1$. For any constant $d'>64g + 80$, we have that if $q$ is large enough so that $d'(4/\sqrt{q})^{(1-\eta)}< 1$ and $(4/\sqrt{q})^\eta < 1$, then for all $J > 0$, we have 
    \[
        \left|\Frob, \tau_{\leqslant 2m_\alpha-J}\CA_\alpha\right| = O\left(q^{m_\alpha}(d')^{\sum_i k_i}(4/\sqrt{q})^{(1-\eta)J}\right),
    \]
    where $m_\alpha = -K_S\cdot \alpha + 2(1-g) + 2$.
\end{theorem}

\begin{proof}
    We first consider the relative estimate following the proof of \cite[Theorem 7.9]{DLTT25}. Indeed, the only modification is the final application of \cite[Proposition 2.10]{DLTT25}, which is generalized in Lemma \ref{L1_trace_invariants}. Using the notation of \cite[Section 7.5]{DLTT25} and of the proof of \cite[Proposition 2.10]{DLTT25}, we have 
    \[
        |\Frob, (\mathcal{C})_{n,i}|=O\left(id'^n(4/\sqrt{q})^i\right),
    \]
    where 
    \[
        d'>\max\{2\cdot 8\cdot 1+8\cdot 8 \cdot 1 \cdot q^{-1/2}gq^{1/2}+8\cdot 8 \cdot 1 \cdot (q^{-1/2})^2q,4\}=64g+80.
    \]
    Observing that $\mathcal{C}$ is stratified by the size $s$ of the support of $T$, where the homological degree ranges from $J+s$ to $\infty$, and the first grading is concentrated in one degree, namely $\sum k_i + s$, it follows that 
    \begin{align*}
        \left|\Frob, \tau_{\leqslant 2(m_\alpha- 2g) - J}\CA_{\alpha}^{\rel}\right| &\ll q^{m_\alpha - 2g}\sum_{s=0}^\infty\sum_{i=J+s}^\infty i d'^{\sum_i k_i+s}(4/\sqrt{q})^i \\
        &\ll q^{m_\alpha - 2g}d'^{\sum_i k_i}\sum_{s\geqslant 0}d'^s\sum_{i\geqslant J+s}((4/\sqrt{q})^{1-\eta})^i.
    \end{align*}
    This is an arithmetico-geometric sequence, and simplifies to 
    \[
        \left|\Frob, \tau_{\leqslant 2(m_\alpha- 2g) - J}\CA_{\alpha}^{\rel}\right| 
        \ll q^{m_\alpha - 2g}(d')^{\sum_{i=1}^4k_i}(4/\sqrt{q})^{(1-\eta)J}.
    \]
    
    To give an estimate for the total virtual bar complex, we first notice that Lemma \ref{lem:mobiusdescent} and the K\"unneth formula give the decomposition 
    \[
        \CA_{\alpha} = H_c^*\left(\mathcal{B}, \BQ_\ell\right)\otimes \CA_{\alpha}^{\rel}.
    \]
    Let $0\leqslant r \leqslant 4g$ and $t = 4g - r$. Then $H^r_c(\CB, \Ql)$ contributes to $\tau_{\leqslant 2m_\alpha-J}\CA_\alpha$ only through the truncation $\tau_{\leqslant 2(m_\alpha-2g) - (J - t)}\CA_{\alpha}^{\rel}$. By Deligne's estimates, the absolute value of a Frobenius eigenvalue on $H^r_c(\CB, \Ql)$ is $q^{2g - t/2}$. Since the sum of Betti numbers $b_c(\CB)$ is independent of $q$ and $\alpha$ (in fact is equal to $16^g$), summing up over all possible $r$ gives the desired bound.

\end{proof}

\section{Manin's conjecture for higher genus curves}\label{sec:Manin's conjecture for higher genus curve}

In this last section, our goal is to prove Theorem \ref{thm:main-manin}. To achieve this, we will perform the point-counting on the $\BG_m^2$-torsor $\widetilde{M}_{\alpha} \to M_{\alpha}$ via the Grothendieck--Lefschetz trace formula:
\begin{equation}\label{eq:GLtrace}
    \#\widetilde{M}_{\alpha}(k) = \sum_{i} (-1)^i \trace (\Frob \curvearrowright H^i_c((\widetilde{M}_{\alpha})_{\overline{k}}, \Ql)).
\end{equation}

Our strategy follows that of \cite[Section 9]{DLTT25}, where we decompose the above summation into two parts based on the cohomological degree. We first provide the set-up needed for such a decomposition.

\subsection{Preparation for Manin's Conjecture}

\subsubsection{Truncation parameters}\label{subsec:Truncation parameters}

Define the truncation parameter $I\in \BQ$ to be:
\[
    I = \frac{1}{8}\min\{2a - \sum_i k_i, 2a' - \sum_i k_i, a - \max_i\{k_i\}, a' - \max_i\{k_i\}\} - \frac{1}{2},
\]
and set 
\[
    J = I - 2g.
\]
From now on, we assume 
\[
    \min \{2a - \sum_i k_i, 2a' - \sum_i k_i, a - \max_i\{k_i\}, a' - \max_i\{k_i\}\} \geqslant 16g + 8
\]
so that $J > 0$. The constant $J$ will play the role of dividing the trace sum into a major and a minor part and allows us to control the asymptotics of each. Furthermore, the lower bound validates the assumptions in Section \ref{sec:Higher genus curves on del Pezzo surfaces} and ensures that the space $\widetilde{M}_\alpha$ is irreducible by Theorem \ref{thm:irreducibility of M_alpha}.

\subsubsection{The incidence stratification}

We now truncate the incidence space $Z^{\alg}$ introduced in
Section \ref{subsec:Stratifying the space of sections}. Let
\[
    P =
    \left\{
        (w<x)\in P_{\infty} \,\middle|\, \gamma(w<x)\leqslant 2I
    \right\}.
\]
By \cite[Lemma 8.10]{DLTT25}, the poscheme $P$ is proper over $U_{\bk}$. We write
\[
    Z^{\alg}_P = Z^{\alg} \times_{\CB\times P_{\infty}} (\CB\times P)
\]
for the resulting truncated incidence space. The following proposition produces a good locus inside $\CB\times U_{\bk}$, over which every incidence space has its expected dimension.

\begin{prop}\label{prop:strata of Z with exp dimension}
    There exists a locus $F_1 \cup F_2$ of codimension at least $4J$ in $U_{\bk}$ such that for any $(w<x) \in P|_{(U_{\bk}\setminus(F_1 \cup F_2))}(\overline{k})$ with $x$ being saturated and $y \in Q^{JB}(\overline{k})$ such that $x \prec y$, the spaces $Z^{\alg}_{\fb,w<x}$ and $Z^{\alg}_{\fb,w<y}$ have the expected dimensions for any $\fb \in \CB(\overline{k})$.
\end{prop}

\begin{proof}

    Let $(w<x)$ be a saturated element of $P(\overline{k})$ and $(w<y)$ be such that $x\prec y$. Then there is an identification
    \[
        Z^{\alg}_{\fb,w<y} = \Gamma(B_{\overline{k}}, \CL_1\otimes V_1)_{y_1} \oplus \Gamma(B_{\overline{k}}, \CL_2\otimes V_2)_{y_2}.
    \]
    Here, the space $\Gamma(B_{\overline{k}}, \CL_j\otimes V_j)_{y_j}$ is defined to be the space of sections $s \in \Gamma(B_{\overline{k}}, \CL_j\otimes V_j)$ satisfying $s(y_j(l_{i,j})) \subset \CL_j\otimes l_{i,j}$ for $i=1,\dots, 4$.
    By Proposition \ref{prop:dimboundWak}, the space $\Gamma(B_{\overline{k}}, \CL_1\otimes V_1)_{y_1}$ has the expected dimension if $y_1$ is in the image of
    \[
        \phi_{a_y,\bk_y}\colon Z_{a_y, \bk_y, 1}^\circ \to W_{a_y, \bk_{y,1}}.
    \]
    Here, $a_y = a - m_0(y) - m_{V_2}(y) - \sum_i m_{l_{i,2}}(y)$ and $k_{y,i,1} = m_{l_{i,1}}(y) + m_{l_{i,1}\oplus l_{i,2}}(y)$. In particular, this requires the following hypotheses to be satisfied:
    \begin{itemize}
        \item $a_y > 4g$,
        \item $a_y - \max_i\{k_{y,i,1}\} > 5g$, and
        \item $2a_y - \sum_{i=1}^4 k_{y,i,1} > 6g - 2$
    \end{itemize}
    
    For the second inequality, we have that for any $j$,
    \[
        a_y - k_{y,j,1} \geqslant a - k_j - (2I + 2).
    \]
    This can be shown by the following computation. Set 
    \[
        A_y = m_0(y) + m_{V_2}(y) + \sum_{i=1}^4 m_{l_{i,2}}(y)
    \] 
    and
    \[
        E_1(w<y) = 2A_y + \sum_{i=1}^4 (k_{y,i,1} - k_i).
    \]
    We may write
    \[
        a_y - k_{y,j,1} = a - k_j - (A_y + k_{y,j,1} - k_j) = a - k_j - (E_1(w<y) - A_y - \sum_{i\neq j} (k_{y,i,1} - k_i)).
    \]
    Since $E_1(w < y) \leqslant \gamma(w<y) \leqslant 2I + 2$, it suffices to show
    \[
        \sum_{i \neq j} k_i \leqslant A_y + \sum_{i\neq j}(m_{l_{i,1}}(y) + m_{l_{i,1}\oplus l_{i,2}}(y)).
    \]
    We compute this pointwise using chains. Let $c$ be a point of $B$ and denote by $t_i(c) = \len_c(T_i)$, where $T_i = w(l_{i,1}\oplus l_{i,2}) \subset B_i$. Then we have
    \[
        \sum_{i\neq j} k_i = \sum_{i\neq j} \sum_{c \in \Supp(T_i)} t_i(c).
    \]
    On the other hand, since $w < y$, we have $t_i(c) \leqslant \len_c(y(l_{i,1}\oplus l_{i,2}))$, where the latter quantity is
    \[
        \len_c(y(l_{i,1}\oplus l_{i,2})) = m_0(f_{y,c}) + m_{l_{i,1}}(f_{y,c}) + m_{l_{i,2}}(f_{y,c}) + m_{l_{i,1}\oplus l_{i,2}}(f_{y,c}).
    \]
    Notice that since the supports of the $T_i$'s are disjoint, each point $c$ can be in at most one $T_i$, and thus
    \[
        \sum_{i\neq j}\sum_{c\in \Supp(T_i)} m_0(f_{y,c}) \leqslant \sum_{c\in B} m_0(f_{y,c}) = m_0(y).
    \]
    This gives
    \begin{align*}
        \sum_{i\neq j} k_i &\leqslant \sum_{i\neq j}\sum_{c\in \Supp(T_i)} \len_c(y(l_{i,1}\oplus l_{i,2})) \\
        &\leqslant m_0(y) + \sum_{i\neq j}\sum_{c\in \Supp(T_i)} \left(m_{l_{i,1}}(f_{y,c}) + m_{l_{i,2}}(f_{y,c}) + m_{l_{i,1}\oplus l_{i,2}}(f_{y,c}) \right)\\
        &\leqslant m_0(y) + \sum_{i\neq j} \left( m_{l_{i,1}}(y) + m_{l_{i,2}}(y) + m_{l_{i,1}\oplus l_{i,2}}(y) \right)\\
        &\leqslant m_0(y) + m_{V_2}(y) + \sum_{i=1}^4 m_{l_{i,2}}(y) + \sum_{i\neq j}(m_{l_{i,1}}(y) + m_{l_{i,1}\oplus l_{i,2}}(y)).
    \end{align*}

    The same computation in \cite[Section 8.2.1]{DLTT25} gives
    \[
        2a_y - \sum_{i=1}^4 k_{y,i,1} \geqslant 6I + 2
    \]
    and
    \[
        a_y \geqslant \frac{1}{2}(2a_y - \sum_{i=1}^4 k_{y,i,1}) \geqslant 3I + 1.
    \]
    The assumption on $I$ validates the above inequalities, and hence by Proposition \ref{prop:dimboundWak}, the complement of the image of $\phi_{a_y,\bk_y}$ has codimension at least
    \[
        2a_y - \sum_i k_{y,i,1} - 8g + 1.
    \]
    For each $y$ in the complement, the same computation in \cite[Section 8.2.1]{DLTT25} shows that the locus of the corresponding $w$ such that $w<y$ has codimension at least $4I - 8g$.
    
    Let $F_1$ be the union of the Galois orbits of the closures of all such loci over $\overline{k}$ while $T$ runs over all (finitely many) combinatorial types we consider in the definition of $P$. We construct $F_2$ similarly for $\Gamma(B, \CL_2\otimes V_2)_{y_2}$. Then the codimension of $F_1\cup F_2$ in $U_{\bk}$ is at least $4I - 8g = 4J$.
\end{proof}

Finally, let $F_0\subset U_{\bk}$ be the complement of $U^\circ$ for $U^\circ$ defined before Proposition \ref{prop:structure theorem for M_alpha} and set $F = F_0\cup F_1 \cup F_2$. Then by the previous proposition and Proposition \ref{prop:dimboundWak}, the codimension of $F$ in $U_{\bk}$ is at least $4J$. We define $U =  \CB\times (U_{\bk} \setminus F)$ and let $Z^{\alg}_{P_U}$ be the restriction of $Z^{\alg}_{P}$ to $U$.

\subsection{Counting via the bar complex}\label{subsec:Counting via the bar complex}

Write $k = \BF_q$ and define
\[
    m_\alpha = \dim \widetilde{M}_{\alpha} = 2a+2a'-\sum_i k_i + 2(1-g) + 2.
\]
We divide the right-hand side of the Grothendieck--Lefschetz trace formula \ref{eq:GLtrace}:
\[
    \#\widetilde{M}_{\alpha}(k) = \sum_{i} (-1)^i \trace (\Frob \curvearrowright H^i_c((\widetilde{M}_{\alpha})_{\overline{k}}, \Ql)).
\]
into two parts based on the cohomological degree:
\[
    \trace_{i < 2m_\alpha - J + 2} = \sum_{i < 2m_\alpha - J + 2} (-1)^i \trace (\Frob \curvearrowright H^i_c((\widetilde{M}_{\alpha})_{\overline{k}}, \Ql))
\]
and
\[
    \trace_{i\geqslant 2m_\alpha - J  + 2} = \sum_{i\geqslant 2m_\alpha - J  + 2} (-1)^i \trace (\Frob \curvearrowright H^i_c((\widetilde{M}_{\alpha})_{\overline{k}}, \Ql)).
\]

For the first summand, the Betti number bound in Theorem \ref{thm:main}, Deligne's estimate, and the Leray spectral sequence give:

\begin{prop}\label{prop:bound on first summand}
    There exists a constant $C_1$ such that 
    \[
        \trace_{i < 2m_\alpha - J  + 2} = O(q^{m_\alpha - J/2  + 1}C_1^h).
    \]
\end{prop}

For the second summand, we restrict our attention to the locus $U$. Since the codimension of $F$ is at least $4J$ and $J > 0$, we have 
\[
    2m_\alpha - J + 2 > 2m_\alpha - 8J + 2.
\]
Moreover, the assumption on $I$ shows that $M_\alpha \to U_{\bk}$ is flat, hence there is an isomorphism of cohomology groups in the range $i\geqslant 2m_\alpha - J + 2$:
\[
    H^i_c((\widetilde{M}_{\alpha})_{\overline{k}},\Ql) = H^i_c((\widetilde{M}_{\alpha}|_U)_{\overline{k}},\Ql).
\]
Over $U$, we have the decomposition
\[
    E|_U = \widetilde{M}_{\alpha}|_U \sqcup \image(Z^{\alg}_{P_U} \to E|_U).
\]
By appealing to the excision exact sequence, it suffices to bound the trace for $E|_U$ and $\image(Z^{\alg}_{P_U}\to E|_U)$ separately.

Let's first analyze the traces for $\image(Z^{\alg}_{P_U}\to E|_U)$. By Proposition \ref{prop:strata of Z with exp dimension}, \cite[Theorem 7.5]{DLTT25}, \cite[Proposition 7.1]{DLTT25}, and \cite[Lemma 2.3]{DLTT25}, we have the estimate
\begin{align}\label{eq:trace for Z alg}
    &|\sum_{i\geqslant 2m_\alpha - J  + 2} (-1)^i \trace (\Frob \curvearrowright H^i_c(\image(Z^{\alg}_{P_U}\to E|_U)_{\overline{k}}, \Ql)) - \\
    &\sum_{T:\ \text{ess. type of}\ P}\sum_{i\geqslant 2m_\alpha - J + 2} (-1)^i \trace (\Frob \curvearrowright H^i_c((Z^{\alg}|_{U\times_{ U_{\bk}} \CN_T})_{\overline{k}}, \mu'(T)[1]\otimes\Ql))| \leqslant N. \nonumber
\end{align}
where $\mu'(T)$ is the complex defined in \cite[Section 7]{DLTT25} and
\[
    N = |\Frob, \bigoplus_{T:\ \text{ess. type of}\ P} H^{\lceil 2m_\alpha - J + 2 \rceil}_c((Z^{\alg}|_{U\times_{U_{\bk}} \CN_T})_{\overline{k}}, \mu'(T)[1]\otimes\Ql)|.
\]
Hence it suffices to compute the trace for $H^i_c((Z^{\alg}|_{U\times_{ U_{\bk}} \CN_T})_{\overline{k}}, \mu'(T)[1]\otimes\Ql))$. For this, we appeal to the virtual bar complex, and we will need the following lemma:

\begin{lemma}\label{lem:pass to bar complex}
    For $i \geqslant 2m_\alpha - J + 2$, there is an isomorphism
    \begin{align*}
        H^i_c((Z^{\alg}|_{U\times_{ U_{\bk}} \CN_T})_{\overline{k}}, &\mu'(T)[1]\otimes\Ql)\\
        &\cong H^i_c((\CB\times \CN_T)_{\overline{k}},  \mu(T)(-n_{\alpha,T})[-2n_{\alpha,T}+1]\otimes\Ql),
    \end{align*}
    where $n_{\alpha,T} = 2(a-g+1)+2(a'-g+1)-2\sum_i k_i -\gamma(T)$.
\end{lemma}

\begin{proof}
    Since $Z^{\alg}$ has the expected codimension over $U\times_{U_{\bk}} \CN_T$, we have
    \begin{align*}
        H^i_c((Z^{\alg}|_{U\times_{ U_{\bk}} \CN_T})_{\overline{k}}, &\mu'(T)[1]\otimes\Ql) \\
        &\cong  H^i_c(({U\times_{U_{\bk}} \CN_T})_{\overline{k}}, \mu(T)(-n_{\alpha,T})[-2n_{\alpha,T}+1]\otimes\Ql).
    \end{align*}
    Using the exact triangle defining $\mu(T)[2]$, the cohomology group
    \[
        H^i_c((F\times_{ U_{\bk}} \CN_T)_{\overline{k}}, \mu(T)(-n_{\alpha,T})[-2n_{\alpha,T}+1]\otimes\Ql)
    \]
    vanishes for $i$ greater than 
    \[
        2\sum_i k_i + 4g - 8J + 2|\Supp (T)| + \rank(T) + 2n_{\alpha,T} + 1 = 2m_\alpha - 8J + 1 - \kappa(T).
    \]
    Here, $\sum_i k_i + 2g - 4J + |\Supp (T)|$ is the dimension bound for $F\times_{ U_{\bk}} \CN_T$ and $\rank(T)$ is the expected cohomological dimension of the simplicial scheme in the definition of $\mu(T)$. On the other hand, the cohomological dimension of the nerve defining $\mu(T)$ is at most $\rank(T)$. Thus 
    \begin{align*}
        \dim (Z^{\alg}|_{F\times_{U_{\bk}} \CN_T}) + &2n_{\alpha,T} + 1\\
        &\leqslant 2\sum_i k_i + 4g - 8J + 2|\Supp (T)| + \rank(T) + 2n_{\alpha,T} + 1\\
        & = 4a + 4a' - 4g - 2\sum_i k_i + 9 - 8J - \kappa(T)\\
        & = 2m_\alpha - 8J + 1 - \kappa(T).
    \end{align*}  
    This implies that for $i\geqslant 2m_\alpha - J + 2$, we have an isomorphism
    \begin{align*}
        H^i_c(({U\times_{U_{\bk}} \CN_T})_{\overline{k}},  \mu(T)(-n_{\alpha,T})&[-2n_{\alpha,T}+1]\otimes\Ql) \cong\\
        &H^i_c((\CB\times \CN_T)_{\overline{k}},  \mu(T)(-n_{\alpha,T})[-2n_{\alpha,T}+1]\otimes\Ql).
    \end{align*}
\end{proof}

Fix constants $C_2 > 128g+160$ and  $C_3> (128g+160)^4$. We have the following:

\begin{lemma}\label{lem:bound on second summand}
    Suppose that $q > C_3$. Then we have 
    \begin{align*}
        &\sum_{i\geqslant 2m_\alpha - J  + 2} (-1)^i \trace (\Frob \curvearrowright H^i_c(\image(Z^{\alg}_{P_U}\to E|_U)_{\overline{k}}, \Ql))\\
        &= -(\#\Jac(B)(k))^2q^{2a+2a'+4-4g} \sum_{(w < x) \in (U_{\bk} < Q^{JB})(k)}  \mu_k(w,x) q^{-\gamma(x)} + O(q^{m_\alpha - J/4 + 1}  C_2^h).
    \end{align*}
\end{lemma}

\begin{proof}
    The idea is to pack the second trace summation in the left-hand side of the inequality \ref{eq:trace for Z alg} into the total virtual bar complex defined in Section \ref{sec:The Virtual Bar Complex}. By Lemma \ref{lem:pass to bar complex}, we have an isomorphism 
    \begin{align*}
        H^i_c((Z^{\alg}|_{U\times_{ U_{\bk}} \CN_T})_{\overline{k}}, &\mu'(T)[1]\otimes\Ql)\\
        &\cong H^i_c((\CB\times \CN_T)_{\overline{k}},  \mu(T)(-n_{\alpha,T})[-2n_{\alpha,T}+1]\otimes\Ql).
    \end{align*}
    For $T$ an essential type not in $P$, we have $\gamma(T) > 2I \geqslant 2J$. Moreover, a simple computation shows that $\gamma(T) \leqslant 2\kappa(T)$. Hence $\kappa(T) > J$. This implies the cohomology group
    \[
        H^i_c((\CB\times\CN_T)_{\overline{k}}, \mu(T)(-n_{\alpha,T})[-2n_{\alpha,T}+1]\otimes\Ql)
    \]
    vanishes for $i$ greater than 
    \[
        4g + 2\sum_i k_i + 2|\Supp(T)| + 2n_{\alpha,T} + 1 + \rank(T) = 2m_\alpha + 1 - \kappa(T) < 2m_\alpha + 1 - J.
    \]
    Hence by \cite[Proposition 7.8]{DLTT25} and Theorem \ref{thm:bounding truncated virtual bar complex}, we have 
    \begin{align*}
        &\sum_{T:\ \text{ess. type of}\ P}\sum_{i\geqslant 2m_\alpha - J + 2} (-1)^i \trace (\Frob \curvearrowright H^i_c((Z^{\alg}|_{U\times_{ U_{\bk}} \CN_T})_{\overline{k}}, \mu'(T)[1]\otimes\Ql)) \\
        &= \sum_{T:\ \text{ess. type of}\ P}\sum_{i\geqslant 2m_\alpha - J + 2} (-1)^i \trace (\Frob \curvearrowright H^i_c((\CB\times\CN_T)_{\overline{k}}, \mu(T)(-n_{\alpha,T})[-2n_{\alpha,T}+1]\otimes\Ql)) \\
        &= \sum_{T:\ \text{ess. type}}\sum_{i\geqslant 2m_\alpha - J + 2} (-1)^i \trace (\Frob \curvearrowright H^i_c((\CB\times\CN_T)_{\overline{k}}, \mu(T)(-n_{\alpha,T})[-2n_{\alpha,T}+1]\otimes\Ql))\\
        & = \sum_{T:\ \text{ess. type}}\sum_{i} (-1)^i \trace (\Frob \curvearrowright H^i_c((\CB\times\CN_T)_{\overline{k}}, \mu(T)(-n_{\alpha,T})[-2n_{\alpha,T}+1]\otimes\Ql))\\
        & \hspace{10mm} + O(q^{m_\alpha - J/4 + 1}  C_2^h)\\
        & =  -(\#\Jac(B)(k))^2q^{2a+2a'+4-4g} \sum_{(w < x) \in (U_{\bk} < Q^{JB})(k)}  \mu_k(w,x) q^{-\gamma(x)} + O(q^{m_\alpha - J/4 + 1}  C_2^h).
    \end{align*}
    
    Finally, a computation similar to that in the proof of \cite[Proposition 9.4]{DLTT25} and Theorem \ref{thm:bounding truncated virtual bar complex} with $\eta = \frac{1}{2}$ and $h\geqslant \max\{\sum_i k_i, J\}$ shows that for any constants $C_2>128g+160$ and $C_3>(128g+160)^4$, if $q>C_3$, then
    \[
        N = O(q^{2a+2a'-\sum_i k_i - 2g + 4 - J/4+1}C_2^h).
    \]
    This concludes the proof.
\end{proof}

The following proposition concludes the estimate of the second partial sum of traces:

\begin{prop} \label{prop:bound on second summand}
    Let $q > C_3$. Then we have
    \begin{align*}
        &\trace_{i\geqslant 2m_\alpha - J + 2}  \\
        & = (\#\Jac(B)(k))^2q^{2a+2a'+4-4g} \sum_{(w \leqslant x) \in (U_{\bk} \leqslant Q^{JB})(k)}  \mu_k(w,x) q^{-\gamma(x)} + O\left(q^{m_\alpha - J/4 + 2}  \max\{C_1, C_2\}^h\right).
    \end{align*}
\end{prop}

\begin{proof}
    The decomposition $E|_U = \widetilde{M}_{\alpha}|_U \sqcup \image(Z^{\alg}_{P_U} \to E|_U)$ and the assumption on $U$ and $J$ give the estimate
    \begin{align*}
        &\trace_{i\geqslant 2m_\alpha - J + 2}  \\
        & = \sum_{i\geqslant 2m_\alpha - J + 2} (-1)^i ( \trace (\Frob \curvearrowright H^i_c((E|_U)_{\overline{k}}, \Ql)) - \trace (\Frob \curvearrowright H^i_c(\image(Z^{\alg}_{P_U}\to E|_U)_{\overline{k}}, \Ql)))\\
        & \hspace{10mm} +O(q^{m_\alpha - J/2 + 2}C_1^h),
    \end{align*}
    where the last error term comes from a bound on the Frobenius acting on a subspace of 
    \[
        H^j_c((\widetilde{M}_{\alpha}|_U)_{\overline{k}},\Ql),
    \]
    with $j$ being the largest integer less than $2m_\alpha - J + 2$. A computation similar to that in Lemma \ref{lem:bound on second summand} together with an analogue of Theorem \ref{thm:bounding truncated virtual bar complex} for the constant sheaf $\Ql$ gives the following bound for traces on $E|_U$:
    \begin{align*}
        &\sum_{i\geqslant 2m_\alpha - J + 2} (-1)^i  \trace (\Frob \curvearrowright H^i_c((E|_U)_{\overline{k}}, \Ql)) \\
        = & (\#\Jac(B)(k))^2q^{2a+2a'+4-4g} \sum_{w \in U_{\bk}(k)}   q^{-\gamma(w)} + O(q^{m_\alpha - J/4 + 1}  C_2^h).
    \end{align*}
    Combining this with the estimates in Lemma \ref{lem:bound on second summand} gives the desired result.
\end{proof}

Altogether, Proposition \ref{prop:bound on second summand} and Proposition \ref{prop:bound on first summand} give:

\begin{theorem}\label{thm:expressionforMtilde}
    Let $q > C_3$. Then we have
    \begin{align*}
        &\#\widetilde{M}_\alpha(k) \\
        &= (\#\Jac(B)(k))^2 q^{2a+2a'+4-4g} \sum_{(w\leqslant x)\in (U_{\bk}\leqslant Q^{JB})(k)} \mu_k(w,x)q^{-\gamma(x)} + O\left(q^{m_\alpha-J/4+2}\max\{C_1,C_2\}^h\right),
    \end{align*}
    where the implied constants do not depend on $q, a, a', \bk$.
\end{theorem}

Finally, we have the following proposition which allows us to compute the formula in Theorem \ref{thm:expressionforMtilde} explicitly.

\begin{prop}\label{betterresidueofbarcomplex}
    There exists $\eta_1 > 0$ which does not depend on $\bk$ and $q$ such that 
    \[
        q^{\sum_i k_i} \sum_{(w\leqslant x) \in (U_{\bk}\leqslant Q^{JB})(k)} \mu_k(w,x) q^{-\gamma(x)} = \tau_{-K_S}(S)(\#\Jac(B)(k))^{-2} q^{4g-2} (1-q^{-1})^{2} + O(q^{-\eta_1 \min\{k_i\}}).
    \]
\end{prop}

\begin{proof}
    Let $\textbf{Z}(t)$ be the virtual height zeta function introduced in \cite[Section 9.4]{DLTT25}. We have 
    \[
        \lim_{t_i\to 1} \prod_{i=1}^4 (1-t_i) \textbf{Z}(t) =  P_B(q^{-1})^4 (1-q^{-1})^{-4}  \prod_{c\in|B|} (1 - q^{-|c|})^6 (1+6q^{-|c|} + q^{-2|c|}),
    \]
    where $P_B(t)$ is the numerator of the Hasse--Weil Zeta function for $B$. On the other hand, since $|\Jac(B)(k)| = P_B(1)$ and
    \[
    P_B(t) = \prod_{i=1}^{2g} (1-\alpha_i t),
    \]
    where $\alpha_i$ comes in conjugates and $|\alpha_i| = q^{1/2}$, we have
    \[
    P_B(q^{-1}) =  \prod_{i=1}^{2g} (1-\frac{\alpha_i}{q}) = \prod_{i=1}^{2g} (1-\frac{1}{\overline{\alpha_i}}) = q^{-g} P_B(1).
    \]
    Hence $(\#\Jac(B)(k)) = q^g P_B(q^{-1})$. This allows us to rewrite the Tamagawa constant as
    \[
    \tau_{-K_S}(S) = q^{2-8g} (1 - q^{-1})^{-6} (\#\Jac(B)(k))^6 \prod_{c\in |B|} \left( (1 - q^{-|c|})^{6} \frac{\#S(\BF_{q^{|c|}})}{q^{2|c|}}\right).
    \]
    Since $\displaystyle\frac{\#S(\BF_{q^{|c|}})}{q^{2|c|}} = 1+6q^{-|c|} + q^{-2|c|}$, we have
    \[
    \lim_{t_i\to 1} \prod_{i=1}^4 (1-t_i) \textbf{Z}(t) = \tau_{-K_S}(S)(\#\Jac(B)(k))^{-2} q^{4g-2} (1-q^{-1})^{2}.
    \]
    The result then follows from \cite[Proposition 9.7]{DLTT25}.
\end{proof}

\subsection{Proof of the main results}\label{subsec: Proof of the main results}

Let $\mathscr{R}$ denote the finite set of birational morphisms
\[
    \rho\colon S\longrightarrow \BP^1\times\BP^1
\]
which contract four pairwise disjoint $(-1)$-curves described in Section \ref{subsec:quartic del pezzo surfaces}. For each $\rho\in\mathscr{R}$, we let $F_\rho,F'_\rho$ denote the pullbacks
of the two rulings on $\BP^1\times\BP^1$, and let
\[
    E_{\rho,1},\ldots,E_{\rho,4}
\]
be the exceptional curves of $\rho$. Set 
\[
    D_\rho(\alpha) = (2F_\rho - \sum_{i=1}^4E_{\rho,i})\cdot\alpha,\ D'_\rho(\alpha) = (2F'_\rho - \sum_{i=1}^4E_{\rho,i})\cdot\alpha
\]
and
\[
    G_\rho(\alpha) = \min_i\{(F_\rho - E_{\rho,i})\cdot \alpha\},\ G'_\rho(\alpha) = \min_i\{(F'_\rho - E_{\rho,i})\cdot \alpha\},\ 
    k_{\rho,i}(\alpha) = E_{\rho,i}\cdot \alpha.
\]
For any $\alpha\in\Nef_1(S)$, define
\[
    \ell_\rho(\alpha)\coloneqq\min\left\{ \frac{1}{32}D_\rho(\alpha),\frac{1}{32}D'_\rho(\alpha), \frac{1}{32}G_\rho(\alpha), \frac{1}{32}G'_\rho(\alpha),
    k_{\rho,1}(\alpha),\ldots,k_{\rho,4}(\alpha) \right\},
\]
and set
\[
    \ell(\alpha) \coloneqq \max_{\rho\in\mathscr{R}}\ell_\rho(\alpha).
\]
By the chamber decomposition of $\Nef_1(S)$ described in
Section \ref{subsec:quartic del pezzo surfaces}, the function $\ell$ is non-negative and positive on a dense open subcone of $\Nef_1(S)$. Since
$\mathscr{R}$ is finite, $\ell$ is in particular a rational, homogeneous,
continuous, and piecewise linear function. We define 
\[
    \Nef_1(S)_{\ell,\varepsilon} = \{\alpha \in \Nef_1(S)\ |\ \ell(\alpha) \geqslant \varepsilon h(\alpha)\},
\]
where $h(\alpha) = -K_S\cdot \alpha$. For each $\rho \in \mathscr{R}$, define
\[
    \CC_{\rho,\varepsilon}\coloneqq\left\{\alpha \in \Nef_1(S)\ \middle|\
    \begin{aligned}
    D_\rho(\alpha), D'_\rho(\alpha) &\geqslant 32\varepsilon h(\alpha),\\
    G_\rho(\alpha), G'_\rho(\alpha) & \geqslant 32\varepsilon h(\alpha),\\
    k_{\rho,i}(\alpha) &\geqslant \varepsilon h(\alpha),\ i = 1,\dots, 4
    \end{aligned}
    \right\}.
\]
Each $\CC_{\rho,\varepsilon}$ is a rational polyhedral cone. Then we have
\[
    \Nef_1(S)_{\ell,\varepsilon} = \bigcup_{\rho\in\mathscr{R}} \CC_{\rho,\varepsilon}.
\]
Note that we may pick $\varepsilon$ small enough such that the above region is full-dimensional.

\begin{prop}\label{prop:chamber satisfying numerical condition}
    Let $\varepsilon > 0$ and $\alpha \in \Nef_1(S)_{\ell,\varepsilon}$. If $h(\alpha) > (17g+8)/(4\varepsilon)$, then there exists a contraction $\rho$ such that
     \[
        \min\{D_\rho(\alpha), D'_\rho(\alpha), G_\rho(\alpha), G'_\rho(\alpha)\} \geqslant 16g + 8,
    \]
    and $a(\alpha), a'(\alpha) > 4g$, and $a(\alpha), a'(\alpha) >  \max_i\{k_{\rho,i}(\alpha)\} + 5g$.
\end{prop}

\begin{proof}
    Let $\rho \in \mathscr{R}$ be the birational contraction giving $\ell(\alpha) = \ell_\rho(\alpha)$. Then the claim follows from an easy computation. 

\end{proof}

We first give the point-counting estimate on $M_{\alpha}$, confirming Peyre's all the heights version of Manin's conjecture in the desired degree range:

\begin{theorem}\label{thm:point-counting on M_alpha}
    There exist constants $C_g>0$ and
    $\delta>0$, depending only on $g$, with the following property:

    Fix a sufficiently small rational $\varepsilon>0$ and a prime power $q$ such that  $q^\varepsilon>C_g$. Let $B$ be a smooth projective geometrically connected curve of
    genus $g$ over $k=\BF_q$, and let $S$ be a smooth split quartic del Pezzo surface over $k$. Let $\alpha\in \Nef_1(S)_{\ell,\varepsilon} \cap N_1(S)_{\BZ}$
    and set $h=h(\alpha)$. If $h$ is sufficiently large in terms of $g$ and $\varepsilon$, then we have
    \[
        \# M_\alpha(k) = \tau_{-K_S}(S)q^h +
        O\left(q^{(1- \delta\varepsilon)h}\right),
    \]
    where the implied constant is independent of $q$ and $\alpha$.
\end{theorem}

\begin{proof}
    Choose $\rho\in\mathscr{R}$ such that $\ell(\alpha)=\ell_\rho(\alpha)$, and use the corresponding notation
    \[
        a=F_\rho\cdot\alpha,\qquad
        a'=F'_\rho\cdot\alpha,\qquad
        k_i=E_{\rho,i}\cdot\alpha
    \]
    as in Section \ref{subsec:quartic del pezzo surfaces}. By the choice of $\alpha$, we have $\ell(\alpha)\geqslant\varepsilon h$, and hence
    \begin{align*}
        \frac{J}{4} &= \frac1{32} \min \left\{ 2a-\sum_i k_i,\  2a'-\sum_i k_i,\ a - \max_i\{k_i\}, a' - \max_i\{k_i\}\right\} -\frac{g}{2}-\frac18\\
        &\geqslant \varepsilon h -\frac{g}{2} -\frac{1}{8}.
    \end{align*}
    Thus, for $h$ sufficiently large with respect to $g$ and $\varepsilon$, all the numerical hypotheses imposed in this section are satisfied, and we have $\widetilde{M}_{\alpha}$ is irreducible by Proposition \ref{prop:chamber satisfying numerical condition} and Theorem \ref{thm:irreducibility of M_alpha}. Set $C_0\coloneqq\max\{C_1,C_2\}$. We may choose $C_g$ sufficiently large such that $q^{\varepsilon/2} \geqslant C_0$ and $q > C_3$, where $C_1, C_2, C_3$ are constants defined in this section. By Theorem \ref{thm:expressionforMtilde} and Proposition \ref{betterresidueofbarcomplex}, we have
    \begin{align*}
        \#\widetilde{M}_\alpha(k) &= q^{h+2}(1-q^{-1})^2\tau_{-K_S}(S)\\
        &\quad + O\left((\#\Jac(B)(k))^2 q^{h+4-4g-\eta_1\min_i \{k_i\}} \right) + O\left(q^{m_\alpha-J/4+2}C_0^h\right).
    \end{align*}
    Since $\widetilde{M}_\alpha\longrightarrow M_\alpha$ is a $\BG_m^2$-torsor, we have
    \[
        \#\widetilde{M}_\alpha(k) = (q-1)^2(\#M_\alpha(k)) = q^2(1-q^{-1})^2(\#M_\alpha(k)).
    \]
    Hence
    \begin{align*}
        \# M_\alpha(k) &=
        \tau_{-K_S}(S)q^h\\
        &\quad + O\left((\#\Jac(B)(k))^2 q^{h+2-4g-\eta_1\min_i \{k_i\}} \right) + O\left(q^{m_\alpha-J/4}C_0^h\right).
    \end{align*}
    Since $\min_i \{k_i\}\geqslant\varepsilon h$ and $\#\Jac(B)(k)=O(q^g)$, the first error term is
    \[
        O\left(q^{h-\eta_1\varepsilon h + O(1)} \right).
    \]
    On the other hand, the second error term is
    \[
        q^{m_\alpha-J/4}C_0^h \ll q^{h-\varepsilon h + O(1)} q^{\varepsilon h/2} = q^{h-\varepsilon h/2+O(1)}.
    \]
    Thus, after increasing the lower bound on $h$, there exists $\delta>0$, depending only on $g$, such that
    \[
        \#M_\alpha(k) = \tau_{-K_S}(S)q^h + O\left(q^{(1-\delta\varepsilon)h}\right).
    \]
\end{proof}

We are ready to prove our main theorem:

\begin{proof}[Proof of Theorem \ref{thm:main-manin}]
    Set $\mathscr{C} \coloneqq \Nef_1(S)_{\ell,\varepsilon}$. By Theorem \ref{thm:point-counting on M_alpha}, there exists some $\delta>0$ such that, for every $\alpha\in \mathscr{C}\cap N_1(S)_{\BZ}$ of sufficiently large anticanonical degree, we have the estimate
    \[
        \# M_\alpha(k) = \tau_{-K_S}(S)q^{h(\alpha)} + O\left(q^{(1-\delta\varepsilon)h(\alpha)} \right).
    \]
    Consequently,
    \begin{align*}
        N^{\mathscr{C}}(B,S,-K_S,d) = \tau_{-K_S}(S) \sum_{\substack{
            \alpha\in\mathscr{C}\cap N_1(S)_{\BZ}\\
            h(\alpha)\leqslant d
        }} q^{h(\alpha)} + O\left(
            \sum_{\substack{
            \alpha\in\mathscr{C}\cap N_1(S)_{\BZ}\\
            h(\alpha)\leqslant d
            }} q^{(1-\delta\varepsilon)h(\alpha)}
        \right).
    \end{align*}

    For $n\geqslant 0$, let $P_{\mathscr{C}}(n) \coloneqq \#\left\{\alpha\in \mathscr{C}\cap N_1(S)_{\BZ}\colon h(\alpha)=n\right\}$ be the Ehrhart quasi-polynomial. Since $\rho(S)=6$ and $\mathscr{C}$ is a rational polyhedral region, the standard lattice-point estimate gives
    \[
        P_{\mathscr{C}}(n) = \alpha(-K_S,\mathscr{C})n^5 + O(n^4).
    \]
    It follows immediately that
    \[
        \sum_{n\leqslant d} P_{\mathscr{C}}(n) q^{(1-\delta\varepsilon)n} = O\left(d^5q^{(1-\delta\varepsilon)d} \right).
    \]

    It remains to evaluate the main term. Writing $r=d-n$, we have
    \begin{align*}
        \frac{1}{q^dd^5} \sum_{\substack{
            \alpha\in\mathscr{C}\cap N_1(S)_{\BZ}\\
            h(\alpha)\leqslant d
        }}
        q^{h(\alpha)} &= \sum_{r=0}^d q^{-r} \frac{P_{\mathscr{C}}(d-r)}{d^5}.
    \end{align*}
    For every fixed $r$, as $d$ approaches infinity, 
    \[
        \frac{P_{\mathscr{C}}(d-r)}{d^5}
        \longrightarrow
        \alpha(-K_S,\mathscr{C}),
    \]
    and therefore by dominated convergence, we have 
    \begin{align*}
        \lim_{d\to\infty} \frac{1}{q^dd^5} \sum_{\substack{ \alpha\in\mathscr{C}\cap N_1(S)_{\BZ}\\
            h(\alpha)\leqslant d
        }} q^{h(\alpha)}  &= \alpha(-K_S,\mathscr{C})
        \sum_{r=0}^\infty q^{-r}\\
        &= (1-q^{-1})^{-1}
        \alpha(-K_S,\mathscr{C}).
    \end{align*}

    We conclude that
    \[
        N^{\mathscr{C}}(B,S,-K_S,d) \sim (1-q^{-1})^{-1} \alpha(-K_S,\mathscr{C}) \tau_{-K_S}(S)q^dd^5.
    \]
    This is precisely the claimed asymptotic.
\end{proof}

\part*{Appendix: Peyre's constant for a split quartic del Pezzo surface}

Peyre's refinement of Manin's conjecture expresses the leading constant in terms of the geometry of the effective cone, Galois cohomology of the Picard group, and an adelic Tamagawa measure \cite{Pey95}. The corresponding formalism over global function fields appears, for example, in \cite{Pey12} and \cite{LT26}. In this appendix, we record the simplifications for the split constant surface considered in this paper.

Let $K=\BF_q(B)$. For a split quartic del Pezzo surface $S$, we have $\alpha(-K_S)=1/180$ for the full nef cone by \cite[Theorem 4]{Der07}; for the restricted region $\mathscr C$, this is replaced by $\alpha(-K_S,\mathscr{C})$. Since $S$ is split and the Galois action on $\Pic(S_{\overline K})\cong \BZ^6$ is trivial, we have $\beta(S)=\# H^1(K,\Pic(S_{\overline K})) =1$. Thus it remains to describe the Tamagawa constant. With the normalization of \cite{Pey12,LT26}, it is given by the following regularized adelic volume
\[
    \tau_{-K_S}(S) = \lim_{s\to 1} (1-q^{1-s})^6 L\bigl(s,\Pic(S_{\overline K})\bigr)\, \omega_{-K_S} \left(S(\BA_K)^{\operatorname{Br}(S_K)}\right).
\]

In the present situation, both factors admit a simple description. Since $S$ is split, Frobenius acts trivially on $\Pic(S_{\overline K})$, and hence
\[
    L\bigl(s,\Pic(S_{\overline K})\bigr) = Z_B(q^{-s})^6,
\]
where
\[
    Z_B(t)=\frac{P_B(t)}{(1-t)(1-qt)}
\]
is the Hasse--Weil zeta function of $B$. It follows immediately that the regularized value at $s=1$ is
\[
    \lim_{s\to 1}
    (1-q^{1-s})^6
    L\bigl(s,\Pic(S_{\overline K})\bigr)
    =
    (1-q^{-1})^{-6}P_B(q^{-1})^6.
\]

The adelic factor is also explicit. The constant model
$S\times B\to B$ has good reduction at every closed point $c\in |B|$.
Consequently the local volume is
\[
    \omega_c\bigl(S(K_c)\bigr)
    =
    \frac{\# S(\BF_{q^{|c|}})}{q^{2|c|}},
\]
while the local convergence factor contributed by the Picard group is $(1-q^{-|c|})^6$. Moreover, since $S_K$ is rational,
\[
    \operatorname{Br}(S_K)/\operatorname{Br}(K)=0,
\]
so the Brauer set is the entire set: $S(\BA_K)^{\operatorname{Br}(S_K)}=S(\BA_K)$. Taking into account the discriminant factor $q^{2(1-g)}$, we obtain
\[
    \tau_{-K_S}(S) = q^{2-2g}(1-q^{-1})^{-6}P_B(q^{-1})^6 \prod_{c\in |B|} \left( (1-q^{-|c|})^6 \frac{\#S(\BF_{q^{|c|}})}{q^{2|c|}}\right).
\]

Finally, the functional equation for $Z_B(t)$ gives
\[
    P_B(q^{-1})=q^{-g}P_B(1), \text{ with } P_B(1)=\#\Jac(B)(\BF_q).
\]
Thus the Tamagawa constant takes the form
\[
    \tau_{-K_S}(S) = q^{2-8g}(1-q^{-1})^{-6}(\#\Jac(B)(\BF_q))^6 \prod_{c\in |B|}
    \left(
        (1-q^{-|c|})^6
        \frac{\#S(\BF_{q^{|c|}})}{q^{2|c|}}
    \right).
\]

\section*{Disclosure on AI usage}
During the final preparation of this article, ChatGPT 5.6 Sol was used to improve the language, clarity, and exposition of arguments developed by the authors. Fable 5 suggested strengthening Theorem \ref{thm:main} by making the constant depend on the curve $C$ only through its genus $g$. In particular, Fable 5 suggested Lemmas \ref{lem:peel} and \ref{lem:noeth}. The authors independently developed and verified all AI-assisted suggestions, and take full responsibility for the content of this article.

\bibliographystyle{alphaurl}
\bibliography{ref}

\end{document}